\documentclass[a4paper,10pt]{article}
\usepackage{amssymb}
\usepackage{latexsym}
\usepackage{amsmath}
\usepackage{amsthm}
\usepackage{amsfonts}
\usepackage{amssymb}
\usepackage{amsmath,amscd}
\usepackage{graphicx}
\usepackage{color}
\usepackage{array}
\newtheorem{Th}{Theorem}

\newtheorem{Lm}{Lemma}
\newtheorem{Lma}{Lemma}[section]

\newtheorem{Rm}{Remark}

\newtheorem{OP}{Open Problem}
\newcommand{\be}{\begin{equation}}
\newcommand{\ee}{\end{equation}}
\newcommand{\bes}{\begin{equation*}}
\newcommand{\ees}{\end{equation*}}

\newcommand{\R}{\mathbb{R}}

\newcommand{\C}{\mathbb{C}}
\newcommand{\Z}{\mathbb{Z}}

\newcommand{\tr}{\mathrm{tr}}
\newcommand{\cqfd}
{%
\mbox{}%
\nolinebreak%
\hfill%
\rule{2mm}{2mm}%
\medbreak%
\par%
}

\newcommand\res{\mathop{\hbox{\vrule height 7pt width .5pt depth 0pt
\vrule height .5pt width 6pt depth 0pt}}\nolimits}

\newcommand{\reset}{\setcounter{equation}{0}\setcounter{Th}{0}\setcounter{Prop}{0}\setcounter{Co}{0}
\setcounter{Lm}{0}\setcounter{Rm}{0}}

\def\XXint#1#2#3{{\setbox0=\hbox{$#1{#2#3}{\int}$}
\vcenter{\hbox{$#2#3$}}\kern-.5\wd0}}

\def\un{\underline}

\def\ti{\tilde}
\def\lf{\left}
\def\rg{\right}

\def\al{\alpha}
\def\la{\lambda}

\def\eps{\varepsilon}
\def\ve{\varepsilon}
\def\ds{\displaystyle}
\def\ov{\overline}

\def\om{\omega}
\def\p{\partial}

\def\res{\mathop{\hbox{\vrule height 7pt width .5pt 
depth 0pt\vrule height .5pt width 6pt depth 0pt}}\nolimits}

\makeatletter
\newcommand{\tpitchfork}{%
  \vbox{
    \baselineskip\z@skip
    \lineskip-.52ex
    \lineskiplimit\maxdimen
    \m@th
    \ialign{##\crcr\hidewidth\smash{$-$}\hidewidth\crcr$\pitchfork$\crcr}
  }%
}
\makeatother

\begin{document}

\title{Gluing instantons \`a la Brezis-Coron in dimension four and the dipole construction}

\author{Luca Martinazzi\footnote{Department of Mathematics Guido Castelnuovo,
Universit\`a di Roma La Sapienza, Italy.
} \ and Tristan Rivi\`ere\footnote{Department of Mathematics, ETH Zentrum,
CH-8093 Z\"urich, Switzerland.}}

\maketitle
\begin{abstract}
Given a connection $A$ on a $SU(2)$-bundle $P$ over $\R^4$ with finite Yang-Mills energy $YM(A)$ and nonzero curvature $F_A(0)$ at the origin, and given $\rho>0$ small enough, we construct a new connection $\hat A$ on a bundle $\hat P$ of different Chern class ($|c_2(A)-c_2(\hat A)|=8\pi^2$), in such a way that $\hat A$ is gauge equivalent to $A$ in $\R^4\setminus B_\rho(0)$, gauge equivalent to an instanton in a smaller ball $B_{\tau \rho}(0)$, and
$$YM(\hat A)\le YM(A)+8\pi^2-\ve_0\rho^4|F_A(0)|^2,$$
where $\tau\in (0.3,0.4)$ and  $\ve_0>0$ are universal constant independent of $A$ and $\rho$. Our gluing method is similar in spirit to the one of Brezis-Coron for harmonic maps. We compare it with classical results by Taubes and discuss applications and open problems.
\end{abstract}

\section{Introduction}

Strict inequalities have come to play a central role in solving variational problems since the seminal work of Aubin \cite{Aub} on the solution of the Yamabe problem in dimension $n\ge 6$, followed by the work of Schoen \cite{Sch} solving the Yamabe problem in the remaining dimensions, via the celebrated positive mass theorem. The broad general principle at play is the following: in variational problems with lack of compactness, a blow up phenomenon might lead to the absence of solutions. On the other hand, usually blow up implies an (energy) inequality, so that proving the strict opposite inequality is the crucial tool in preventing blow up and recovering enough compactness to apply variational methods.

\subsection{The Brezis-Coron construction in dimension $2$, dipoles and the regularity of harmonic maps in dimension $3$}

This principle was elegantly used by Brezis and Cor\'on in the setting of $S^2$-valued harmonic maps of dimension $2$. Solving an open problem raised by Giaquinta and Hildebrant \cite{GH}, they showed that given and smooth non-constant boundary datum $g:\partial D^2\to S^2$, it is possible to extend it to two distinct harmonic maps $\underline u$ and $\overline u$. In fact $\underline u$ minimizes
$$E(u):=\int_{D^2} |\nabla u|^2 dxdy$$
in the set $\mathcal{E}_g:=\{u\in H^1(D^2,S^2): u|_{\partial D^2}=g\}$,
and its existence follows immediatly be direct methods, while $\overline u$ minimizes $E$ in $\mathcal{E}_{g,\pm1}$, where $\mathcal{E}_{g,k}$ is the set of maps in $\mathcal{E}_g$, which additionally satisfy the  \emph{topological constraint} that the Brouwer degree (suitably defined) of $\overline u$ relative to $\underline u$ is $k$.

Existence in $\mathcal{E}_{g,k}$ is a typical problem with lack of compactness, as a minimizing sequence $(u_n)\subset \mathcal{E}_{g,k}$ can concentrate and ``jump'' to a  different homotopy class, so that $u_n\rightharpoonup u_\infty\in \mathcal{E}_{g,k'}$ for some $k'\ne k$. If this happens, due to concentration, the energy drops of at least $4\pi$, in fact
\be\label{jump}
E(u_\infty)\le \liminf_{n\to\infty} E(u_n)-8\pi|k-k'|.
\ee
Then, if one proves that for $k=1$ or $k=-1$
\be\label{enest}
\inf_{\mathcal{E}_{g,k}} E <\inf_{\mathcal{E}_{g}}E+8\pi,
\ee
concentration becomes impossible, since otherwise \eqref{jump} and \eqref{enest} combined would imply $E(u_\infty)<\inf_{\mathcal{E}_{g}} E$.

\begin{figure}
\begin{center}
\includegraphics[width=7cm]{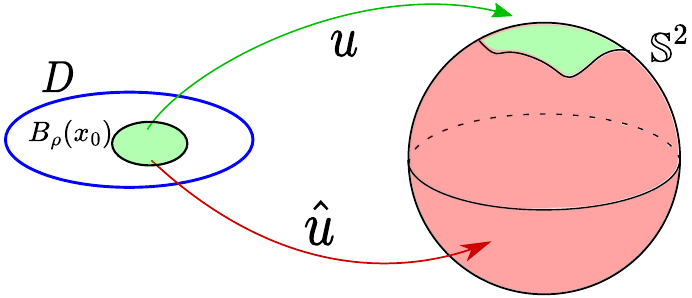}
\caption{Brezis-Coron method of replacing $u$ by a new map of different Brouwer degree.}
\label{f:BC}
\end{center}
\end{figure}

Hence, the strict inequality \eqref{enest} is the crucial ingredient in the existence of a second harmonic map with boundary value $g$. In order to obtain it, Brezis and Coron \cite{BC} showed that given a smooth map $u:\Omega\subset\R^2 \to S^2$ and $x_0\in \Omega$ such that $\nabla u(x_0)\ne 0$, it is possible to change $u$ in a ball $B_\rho(x_0)$  to a new smooth map $\hat u$ (see Figure \ref{f:BC}) such that the Brouwer degree of $\hat u$ relative to $u$ is $\pm 1$\footnote{depending on whether $\nabla u(x_0)$ preserves or reverses orientation, a condition that in complex coordinates can be expressed as
\be\label{condBC}
|\partial_{\bar z} u|\le |\partial_{z}u|,\quad \text{ or }\quad |\partial_z u|\le |\partial_{\bar z}u|,
\ee
the counterparts of our conditions \eqref{0.1} and \eqref{0.1bis} below.} and
\begin{equation}\label{estimateBC}
\int_\Omega |\nabla \hat u|^2 \ dx^2 \le \int_\Omega |\nabla  u|^2 \ dx^2 + 8\pi - \ve_0 |\nabla u(x_0)|^2\rho^2\ ,
\end{equation}
for $\rho>0$ sufficiently small. To fix the ideas, if $x_0=0$, $u(0)=(0,0,1)\in S^2$ and $\nabla u(0)\ne 0$, Brezis and Coron construct $\hat u$ such that $\hat u=u$ in $\Omega\setminus B_\rho(0)$, $\hat u$ is, \emph{a rotation and a rescaling} $v_\lambda$ of the inverse stereographic projection in $B_{\rho/2}(0)$, and $\hat u$ is an interpolation between $u$ and $v_\lambda$ in $B_\rho(0)\setminus B_{\rho/2}(0)$. The subtle part in \cite{BC} is proving that the interpolation energy $E(\hat u,B_\rho(0)\setminus B_{\rho/2}(0))$, which we will call ``cost of gluing'' is smaller than the ``energy saving'' $8\pi - E(v_\lambda, B_{\rho/2}(0))+ E(u, B_{\rho}(0))$, although \emph{they are both of order} $\rho^2$. The difference between energy saving and cost of gluing will be called ``energy gain''.

Also Jost \cite{jos}, independently of Brezis-Coron constructed a second harmonic map proving an estimate similar to \eqref{estimateBC}, with a slightly different method. We further mention that the Brezis-Coron construction has been recently adapted by the first author and Hyder \cite{HM} to the case of half-harmonic maps from the line into the circle, proving an analog to \eqref{estimateBC} and the existence of a second half-harmonic map for any non-constant Dirichlet boundary condition.

\medskip

Notice that, in order to solve the problem raised by Giaquinta and Hildebrandt, the non -quantitative estimate
\begin{equation}\label{estimateBC2}
\int_\Omega |\nabla \hat u|^2 \ dx^2 < \int_\Omega |\nabla  u|^2 \ dx^2 + 8\pi
\end{equation}
would have been sufficient. On the other hand the quantitative form \eqref{estimateBC}, in which the energy gain is of the order $\rho^2$, turned out to be very relevant in the study of the relaxed Dirichlet energy for 3-dimensional maps into $S^2$ and its applications to regularity, particularly using $3$-dimensional ``strict'' dipoles, as we shall recall.

The name dipole, as introduced by Brezis-Coron-Lieb \cite{BCL}, refers to a map $w: \R^3\to S^2$, constant outside a compact set, smooth away from two points $\{P,N\}\in \R^3$, where there are singularities of Brouwer degree $\pm 1$ ($\deg(w|_{\partial B_\delta(P)})=1$, $\deg(w|_{\partial B_\delta(N)})=-1$ for small $\delta>0$).
As proven in \cite[Thm. 3.1]{BCL}, for every $\eps>0$ one can construct such a dipole satisfying
\be\label{dipoleBCL}
8\pi< E_{3d}(w)\le 8\pi h +\eps\ ,
\ee
where $h=|P-N|$ and $E_{3d}$ is the Dirichlet energy in dimension $3$.
Estimate \eqref{dipoleBCL} suggests that the cost of creating two topological singularities $\{P,N\}$ is bigger than $8\pi |P-N|$.

\begin{figure}
\begin{center}
\includegraphics[width=7cm]{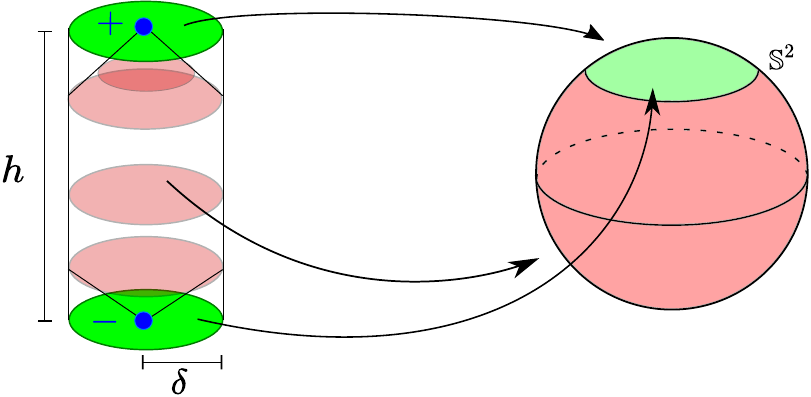}
\caption{The \emph{strict} dipole construction}
\label{f:dipole}
\end{center}
\end{figure}

On the contrary, as shown by Hardt-Lin-Poon \cite{HLP} in the axially symmetric case and by the second author \cite{Riv} in the general case, given a map $u:\Omega\subset\R^3\to S^2$, smooth near $\xi\in \Omega$ with $\nabla u(\xi)\ne 0$, and an arbitrarily small neighborhood $\Omega'\Subset\Omega$ of $\xi$, one can ``insert'' a dipole with singularities $\{P,N\}\subset\Omega'$ of Brouwer degree $\pm 1$ as above, obtaining a new map $v$ which equals $u$ in $\Omega\setminus \Omega'$, $v|_{\Omega'}\in C^\infty(\overline{\Omega'}\setminus \{P,N\})$, and
satisfies the strict inequality
\be\label{strict}
E_{3d}(v)<E_{3d}(u)+8\pi h\ ,
\ee
where $h=|P-N|$. The crucial point in \eqref{strict} is that the cost of inserting the two singularities $\{P,N\}$ is strictly smaller than $8\pi |P-N|$, contrary to \eqref{dipoleBCL}.\footnote{The map $w$ in \eqref{dipoleBCL} can also be seen as inserting a dipole to a constant map $u$, hence the condition $\nabla u(\xi)\ne 0$ is never satisfied.}

One constructs of $v$ (see Figure \ref{f:dipole}) roughly as follows: up to a rotation and translation take $\xi=0$ and assume $\nabla_{\R^2}u(0)\ne 0$. Set $v=u$, outside the cylinder $B_{\rho}(0)\times (-h/2-h^2,h/2+h^2)$; for $z\in (-h/2,h/2)$ away from $\pm h/2$, replace $u|_{B_{\rho}(0)\times\{z\}}$ with a map of different degree, as constructed by Brezis-Coron, at  a cost
\be\label{strict1}
E(v|_{B_\rho(0)\times\{z\}})\le E(u|_{B_\rho(0)\times\{z\}}) +8\pi-\eps_0|\nabla_{\R^2} u(0,z)|^2\rho^2\ ,
\ee
according to \eqref{estimateBC}. As $z\to \pm h/2$ one shrinks $v(\cdot,z)$, creating the singularities at $(0,\pm h/2)$, still respecting \eqref{strict1}. Integrating \eqref{strict1} for $z\in (-h/2,h/2)$ seems to lead to \eqref{strict}, with an energy gain of order $O(|\nabla_{\R^2} u(0)|^2\rho^2 h)$  but on the other hand, by shrinking $v(\cdot,z)$ as $z\to \pm h/2$ and by gluing $v$ to $u|_{B_\rho(0)\times \{\pm (h/2+h^2)\}}$, one introduces a significant amount of $z$-derivative, adding an energy of at least $O(\rho^4 /h)$. The estimate of this contribution is the main difficulty in the symmetric strict dipole construction of \cite{HLP} and particularly in the non-symmetric case of \cite{Riv}. The additional ``ingredient'' in the dipole construction of the second author in \cite{Riv} was to ``follow'' the evolution of the initial map $u$ along the $z$ direction and to insert  coverings of $S^2$ at each $z-$level centred at $u(0,0,z)$  instead of centring  them at $u(0,0,0)$. In this way the $z$ derivative of $v(0,0,z)$ is equal to the $z$ derivative of $u(0,0,z)$ and no ``useless'' energy is added along the tube in such a way that the cost of closing the dipole at the two singularities is eventually dominated by the gain of energy all along the connecting tube which is chosen to be of length $h=\sqrt{\rho}$.


\medskip

Coming back to the axially symmetric framework, in order to obtain \eqref{strict}, the quantitative term $-\eps_0|\nabla u (x_0)|\rho^2$ in \eqref{estimateBC} (hence in \eqref{strict1}) cannot be made arbitrarily smaller (for instance of order $\rho^4$), since otherwise the $2$-dimensional energy gain  would be outspent by the cost of ``closing the dipole'', i.e. the vertical derivative coming from the rescaling as $z\to \pm h/2$. This was made precise by the first author \cite{Mar-relaxed}, working on $n$-axially symmetric maps for $n>1$. In this case, $\nabla u=0$ along the vertical axis, and the energy gain in \eqref{estimateBC} on disks transversal to the axis is only at most $O(\rho^{4})$, smaller than the cost of closing the dipole.

Now, considering that the strict dipole construction of Hardt-Lin-Poon was crucial in their proof that axially symmetric minimizers of the relaxed Dirichlet energy (as introduced by Bethuel-Brezis-Coron \cite{BBC} and Giaquinta-Modica-Soucek \cite{GMS}) have only isolated ``degree zero'' singularities, the absence of a strict dipole for $n$-axially symmetric maps makes the corresponding regularity theory still open:

\begin{OP}[see \cite{Mar-relaxed})]\label{OP1}
Are $n$-axially symmetric minimizers of the relaxed Dirichlet energy smooth away from finitely many points when $n>1$, similar to the case $n=1$ treated in \cite{HLP}?
\end{OP}

Due to the same difficulty of constructing dipoles satisfying \eqref{strict} at points where $\nabla u=0$, it is still unknown whether tangent maps to minimizers of the relaxed Dirichlet energy from $\R^3$ to $S^2$ are necessarily constant (see \cite{Mar-relaxed} for more details):

\begin{OP}\label{OP2} Are degree-zero homogenous maps $u:\R^3\to S^2$ locally minimizing the relaxed Dirichlet energy necessarily constant?
\end{OP}



Moreover, the above-mentioned (non-symmetric) dipole construction satisfying \eqref{strict}, was the fundamental ingredient in the second author's construction of weakly harmonic maps from $B^3$ into $S^2$ which are everywhere discontinuous \cite{Riv}. The failure of the dipole construction for $n$-axially symmetric maps when $n\ge2$, raises the following question:\

\begin{OP}\label{OP3} Are there $n$-axially symmetric weakly harmonic maps in $H^1(B^3,S^2)$ which are everywhere singular on the $z$-axis $\{(0,0,z): z\in (-1,1)\}$ when $n\ge 2$?
\end{OP}

The answer is affirmative in the case $n=1$, as shown by the second author \cite{Riv0}.

\subsection{Instanton's insertion for the Yang-Mills energy in dimension $4$}

The main purpose of this work is to show that a construction very similar in spirit to the one of Brezis-Coron is also possible for the Yang-Mills energy in dimension $4$, and to discuss the possibility of a $5$-dimensional dipole construction for the Yang-Mills energy, with an eye towards the regularity theory for weak Yang-Mills connections in dimension $5$.

Consider a connection $A\in \Omega^1(\R^4)\otimes su(2):=C^\infty(\R^4,\Lambda^1\R^4\otimes su(2))$ on a $SU(2)$-bundle $P$. Up to modifying $A$ outside a compact set, we can assume that $A$ is the pull-back via stereographic projection of a smooth connection on a $SU(2)$-bundle over $S^4$. In particular, we can define its Chern class
\be
\label{0.0}
c_2(A):=\int_{{\R}^4}\tr (F_A\wedge F_A)\in 8\pi^2{\Z}\,,
\ee 
and its Yang-Mills energy
\be
\label{0.0b}
YM(A):=\int_{{\R}^4}|F_A|^2 dx^4 <\infty\,.
\ee 
The precise definition $\mbox{tr}(F_A\wedge F_A)$ and $|F_A|$ will be given in the next section.

The Yang-Mills energy in dimension $4$ enjoys several properties of the Dirichlet energy in dimension $2$, starting with scale invariance. Moreover, as we shall discuss in the following section, any $2$-form (for instance $F_A$) in dimension $4$ can be split into a self-dual and an anti-self dual part, similar to the splitting of $\nabla u$ in dimension 2 (e.g. $u:\R^2\to S^2$) into a conformal and anticonformal part. The analog of conformal (or anti-conformal) maps such as M\"obius transformations, for Yang-Mills is given by the self-dual instantons and the anti-self-dual instantons, also discussed in the next section. Finally, the role played by the Brouwer degree for $2$-dimensional maps into $S^2$, will be taken over by the 2nd Chern-class \eqref{0.0}.

With these ingredients, we can state the main result of this paper, which can be seen as the Yang-Mills counterpart of the Brezis-Coron construction:

\begin{Th}
\label{th-I.1} Let $A\in \Omega^1(\R^4)\otimes su(2)$ be a smooth connection form on ${\R}^4$ with finite Yang-Mills energy such that $F_A(0)\ne 0$.
Assume that
\be
\label{0.1}
|P_+F_A(0)| \le  |P_-F_A(0)|\, ,
\ee
where $P_+$ and $P_-$ are respectively the projections of $2$-forms onto self-dual or anti-self dual 2-forms in ${\R}^4$. Then, there exists $\tau\in (0.3,0.4)$ such that for $\rho>0$ small enough, one can modify $A$ into $\hat{A}$ such that
$A$ and $\hat{A}$ are gauge equivalent to each other on ${\R}^4\setminus B_\rho(0)$, $\hat{A}$ is gauge equivalent to a self-dual instanton in $B_{\tau\rho}(0)$,
\be
\label{0.2}
\int_{{\R}^4}\tr(F_{\hat{A}}\wedge F_{\hat{A}})=\int_{{\R}^4}\tr(F_A\wedge F_A)- 8\pi^2\ ,
\ee
and
\be
\label{0.3}
\int_{{\R}^4}|F_{\hat{A}}|^2\ dx^4\le \int_{{\R}^4}|F_{{A}}|^2\ dx^4+8\pi^2 -\eps_0 |F_A(0)|^2\, \rho^4\ ,
\ee
for a positive constant $\eps_0$ independent of $A$ and $\rho$.
\hfill $\Box$
\end{Th}
\begin{Rm}
The corresponding result holds in case
\be
\label{0.1bis}
|P_+F_A(0)|\ge |P_-F_A(0)|\ ,
\ee
with (\ref{0.2}) replaced by
\be
\label{0.2bis}
\int_{{\R}^4}\tr(F_{\hat{A}}\wedge F_{\hat{A}})=\int_{{\R}^4}\tr (F_A\wedge F_A)+8\pi^2\ ,
\ee
and with $\hat{A}$ gauge equivalent to an anti-selfdual instanton in $B_{\tau\rho}(0)$.
\hfill $\Box$
\end{Rm}

\begin{figure}
\begin{center}
\includegraphics[width=5cm]{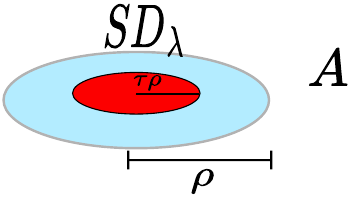}
\caption{In Theorem \ref{th-I.1} the instanton is glued to $A$ in the region $B_{\rho}(0)\setminus B_{\tau \rho}(0)$}
\label{fig1}
\end{center}
\end{figure}

The geometric meaning of \eqref{0.2} and \eqref{0.2bis} is that while $A$ is a connection on a given $SU(2)$-bundle $P$, $\hat A$ is a connection on a different $SU(2)$-bundle $\hat P$ of Chern class $c_2(\hat P)=c_2(P)\pm 8\pi^2$, obtained by a gluing procedure, already described for instance in \cite{FU,Law, Tau}.
Roughly speaking, $\hat P$ is obtained gluing the trivial $SU(2)$-bundle over $B_\rho(0)$ to the trivial $SU(2)$-bundle over $\R^4\setminus \{0\}$ via the non-trivial gauge changes $g(x)= \frac{x}{|x|}$, or $g(x)= \frac{\ov{x}}{|x|}$ , where we identify $\R^4$ with the space of quaternions $\mathbb{H}$, see Section \ref{sec:II.1}.

Since such change of bundle/Chern class can be obtained by gluing a rescaled (concentrated) self-instanton $SD_\lambda$ as given in \eqref{II.6}, and since $YM(SD_\la)=8\pi^2$, it is not difficult to construct $\hat A$ such that $YM(\hat A)\le YM(A)+8\pi^2+\ve$ with $\ve>0$ arbitrarily small. In fact, scaling arguments lead to the more quantitative 
\be
\label{0.3bis}
YM(\hat A)= YM(A)+8\pi^2+O(\rho^4),\quad \text{as }\rho\to 0\ ,
\ee
where $\rho$ is the scale at which the gluing occurs.
This can be made to hold true irrespective of whether condition \eqref{0.1} is satisfied or not. On the other hand, for applications it is crucial in Theorem \ref{th-I.1} that the term $O(\rho^4)$ in \eqref{0.3bis} is negative and this is precisely the content of \eqref{0.3}.

\medskip

Regarding the main difficulty in proving \eqref{0.3}, there is a strong analogy with \eqref{estimateBC}, in that the energy saving due to the removal of the original connection in $B_\rho$ and to the energy of the glued instanton (which is strictly less than $8\pi^2$), is of the same order $\rho^4$ of the cost of gluing the instanton to the original connection. As in \cite{BC}, there seems to be no deep reason why the energy saving should be bigger than the the cost of gluing, giving a positive energy gain. In fact, by following in the steps of \cite{BC}, i.e. replacing $A$ in $B_{\rho/2}(0)$ with a  self-instanton $SD_\lambda$, obtained by scaling by a suitable factor $\lambda>0$, and gluing in the annulus $B_\rho(0)\setminus B_{\rho/2}(0)$  we would obtain that in the most degenerate (least symmetric) cases, the cost of gluing is bigger than the energy saving \emph{for every choice of} $\lambda>0$. It is only by inserting $SD_\lambda$ in $B_{\tau\rho}(0)$ for some $\tau\in (0,1)$, then gluing in $B_\rho(0)\setminus B_{\tau\rho}(0)$ and fine-tuning both $\tau$ and $\lambda$ that we can obtain a net energy gain (and actually quite small: the energy saving is only around $1\%$ bigger than the gluing cost, in the worst possible case).

\medskip

A different approach to the instanton insertion was taken by Taubes \cite{Tau} as part of his important work on the Atiyah-Jones conjecture regarding the homotopy type of the space of instantons versus general connections modulo gauge action (see Segal \cite{Seg} for a counterpart regarding the space of holomorphic versus general maps from $\mathbb{C}$ into $\mathbb{CP}^1$). Taubes' method in some sense weakens the estimate but simplifies the proof and has the advantage of holding true even without condition \eqref{0.1}. The main difference between Taubes' approach and ours is the thickness of the gluing region. Following Brezis-Coron, we want the gluing region to be an annulus $B_{\rho}(0)\setminus B_{\tau\rho}(0)$.  
Correspondingly, we choose $\hat A$ of the form (see Figure \ref{fig1})
$$\hat A= \eta_\rho \widetilde{\underline{A}}^{g_0} + (1-\eta_\rho) SD_\lambda,$$
where $\widetilde{\underline{A}}^{g_0}$ will be obtained from $A$ restricted to $\partial B_\rho(0)$, via a projection, a  gauge change and $SD_\lambda$ is a suitably scaled instanton ($\lambda$ depending on $\rho$). Most importantly, $\eta_\rho$ is a cut-off function vanishing in $B_{\tau\rho}(0)$ and equal to $1$ in $\R^4\setminus B_{\tau\rho}(0)$.

Taubes, instead, performed the gluing in the \emph{thicker} annulus $B_{\rho^{1/2}}(0)\setminus B_{\rho^{3/2}}(0)$, using two cut-off functions, so that
$$\hat A_{Taubes}= \varphi_\rho \tilde A^{g_0} + \psi_\rho SD_\lambda,$$
with $\varphi_\rho=0$ in $B_{\rho^{3/2}}(0)$,  $\varphi_\rho=1$ in $\R^4\setminus B_{2\rho^{3/2}}(0)$, $\psi_\rho=1$ in $B_{\sqrt{\rho}}(0)$ and $\psi_{\rho}(0)=0$ in $\R^4\setminus B_{2\sqrt{\rho}}(0)$, see Figure \ref{fig2}.

\begin{figure}
\begin{center}
\includegraphics[width=5cm]{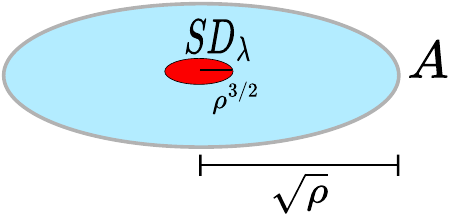}
\caption{In Taubes' original approach \cite{Tau} the instanton the gluing region is $B_{\rho^{1/2}}(0)\setminus B_{\rho^{3/2}}(0)$.}
\label{fig2}
\end{center}
\end{figure}

A mild modification of Taubes' proof gives the following result, of which we shall also give the proof for the sake of completeness:

\begin{Th}
\label{th-I.2} Let $A\in \Omega^1(\R^4)\otimes su(2)$ be a smooth connection form on ${\R}^4$ with finite Yang-Mills energy such that
\be\label{0.1ter}
P_{-}F_A(0)\ne 0\ .
\ee
Then, for $0<b<1<a$ and for $\rho>0$ small enough, one can modify $A$ into $\hat{A}$ such that
$A$ and $\hat{A}$ are gauge equivalent to each other on ${\R}^4\setminus B_{\rho^b}(0)$, $\hat{A}$ is gauge equivalent to a self-dual instanton in $B_{\rho^a}(0)$,
\be
\label{0.2ter}
\int_{{\R}^4}\tr(F_{\hat{A}}\wedge F_{\hat{A}})=\int_{{\R}^4}\tr(F_A\wedge F_A)- 8\pi^2\ ,
\ee
and
\be
\label{0.3ter}
\int_{{\R}^4}|F_{\hat{A}}|^2\ dx^4\le \int_{{\R}^4}|F_{{A}}|^2\ dx^4+8\pi^2 -\frac{4\pi^2}{\sqrt{3}} |P_-F_A(0)|^2 \rho^4+O(\rho^{4+\delta})\ ,
\ee
as $\rho\to 0$, for some $\delta>0$ depending only on $a$ and $b$.
\hfill $\Box$
\end{Th}
\begin{Rm} By replacing condition \ref{0.1ter} with
\be\label{0.1quater}
P_{+}F_A(0)\ne 0\ ,
\ee
one can replace \eqref{0.2ter} with \eqref{0.2bis} by gluing an anti-self-dual instanton.
\end{Rm}

Cutting-off the instanton and the connection $A$ at different scales, while thickening the interpolation region, has the advantage of making several ``error'' terms in the computation of the gluing cost of smaller order compared with the main energy gain. As shown in \cite{Tau2} (and as easily inferable from the proof of Theorem \ref{th-I.2}), one can actually take $b=1$ in Theorem \ref{th-I.2}, at the cost of making the constant $\frac{4\pi^2}{\sqrt3}$ smaller, but still independent of $A$. Instead, if one wants to have a gluing region of the form $B_{\rho}(0)\setminus B_{\tau\rho}(0)$, with $\tau\in (0,1)$ independent of $A$, with our current ``technology'' one needs a condition as in \eqref{0.1} of Theorem \ref{th-I.1}.

On the other hand, in dimension $2$ Kuwert \cite{kuw} was able to refine the Brezis-Coron approach and construct a map $\hat u:\Omega\to S^2$ by inserting an anti-conformal (resp. conformal) map $v$ in $B_{\rho(x_0)}$ and gluing it to the original map $u$ in the very thin region $B_{\rho+\rho^{3/2}}(x_0)\setminus B_{\rho(x_0)}$, in such a way that $\deg(\hat u)= \deg(u)-1$ (resp. $\deg(\hat u)= \deg(u)+1$) with an estimate of the form
$$E(\hat u)\le E(u)- 2\pi |\partial_z u(x_0)|^2 \rho^2 +O(\rho^{3/2}),\quad \text{as }\rho\to 0$$
(resp. $E(\hat u)\le E(u)- 2\pi |\partial_{\bar z} u(x_0)|^2 \rho^2 +O(\rho^{3/2})$) under the only condition $\partial_z u(x_0)\ne 0$ (resp. $\partial_{\bar z} u(x_0)\ne 0$). This allowed Kuwert to prove that for generic boundary data $g:\partial \Omega\subset\R^2\to S^2$, there exists minimizers of the Dirichlet energy in every homotopy class. Kuwert's refined construction is based on the existence of conformal maps which are very close to $u|_{\partial B_\rho(x_0)}$, as already used by Jost \cite{jos}. At this stage, we do not know whether we have sufficiently many self-instantons (anti-self-instantons) to reproduce Kuwert's argument.

\begin{OP}\label{OP4} Is it possible, given a smooth connection $A\in \Omega^1(\R^4)\otimes su(2)$ satisfying \eqref{0.1ter}, to construct $\hat A$ satisfying \eqref{0.2ter} in such a way that $\hat A$ is self-dual (i.e. $P_-F_{\hat A}=0$) in $B_{\rho}(x_0)$, $\hat A= A$ up to gauge change in $B_{\rho+\rho^a}(x_0)$ for some $a>1$, and \eqref{0.3ter} holds as $\rho \to 0$? 
\end{OP}

\subsection{Dipoles in dimension $5$ and regularity questions for the Yang-Mills energy}

Taubes' estimate (essentially Theorem \ref{th-I.2} in our setting) has several applications. We mention in particular the existence of minimizers of the Yang-Mills energy in different Chern classes, as shown by Isobe and Marini \cite{IsoMar}, which is the counterpart of Kuwert's \cite{kuw} result for harmonic maps. Indeed, for this result, the thickness of the gluing region is not relevant.

On the other hand, for us Theorem \ref{th-I.1} was initially motivated not just by curiosity (``Will the  Brezis-Coron construction - which has the advantage to be very natural and handful -  work in the Yang-Mills setting?'') but also by a regularity project for the Yang-Mills energy analogous to the one developed by the second author for harmonic maps, namely the following question:

\begin{OP}\label{OP5} Is there a weak Yang-Mills connection (a critical point of $YM$ defined is a suitable space of $su(2)$-valued $1$-forms) in dimension $5$ or higher which is everywhere discontinuous?
\end{OP}

The notion of weak connection was originally introduced in \cite{PeRi} and precised in \cite{CaRi}. It can be defined as follows : A $su(2)$ weak connection in the 5 dimensional euclidian space ${\mathbb R}^5$ is an $su(2)-$valued one form $A$ such that
\begin{itemize}
\item[i)] 
\[
A\in L^2(\R^5, \Lambda^1({\mathbb R}^5)\otimes su(2)) 
\]
\item[ii)]
\[
F_A:=dA+A\wedge A\in L^2(\R^5, \Lambda^2({\mathbb R}^5)\otimes su(2))
\]
\item[iii)]
For all open set $U\subset {\mathbb R}^5$, any bi-lipschitz homeomorphism $\Psi$ from $U$ into $\Psi(U)\subset {\mathbb R}^5$  then
\[
\begin{array}{l}
\forall\,f\in C^1_0(\Psi(U),{\mathbb R}^+)\quad,\quad\mbox{ for a. e. }t>0 \qquad\mbox{the restriction of  } A\ \mbox{ to }{\Psi^{-1}(f^{-1}(t))}\\[5mm]
\quad\mbox{ is locally gauge equivalent to a $su(2)$ one form in $L^4$.}
\end{array}
\]
\end{itemize}
One of the main achievements of \cite{PeRi} and \cite{CaRi} was to prove that this class is weakly sequentially closed. Moreover the Yang-Mills equation is well defined on this class and it allows for the existence of topological singularities satisfying
\be
\label{topsingYM}
d\lf(\tr (F_A\wedge F_A)\rg)\ne 0\qquad\mbox{in }{\mathcal D}'({\R}^5)\ .
\ee
This last condition is excluded for smooth connection forms $A\in C^\infty(\R^5, \Lambda^1({\mathbb R}^5)\otimes su(2))$ which must satisfy
\be
\label{topsingYMbis}
d\lf(\tr (F_A\wedge F_A)\rg)=0 \qquad\mbox{in }{\mathcal D}'({\R}^5)\ .
\ee
This follows from formula \eqref{transg} and Stokes' theorem:
$$\int_{B_r(x_0)} d\lf(\tr (F_A\wedge F_A)\rg)=\int_{\partial B_r(x_0)} \tr (F_A\wedge F_A) = \int_{\partial B_r(x_0)} d\lf( \tr \lf(A\wedge dA+\frac23 A\wedge [A,A]\rg) \rg)= 0,$$
for every ball $B_r(x_0)\subset\R^5$.
By approximation, \eqref{topsingYMbis} holds also as long as $|\nabla A| |A|^2+|A|^3\in L^1_{loc}(\R^5)$. In particular, this is the case for $A\in W^{1,p}_{loc}(\R^5, \Lambda^1({\mathbb R}^5)\otimes su(2))$ with $p\ge \frac{15}{7}>2$. The case $p=2$ is more subtle: first of all notice that for $A\in W^{1,2}_{loc}(\R^5, \Lambda^1({\mathbb R}^5)\otimes su(2))$, $A\wedge A\not\in L^2_{loc}(\R^5, \Lambda^2({\mathbb R}^5)\otimes su(2))$ in general, so that \eqref{topsingYMbis} is not well defined. If we assume both $A\in W^{1,2}_{loc}$ and $F_A\in L^2_{loc}$, then \eqref{topsingYMbis} is well defined and it does indeed hold true, but its proof is less direct and will be discussed in an upcoming work.

Observe that the counterpart to (\ref{topsingYM}) in the previously commented framework of harmonic maps from ${B}^3$ into $S^2$ is the condition
\be
\label{topsinghar}
d\lf(u^\ast\om\rg)\ne 0\ ,
\ee
where $\om$ is an arbitrary 2-form on $S^2$ such that $\int_{S^2}\om\ne 0$. The everywhere discontinuous harmonic maps constructed by the second author in \cite{Riv} satisfy supp$(d\lf(u^\ast\om\rg))=\overline{B^3}$.
Coming back to the possibility to construct singular Yang-Mills weak connections, following \cite{Riv}, then  a possible methodology would go through the building block construction dipoles for the Yang-Mills energy in 5 dimension, replacing a smooth connection $A\in \Omega^1(\R^5)\otimes su(2)$ inside a small cylinder $B^4_{\rho}(0)\times (-h/2,h/2)\subset\R^4\times \R$ (up to rotation and translation) in such a way that the new connection $B$ satisfies the analog of \eqref{strict}, namely
\be\label{strictbis}
YM_{5d}(B)<YM_{5d}(A)+8\pi^2 h\ ,
\ee
and $B$ has singularities at $(x_0,\pm h/2)$ of second Chern class $\pm 8\pi^2$ i.e. $B$ satisfy
\[
d\lf(\tr (F_B\wedge F_B)\rg)=8\pi^2\ \lf[\delta_{(x_0, h/2)}-\delta_{(x_0, -h/2)}\rg]\ .
\]
As shown in Figure \ref{f:YMdipole}, and in analogy with the harmonic map case seen above, one would replace $A$ in $B^4_{\rho}(0)\times \{0\}$ by $\hat A$ as in Theorem \ref{th-I.2} and then rescale on $B^4_\rho(0)\times \{t\}$ as $t\to \pm h/2$. This motivated the need to have well separated regions where $\hat A$ equals an instanton and where it equals $A$ (up to gauge transformation) and a ``thin'' gluing region, as in Theorem \ref{th-I.1}.

As we investigated this matter, we found something that we did not expect: no matter how nicely one glues the instanton to a connection $A$ in $4$d and how carefully on closes a corresponding a dipole of length $h\ll 1$ in $5$d, its cost is in general bigger than $8\pi^2 h$.
This is due to the energy necessary to interpolate $\hat A$ on $B^4_{\rho}(0)\times\{0\}$ with $A$ on $B_{\rho}^4\times \{\pm h/2\}$, which is of order $\rho^{4}h^{-1}$, in general bigger than the energy gain of order $\rho^4 h$ given by the horizontal replacement of $A$ with $\hat A$.

If we choose $h=\rho^b$ for some $b>0$, we can compare this phenomenon with the case of $n$-axially symmetric harmonic maps, as summerized in the following table.

\begin{figure}
\begin{center}
\includegraphics[width=3cm]{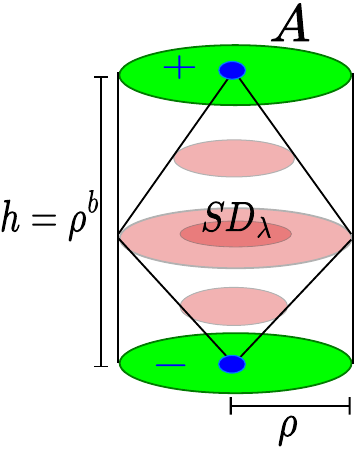}
\caption{The $5$-dimensional dipole construction for the Yang-Mills energy}
\label{f:YMdipole}
\end{center}
\end{figure}

 \begin{small}
\begin{center}
\begin{tabular}{|m{2.5cm}|m{3cm}| m{3.8cm} |m{3cm} |}
\hline
&{Harmonic maps  standard $3d$ dipole} & Harmonic maps $n$-axially symmetric $3d$ dipole & Yang-Mills $5d$ dipole \\ 
\hline
Horizontal Gain&$\ve_0 \rho^{2+b}$&$\ve_0\rho^{2n+b}$& $\ve_0\rho^{4+b}$\\
\hline
Cost of vertical interpolation&$\sim \rho^{4-b}$&$\sim \rho^{4-b}$&$\sim\rho^{4-b}$\\
\hline
\end{tabular}
\end{center}
\end{small}

In particular the cost of closing the dipole is of order $\rho^4 h^{-1}= \rho^{4-b}$ both for Yang-Mills dipoles and for harmonic maps. But for $1$-axially symmetric harmonic maps the energy gain is of order $\rho^{2+b}$, hence greater than the cost of closing the dipole, as long as we choose $b<1$. Instead, for $2$-axially symmetric harmonic maps, the energy balance is exactly the same as for Yang-Mills. Then, in analogy with Open Problem \ref{OP1}, we state:

\begin{OP}\label{OP6}
How regular are $SO(4)$-equivariant Yang-Mills connections in $\R^5=\R^4\times \R$? Can they have a line of singularities at $\{0\}\times\R$?
\end{OP}

The analogy between $2$-axially symmetric harmonic maps from $B^3$ into $S^2$ and $SO(4)$-equivariant $SU(2)$-connections was already noticed by Grotowski and Shatah \cite{GS} in the setting of flows. In fact, Schlatter-Struwe-Tahvildar-Zadeh, \cite{SST}, proved that, contrary to the $2$-dimensional harmonic map flow, the $4$ dimensions $SO(4)$-equivariant Yang-Mills flow does not blow up in finite time and conjectured that this is the case even for non-equivariant flow. Grotowski-Shatah, then noticed that also the $2$-axially symmetric harmonic map flow from $\R^2$ into $S^2$ does not blow-up in finite time, contrary to the $1$-axially symmetric example given by Chang-Ding-Ye \cite{CDY}.

Notice that there is a strong link between closing a (small) dipole and developing a singularity in finite time: in some sense, the axis of the dipole can be seen as the time axis, and developing a singularity in finite time would be an efficient way of closing the dipole, creating a singularity. In fact, in \cite{SST} the flow can blow up, but only in infinite time, in the same way as we are allowed to construct Yang-Mills dipoles satisfying \eqref{strictbis} only if we allow $h$ to be sufficiently large.

Let us also mention that Schlatter-Struwe-Tahvildar-Zadeh's conjecture \cite{SST} was proved by Waldron \cite{Wal}, who showed that in dimension $4$, the Yang-Mills flow (for general compact Lie groups) does not blow-up in finite time. This casts further doubts on the possibility of construcing very irregular Yang-Mills connections in dimension $5$ using the relaxed energy, following the ideas of \cite{Riv}, see Open Problem \ref{OP5}.

\medskip

\paragraph{Acknowledgements} Both authors had the chance to interact at different periods with Ha\"im Brezis in the early part of their career respectively in Paris and at Rutgers University, and have been greatly influenced by his elegant and deep approach to mathematics and by his almost perfect and very inspiring presentations. 

\medskip

The authors are also very grateful to Henri Berestycki and Jean-Michel Coron for inviting them to contribute to this volume in honour of Ha\"im .

\medskip

The first author is partially supported by Fondazione CARIPLO and Fondazione CDP grant n. 2022-2118, and by the PRIN Project 2022PJ9EFL. The second author would like to thank the mathematics departement of ``Sapienza, Universit\`a di Roma'' for its invitation and its great hospitality during the conception of the present work.

\section{Preliminaries and notations}

\subsection{Connection forms, curvatures and self-duality}\label{sec:II.1}
Let $A\in \Omega^1(\R^4)\otimes su(2)$ be a smooth 1-form on ${\R}^4\simeq {\C}^2$ with values into $su(2)$, the Lie algebra of the special unitary group $SU(2)$. We identify $su(2)$ with the vector space $\Im m({\mathbb H})$ of imaginary quaternions ${\mathbf i}, {\mathbf j}, {\mathbf k}$ as
\be\label{quat}
\mathbf{i} =\begin{pmatrix}
i&0\\
0&-i
\end{pmatrix},\quad
\mathbf{j} =\begin{pmatrix}
0&1\\
-1&0
\end{pmatrix},\quad
\mathbf{k} =\begin{pmatrix}
0&i\\
i&0
\end{pmatrix}\ .
\ee
Observe that ${\mathbf i}\,{\mathbf j}={\mathbf k}$, and in particular
\be\label{braquat}
[{\mathbf i},{\mathbf j}]=2\,{\mathbf k},\quad[{\mathbf j},{\mathbf k}]=2\,{\mathbf i}\ ,\quad[{\mathbf k},{\mathbf i}]=2\,{\mathbf j} \ .
\ee
Adding the identity matrix $\mathbf{1}=\begin{pmatrix}
1&0\\
0&1
\end{pmatrix}$, we can identify $\R^4$, $\mathbb{H}$ and the real subspace $M_2(\mathbb{C})$ spanned by $\mathbf{1}, \mathbf{i}, \mathbf{j}, \mathbf{k}$, namely we identify $x=(x_1,x_2,x_3,x_4)\in \R^4$ with 
\be\label{quat2}
x= x_1\mathbf{1}+ x_2\mathbf{i}+x_3\mathbf{j}+ x_4\mathbf{k}\in \mathbb{H}\subset M_2(\mathbb{C}),
\ee
and $SU(2)$ correspond to the set of unitary quaternions, i.e. $x$ as in \eqref{quat2} with $x_1^2+\dots+x_4^2=1$. 
With this identification
\be\label{normquat}
|\mathbf{1}|=|\mathbf{i}|^2=|\mathbf{j}|^2=|\mathbf{k}|^2=2\ .
\ee
Moreover, given $p=p_\mathbf{i}\,\mathbf{i}+p_\mathbf{j}\,\mathbf{j}+p_\mathbf{k}\,\mathbf{k}\in \Im m({\mathbb H})$ we have
\be\label{normquat}
\begin{split}
|p|^2&= \mathrm{tr}( p\, \ov{p})= -\mathrm{tr}\lf((p_\mathbf{i}\,\mathbf{i}+p_\mathbf{j}\,\mathbf{j}+p_\mathbf{k}\,\mathbf{k})(p_\mathbf{i}\,\mathbf{i}+p_\mathbf{j}\,\mathbf{j}+p_\mathbf{k}\,\mathbf{k})\rg)\\
&= -(p_\mathbf{i}^2+p_\mathbf{j}^2+p_\mathbf{k}^2)\mathrm{tr}(\mathbf{1})=-2(p_\mathbf{i}^2+p_\mathbf{j}^2+p_\mathbf{k}^2)\\
&=- 2 \Re (p\, p)
\end{split}
\ee
Using \eqref{quat} we can write
\be
\label{I.1}
A=\sum_{l=1}^4    A^l\ dx_l=\sum_{l=1}^4 \lf(A_{\mathbf i}^l \ {\mathbf i}+A_{\mathbf j}^l \ {\mathbf j}+ A_{\mathbf k}^l \ {\mathbf k}\rg)\ dx_l\ ,
\ee
with pointwise norm
$$|A(x)|^2= \sum_{l=1}^4\lf( |A_\mathbf{i}^l(x)|^2 |\mathbf{i}|^2+ |A_\mathbf{j}^l(x)|^2 |\mathbf{j}|^2+ |A_\mathbf{k}^l(x)|^2 |\mathbf{k}|^2 \rg)=2 \sum_{l=1}^4\lf( |A_\mathbf{i}^l(x)|^2+ |A_\mathbf{j}^l(x)|^2 + |A_\mathbf{k}^l(x)|^2  \rg)\ .$$
We also called such a $A$ a {\it connection one form}\footnote{We shall be exclusively considering $su(2)$ connections in this work.} on ${\R}^4$.

Let us now define for $A,B\in  \Omega^1(\R^4)\otimes su(2) $
\be\label{defbra}
[A,B](X,Y)=\frac12\lf([A(X),B(Y)]-[A(Y),B(X)]\rg)\ .
\ee
In particular
\be\label{defbra2}
[p\, dx_i, q\, dx_j]=\frac12[p,q] dx_i\wedge dx_j,\quad \text{for } p,q\in su(2).
\ee
Then, the curvature of $A$ is the two form taking values into $su(2)$ given by\footnote{We are taking the most spread notation according to which
\[
d x_{l}\wedge dx_{m}(X,Y)=X_l\,Y_m-X_m\,Y_{l}\ .
\]}
\be
\label{I.2}
\begin{array}{l}
\ds\forall\ X,Y\in {\R}^4\quad F_A(X,Y):=dA(X,Y)+[A(X),A(Y)]\\[5mm]
\ds= \sum_{l,m=1}^4 \lf(\p_{x_l} A^m-\p_{x_m}A^l\rg)\ X_l\ Y_m+ \lf[\sum_{l=1}^4 A^l\, X_l,\sum_{m=1}^4 A^m\, Y_m\rg]\\[5mm]
\ds= \sum_{l,m=1}^4 \lf(\p_{x_l} A^m-\p_{x_m}A^l+[A^l,A^m]\rg)\ X_l\ Y_m\ ,
\end{array}
\ee
and naturally we introduce the notation
\be
\label{I.3}
F_A^{lm}:=\p_{x_l} A^m-\p_{x_m}A^l+[A^l,A^m]\ .
\ee
The norm of $F_A$ is given by
$$|F_A|=\sum_{1\le i<j\le 4} \lf|F_A(\partial_{x_i},\partial_{x_j}) \rg|^2,$$
the norms on the RHS being given by \eqref{normquat}.

We recall the definition of the Hodge operator $\star$ on the space of alternated 2-forms in ${\R}^4$
\be
\label{I.15}
\forall \al\,,\,\beta\in \wedge^2{\R}^4\quad \al\wedge\star\beta=\lf<\al,\beta\rg>\ \star {1}\ .
\ee
where $\lf<\cdot,\cdot\rg>$ denotes the standard scalar product on $\wedge^2{\R}^4$ and $\star 1$ is the standard volume form on ${\R}^4$ given by
\be
\label{I.16}
 \star {1}=dx_1\wedge dx_2\wedge dx_3\wedge dx_4\ .
\ee
A form $\al$ is called {\it self-dual} (resp. {\it anti-self dual}) if $\star\al=\al$ (resp. $\star\al=-\al$). The 3 dimensional vector subspace of $\wedge^2{\R}^4$ made of self-dual (resp. anti-self dual) form is
denoted $\wedge^2_+{\R}^4$ (resp. $\wedge^2_-{\R}^4$). An orthonormal basis of  $(\wedge^2{\R}^4)^+$ (resp. $(\wedge^2{\R}^4)^-$) is given by
\be
\label{I.17}
\lf\{
\begin{array}{l}
\ds \om^\pm_{\mathbf i}=\frac{1}{\sqrt{2}}\ \lf(dx_1\wedge dx_2\pm dx_3\wedge dx_4\rg)\\[5mm]
\ds\om^\pm_{\mathbf j}=\frac{1}{\sqrt{2}}\ \lf(dx_1\wedge dx_3\pm dx_4\wedge dx_2\rg)\\[5mm]
\ds\om^\pm_{\mathbf k}=\frac{1}{\sqrt{2}}\ \lf(dx_1\wedge dx_4\pm dx_2\wedge dx_3\rg)
\end{array}
\rg.
\ee
It is important for later purposes to observe
\begin{itemize}
\item[ ]{i)} The sub-vector spaces $(\wedge^2{\R}^4)^+$ and $(\wedge^2{\R}^4)^-$ are orthogonal to each-other and
\[
\wedge^2{\R}^4=(\wedge^2{\R}^4)^+\oplus(\wedge^2{\R}^4)^-\ .
\]
\item[ ]{ii)} The family $( \om^+_{\mathbf i},\, \om^+_{\mathbf j},\, \om^+_{\mathbf k},\, \om^-_{\mathbf i},\, \om^-_{\mathbf j}, \,  \om^-_{\mathbf k})$ realizes an orthonormal basis of $\wedge^2{\R}^4$.
\end{itemize}
 We shall denote $P_\pm$ the orthogonal projections onto $(\wedge^2{\R}^4)^\pm$. We recall the well-known
\be\label{trFA}
\tr (F_A\wedge F_A)=-|P_+ F_A|^2+|P_-F_A|^2
\ee
 
Also observe that
\be
\label{I.18}
\begin{array}{rl}
\ds dx\wedge d\ov{x}&\ds=(\mathbf {1}\,dx_1+{\mathbf i}\,dx_2+{\mathbf j}\,dx_3+{\mathbf k}\,dx_4)\wedge (\mathbf {1}\,dx_1-{\mathbf i}\,dx_2-{\mathbf j}\,dx_3-{\mathbf k}\,dx_4)\\[5mm]
 &\ds=-2\,\sqrt{2}\,\lf(\om^+_{\mathbf i}\,{\mathbf i}+\om^+_{\mathbf j}\,{\mathbf j}+\om^+_{\mathbf k}\,{\mathbf k}  \rg)
\end{array}
\ee
and
\be
\label{I.21}
\begin{array}{rl}
\ds d\ov{x}\wedge d{x}&\ds=(\mathbf {1}\,dx_1-\mathbf{i}\,dx_2-\mathbf{j}\,dx_3-\mathbf {k}\,dx_4)\wedge (\mathbf {1}\,dx_1+{\mathbf i}\,dx_2+{\mathbf j}\,dx_3+{\mathbf k}\,dx_4)\\[5mm]
 &\ds=2\,\sqrt{2}\,\lf(\om^-_{\mathbf i}\,{\mathbf i}+\om^-_{\mathbf j}\,{\mathbf j}+\om^-_{\mathbf k}\,{\mathbf k}  \rg)
\end{array}
\ee
This implies in particular
\be
\label{I.19}
\star(dx\wedge d\ov{x})=dx\wedge d\ov{x}\ , \quad \star(d\ov{x}\wedge d{x})=-d\ov{x}\wedge d{x}\ .
\ee

A special case of interest in this work is given by
\be
\label{I.4}
A(x):= SD(x)=\Im m\lf(\frac{x\, d\ov{x}}{1+|x|^2}\rg)
\ee
where the  $x\in {\R}^4$ is identified canonically with the quaternion as in \eqref{quat2} and $\ov{x}:=x_1\mathbf{1}-x_2\,{\mathbf i}-x_3\,{\mathbf j}-x_4\,{\mathbf k}$ is its conjugate. 
The curvature of $SD$ is given by the formula (see the appendix for details)
\be
\label{I.13}
 F_{SD}=\frac{dx\wedge d\ov{x}}{(1+|x|^2)^2}\ .
\ee
Pulling back via a dilation of factor $\lambda>0$, gives the rescaled version of $SD$, namely
\be
\label{II.6}
SD_\la(x):=\Im m\lf(\frac{\la^2\, {x}\, d\ov{x}}{1+\la^2\,|x|^2}\rg)\ ,
\ee
with curvature
\be
\label{II.6bis}
 F_{SD_\la}=\frac{\la^2dx\wedge d\ov{x}}{(1+\la^2|x|^2)^2}\ .
\ee
Thanks to \eqref{I.19}, we observe
\be
\label{I.19bis}
\star F_{SD_\la}=F_{SD_\la}\ ,
\ee
that we can also rewrite
\be
\label{I.20}
P_+(F_{SD_\la})=F_{SD_\la}\ \quad\text{or, equivalently,}\quad\ P_-(F_{SD_\la})=0\ .
\ee
Denoting
\be
\label{I.22}
ASD_\la(x):=\Im m\lf(\frac{\la^2\,\ov{x}\, d{x}}{1+\la^2\,|x|^2}\rg)=-\,\Im m\lf(\frac{\la^2\, d\ov{x}\,x}{1+\la^2\,|x|^2}\rg)
\ee
we obtain as before
\be
\label{I.23}
F_{ASD_\la}(x)=\frac{\la^2\,d\ov{x}\wedge dx}{(1+\la^2\,|x|^2)^2}
\ee
Combining (\ref{I.21}) and (\ref{I.23}) is implying
\be
\label{I.24}
P_-(F_{ASD_\la})=F_{ASD_\la}\ \quad\Longleftrightarrow\quad\ P_+(F_{ASD_\la})=0\ \quad\Longleftrightarrow\quad\ \star F_{ASD_\la}=-F_{ASD_\la}\ .
\ee
\subsection{Gauge changes and the exponential/polar/radial gauge}
Let $A$ be a smooth $su(2)$ connection 1-form on ${\R}^4$ and $g$ be a smooth map from ${\R}^4$ into $SU(2)$. Following \cite{Riv-YM}  section III.3, we denote
\be
\label{I.25}
A^g:= g^{-1}\,A\, g+g^{-1}\, dg\ .
\ee
where we identify $g$ with the corresponding unit quaternions. A direct computation (see \cite{Riv-YM} section III.3) gives 
\be
\label{I.26}
F_{A^g}=g^{-1}\,F_A\, g\ .
\ee
The passage from $A$ to $A^g$ corresponds on ${\R}^4$ to a change of trivialization (or {\it gauge change}) of the underlying principal trivial $SU(2)$ bundle ${\R}^4\times SU(2)$ for which $A$ is the one form representing a connection (see \cite{Riv-YM} section III). First of all we proceed to a first smooth gauge change ensuring $A(0)=0$. This is possible by taking
\be
\label{I.26-a}
g:=\exp\lf(-\sum_{l=1}^4 x_l\ A^l(0)\rg)
\ee
From now on we will only work with connection form satisfying $A(0)=0$.

\medskip

 In \cite{Uh1} K.Uhlenbeck is considering a special gauge change $g$ such that ultimately
\be
\label{I.27}
\forall x\in {\R}^4\quad\quad 0=A^g(x)\res \,\frac{\p}{\p r}:= \sum_{l=1}^4 (A^g)^l(x)\ \frac{x_l}{|x|}
\ee
where $\res$ is the contraction operator between forms and vectors (more generally between contravariant and covariant tensors). Such a gauge is called in the literature {\it exponential} (or also {\it polar} or {\it radial}) gauge. We recall for the convenience of the reader the construction of $g$. The gauge $g$ is obtained by ``parallel transporting'' with respect to $A$ the identity at the origin along the rays given by the straight half lines emanating from $0$. Precisely we solve for any $\sigma\in S^3$ and $r\in {\R}_+$
\be
\label{I.28}
\lf\{
\begin{array}{l}
\ds \p_rg(r,\sigma)=- A(r\,\sigma)\, g(r,\sigma)\ \\[5mm]
\ds g(0,\sigma)=\mathbf{1}\ ,
\end{array}
\rg.
\ee
where $\mathbf{1}$ is either the neutral element in the quaternions or the unit matrix in $SU(2)$. The existence and uniqueness $g\in  C^1({\R}_+,SU(2))$  for any $\sigma\in S^3$ follows using the smoothness of $A$ by classical ODE theory. We first claim that $g\in Lip_{loc}(S^3\times {\R}_+)$. Indeed for any pair $\sigma$ and $\sigma'$ in $S^3$ the quotient difference $u(r,\sigma,\sigma'):=(g(r,\sigma)-g(r,\sigma'))/|\sigma-\sigma'| $ satisfies
\be
\label{I.28-a}
\lf\{
\begin{array}{l}
\ds \p_ru(r,\sigma,\sigma')=- A(r\,\sigma)\, u(r,\sigma,\sigma') -\frac{A(r\,\sigma)-A(r\,\sigma')}{|\sigma-\sigma'|} \, g(r,\sigma')\ \\[5mm]
\ds u(0,\sigma,\sigma')=0
\end{array}
\rg.
\ee
Hence we have for $r<R$, that $|g(r,\sigma')|=\sqrt 2$,
\be
\label{I.29}
\p_r|u|\le|\p_ru|\le \|A\|_{L^\infty(B_R(0))}\ |u|+ \sqrt 2 r\, \|\nabla A\|_{L^\infty(B_R(0))}\ .
\ee
This imples
\be
\label{I.30}
\p_r\lf(e^{-r\,\|A\|_{L^\infty(B_R(0))}}\ |u|-\sqrt2^{-1}\,r^2\,\|\nabla A\|_{L^\infty(B_R(0))} \rg)\le 0\ .
\ee
Since $|u|(0,\sigma,\sigma')=0$ we deduce
\be
\label{I.31}
\forall \, r\le R\quad\forall \,\sigma,\sigma'\in S^3\quad\quad\frac{|g(r,\sigma)-g(r,\sigma')|}{|r\,\sigma-r\,\sigma'|}\le\, 2^{-1}\, e^{r\,\|A\|_{L^\infty(B_R(0))}}\,r\,\|\nabla A\|_{L^\infty(B_R(0))} \ 
\ee
We have also from (\ref{I.28})
\be
\label{I.32}
\forall \, r,r'\le R\quad\forall \,\sigma\in S^3\quad\quad\frac{|g(r,\sigma)-g(r',\sigma)|}{|r-r'|}\le\,\|A\|_{L^\infty(B_R(0))}\ .
\ee
Let $0<r<r'$ and $\sigma,\sigma'$ in $S^3$ we have
\be
\label{I.33}
\frac{|g(r,\sigma)-g(r',\sigma')|}{|r\,\sigma-r'\,\sigma'|}\le \frac{|g(r,\sigma)-g(r,\sigma')|}{|r\,\sigma-r'\,\sigma'|}+\frac{|g(r,\sigma')-g(r',\sigma')|}{|r\,\sigma-r'\,\sigma'|}
\ee
Because of the convexity of $B_r(0)$ in ${\R}^4$ the point $r\sigma'$ is the nearest point to $r'\sigma'$ in $B_r(0)$  hence
\be
\label{I.33-a}
|r\,\sigma'-r'\,\sigma'|\le|r\,\sigma-r'\,\sigma'|
\ee
We have also
\be
\label{I.33-b}
|r\sigma-r'\sigma'|^2=r^2+(r')^2-2\, r\,r'\, \sigma\cdot\sigma'\ge 2\,r\,r'-2\, r\,r'\, \sigma\cdot\sigma'
\ee
Cauchy-Schwartz inequality gives $|\sigma\cdot\sigma'|\le |\sigma|\,|\sigma'|=1$ hence
\be
\label{I.33-c}
|r\sigma-r'\sigma'|^2\ge 2\, r\,r'\, (1-\sigma\cdot\sigma')\ge 2\, r^2\, (1-\sigma\cdot\sigma')=r^2+r^2-2\, r\,r\, \sigma\cdot\sigma'=|r\sigma-r\sigma'|^2
\ee
Combining (\ref{I.33}) with (\ref{I.33-a}) and (\ref{I.33-c}) gives then
\be
\label{I.33-d}
\frac{|g(r,\sigma)-g(r',\sigma')|}{|r\,\sigma-r'\,\sigma'|}\le \frac{|g(r,\sigma)-g(r,\sigma')|}{|r\,\sigma-r\,\sigma'|}+\frac{|g(r,\sigma')-g(r',\sigma')|}{|r\,\sigma'-r'\,\sigma'|}
\ee
Combining now (\ref{I.33-d}) with (\ref{I.31}) and (\ref{I.32}) is implying
\be
\label{I.33-e}
\frac{|g(r,\sigma)-g(r',\sigma')|}{|r\,\sigma-r'\,\sigma'|}\le 2^{-1}\, e^{r\,\|A\|_{L^\infty(B_R(0))}}\,r\,\|\nabla A\|_{L^\infty(B_R(0))} +\|A\|_{L^\infty(B_R(0))}\ .
\ee
The map $g$ obviously extend to a map on ${\R}^4$ and we have established the following lemma
\begin{Lm}
\label{lm-exp-gauge-lip}
Let $A$ be a smooth $su(2)$ connection one form on ${\R}^4$ such that $A(0)=0$. Let $g$ be the map from ${\R}^4$ into $SU(2)$ solving (\ref{I.28}). Then $g\in C^1({\R}^4, SU(2))$ and for any $R>0$ the following estimate holds:
\be
\label{I.34}
\|\nabla g\|_{L^\infty(B_R(0))}\le 2^{-1}\, e^{r\,\|A\|_{L^\infty(B_R(0))}}\,r\,\|\nabla A\|_{L^\infty(B_R(0))} +\|A\|_{L^\infty(B_R(0))}\ .
\ee
moreover $A^g$ is continuous at the origin and $A^g(0)=0$.
\hfill $\Box$
\end{Lm}
\noindent{\bf Proof of lemma~\ref{lm-exp-gauge-lip}}
The fact that $g$ is locally Lipschitz in ${\R}^4$ and that the estimate (\ref{I.34}) holds is a direct consequence of (\ref{I.33-e}). From the fact that $A(0)=0$ we obtain that
\be
\label{I.34-a}
\lim_{R\rightarrow 0}\|\nabla g\|_{L^\infty(B_R(0))}=0
\ee
This implies the lemma.\hfill $\Box$

\medskip

We recall the standard trivialisation of $TS^3$ by the following orthonormal tangent frame
\be
\label{I.35}
\lf\{
\begin{array}{l}
\ds e_1:=x_1\, \p_{x_2}-x_2\,\p_{x_1}+x_3\,\p_{x_4}-x_4\,\p_{x_3}\\[3mm]
\ds e_2:=x_1\, \p_{x_3}-x_3\,\p_{x_1}+x_4\,\p_{x_2}-x_2\,\p_{x_4}\\[3mm]
\ds e_3:=x_1\, \p_{x_4}-x_4\,\p_{x_1}+x_2\,\p_{x_3}-x_3\,\p_{x_2}\ ,
\end{array}
\rg.
\ee
To which we add the radial vector
$$\partial_r= x_1 \partial_{x_1}+ x_2 \partial_{x_2}+x_3 \partial_{x_3}+x_4 \partial_{x_4}\ .$$
We extend this frame by $0$-homogeneity to $T(\R^4\setminus\{0\}$).

Its dual basis is given for $x\in S^3$ by
\be
\label{I.36}
\lf\{
\begin{array}{l}
\ds e^\ast_1:=x_1\, d{x_2}-x_2\,d{x_1}+x_3\,d{x_4}-x_4\,d{x_3}\\[3mm]
\ds e^\ast_2:=x_1\, d{x_3}-x_3\,d{x_1}+x_4\,d{x_2}-x_2\,d{x_4}\\[3mm]
\ds e^\ast_3:=x_1\, d{x_4}-x_4\,d{x_1}+x_2\,d{x_3}-x_3\,d{x_2}\\[3mm]
\ds dr:= x_1 dx_1+ x_2 dx_2+ x_3 dx_3+ x_4dx_4\ ,
\end{array}
\rg.
\ee
again extended by $0$-homogeneity to $T^*(\R^4\setminus \{0\})$.
We recall also the three standard complex structure $(I,J,K)$ on ${\R}^4$
\be
\label{I.37}
\lf\{
\begin{array}{l}
\ds I\p_{x_1}=\p_{x_2}\quad\mbox{ and }\quad I\p_{x_3}=\p_{x_4}\\[3mm]
\ds J\p_{x_1}=\p_{x_3}\quad\mbox{ and }\quad J\p_{x_4}=\p_{x_2}\\[3mm]
\ds K\p_{x_1}=\p_{x_4}\quad\mbox{ and }\quad K\p_{x_2}=\p_{x_3}
\end{array}
\rg.
\ee
With these notations we have for instance
\be
\label{I.38}
e_1= I\p_r\quad,\quad e_2=J\p_r\quad\mbox{ and }\quad e_3=K\p_r\ ,
\ee
and we have by duality
\be
\label{I.39}
e^\ast_1= Idr\quad,\quad e^\ast_2=Jdr\quad\mbox{ and }\quad e^\ast_3=K dr\ ,
\ee
Given $A^g$ in exponential gauge as above, we can then write
\be
\label{I.40}
A^g=\sum_{l=1}^3A^g_{e_l}\ e_l^\ast \quad\mbox{ where }\quad A^g_{e_l}=A^g\res e_l\ .
\ee
Using the fact that $A^g\res \p_r=0$ we write for any $l=1,2,3$
\be
\label{I.41}
F_{A^g}(\p_r, e_l)= dA^g(\p_r, e_l)+[A^g\res \p_r,A^g\res e_l]=dA^g(\p_r, e_l)\ .
\ee
Using the {\it Cartan Formula} for the exterior derivative of one forms we obtain
\be
\label{I.42}
dA^g(\p_r, e_l)=d(A^g\res e_l)\res \p_r-d(A^g\res \p_r)\res e_l-A^g\res [\p_r,e_l]=\p_r\lf(A^g_{e_l}\rg)-A^g\res [\p_r,e_l]
\ee
where we used again $A^g\res \p_r=0$. We have for any vector field $Y$ in ${\R}^4$
 \be
 \label{I.42-a}
 [r\,\p_r,Y]=\sum_{i,j}\lf[x_j\,\p_{x_j}Y^i-Y^j\,\delta_{ij}\rg]\, \p_{x_i}=\sum_{i}\, r\,\p_rY^i\, \p_{x_i}-Y=r\,\p_rY-Y\ .
 \ee
 In particular, for $l=1,2,3$
 \be
 \label{I.42-b}
  [r\,\p_r,e_{l}]=-e_{l}
 \ee
 Using the formula
 \be
 \label{I-42-c}
 [fX,Y]=f\,[X,Y]-\nabla_Yf\,X
 \ee
 we obtain
 \be
 \label{I-42-d}
 [\p_r,e_{l}]=-r^{-1}\,e_{l}
 \ee
 and
 \be
 \label{I-42-e}
 F_{A^g}(\p_r, e_l)=dA^g(\p_r, e_l)=\p_r\lf(A^g_{e_l}\rg)+r^{-1}\,A^g\res e_l=r^{-1}\,\p_r\lf(r\,A^g_{e_l}\rg)\ .
 \ee
Combining (\ref{I.40}), (\ref{I.41}) and (\ref{I.42}) we obtain the following expression of the connection form in exponential gauge in terms of the curvature (which is maybe the main reason why this gauge is often used
in the literature)
\be
\label{I.43}
A^g(r\sigma)=r^{-1}\,\sum_{l=1}^3\int_0^r F_{A^g}(t\sigma)(\p_r,e_l)\ t\,dt\ e_l^\ast\ .
\ee
where we have used from lemma~\ref{lm-exp-gauge-lip} that $A^g(0)=0$. 

\medskip

From this expression we deduce the following lemma
\begin{Lm}
\label{lm-exp-gauge-dec}
Let $A$ be a smooth $su(2)$ connection one form on ${\R}^4$ such that $A(0)=0$. Let $g$ be the map from ${\R}^4$ into $SU(2)$ solving (\ref{I.28}). Then 
\be
\label{I.44}
\forall R>0\quad\forall x\in B_R(0)\quad\quad |A^g(x)|\le \|F_A\|_{L^\infty(B_R(0))}\ |x|\ ,
\ee
and
$$A^g(x)=\frac{|x|}{2}\sum_{l=1}^3 F_{A^g}(0)(\p_r, e_l) e_l^\ast +O(|x|^2)\, ,\quad \text{as }|x|\to 0\ .$$
\hfill $\Box$
\end{Lm}

\section{Self-instanton insertion and proof of Theorem \ref{th-I.1}}\label{sec:III}
\reset
 Let $A$ be the exponential gauge of a smooth one form connection which was originally chosen to be zero at the origin (unlike the previous section we do not mention the gauge change bringing to the exponential configuration anymore). To fix ideas we assume that $A$ is a connection form issued from a smooth connection of a principal $SU(2)$ bundle $E$ over $S^4$ via the stereographic projection in such a way that we have respectively
\be
\label{II.0}
\int_{{\R}^4}|F_A|^2\ dx^4<+\infty
\ee
and
\be
\label{II.00}
c_2(E)=\int_{{\R}^4}\mbox{tr}\lf(F_A\wedge F_A\rg)=\int_{{\R}^4}\lf(-|P_+F_A|^2+|P_-F_A|^2\rg)\, dx^4= 8\pi^2\ k\ ,\quad\mbox{where } k\in\Z\ .
\ee
Let $g_0\in SU(2)$ (i.e. $g_0$ is constant on ${\R}^4$)  to be fixed later. In a first step we consider the gauge transformed of $A$ given by
\be
\label{II.2}
A^{g_0}=g_0^{-1}\, {A}\, g_0
\ee
This gauge is still exponential. Now, away from zero, we proceed to a gauge change of degree $+1$. Precisely we denote
\be
\label{II.3}
\ti{A}^{g_0}:= \frac{{x}}{|x|}\,A^{g_0}\, \frac{\ov{x}}{|x|}+ \frac{{x}}{|x|}\, d\lf( \frac{\ov{x}}{|x|}\rg)\ .
\ee
It is important at this stage to stress the fact that this new gauge (singular at the origin) is still exponential in the sense that
\be
\label{II.4}
\ti{A}^{g_0}\res \p_r=0\ .
\ee
 Let $\eta$ be a smooth increasing function from ${\R}_+$ into $[0,1]$ such that $\eta\equiv 0$ on $[0,\tau]$, for some $\tau\in (0,1)$ to be fixed later, and
 $\eta\equiv 1$ on $[1,+\infty)$. For $\rho>0$ sufficiently small, we denote
 \be
 \label{II.5}
 \forall x\in {\R}^4\quad\quad\eta_\rho(x):=\eta\lf(\frac{|x|}{\rho}\rg)\ .
 \ee
Moreover, for a constant $c_0>0$ to be fixed independently of $\rho$ we set
\be\label{defla}
\la^2=\frac{1}{c_0\rho^4}, \quad \la\to \infty\text{ as }\rho \to 0^+\ .
\ee
We introduce the following glueing between $\ti{A}^{g_0}$ and $SD_\la$ (compare to \eqref{II.6}). First, for any $1$-form $A$, we set 
\be
\label{II.6-a}
\underline{A}:=\phi_\rho^\ast A\quad \text{in } B_\rho(0)\setminus B_{\tau \rho}(0),
\ee
where $\phi_\rho(x):= \rho\, x/|x|$\ ,
and
\be
\underline{A} := A \quad \text{in } \R^4\setminus B_{\rho}(0).
\ee
Then we define
\be
\label{II.7a}
\hat{A}(g_0,\rho,\la)\ds:= \eta_\rho\, \underline{\ti{A}}^{g_0}+(1-\eta_\rho)\,SD_\la.
\ee
Notice that $\underline{\ti{A}}^{g_0}$ is continuous thanks to \eqref{II.4}, but not smooth. For the sake of Theorem \ref{th-I.1} (estimate \eqref{0.3} in particular) this is not an issue, since $d\underline{\ti{A}}^{g_0}\in L^2_{loc}(\R^4\setminus\{0\}, \Lambda^2\R^4\otimes su(2))$,\footnote{

We have
\[
d\underline{\tilde{A}}^{g_0}=d{\tilde{A}}^{g_0}\ {\mathbf 1}_{{\mathbb R}^4\setminus B_\rho(0)}+
d(\phi_\rho^* \tilde{A}^{g_0}) \ {\mathbf 1}_{B_\rho(0)}\qquad\mbox{ in }{\mathcal D}'({\mathbb R}^4\setminus\{0\})\ ,
\]
where ${\mathbf 1}_{{\mathbb R}^4\setminus B_\rho(0)}$ and $ {\mathbf 1}_{B_\rho(0)}$ are respectively the characteristic functions of ${\mathbb R}^4\setminus B_\rho(0)$ and $B_\rho(0)$. Indeed, for any $\omega\in \Omega^2({\mathbb R}^4\setminus\{0\})$, 
\[
\begin{split}
&\int_{\R^4} \underline{\ti{A}}^{g_0}\wedge d\omega= \int_{B_\rho(0)} \phi_\rho^*{\ti{A}}^{g_0}\wedge d\omega+ \int_{\R^4\setminus B_\rho(0)} {\ti{A}}^{g_0}\wedge d\omega \\
&= - \int_{B_\rho(0)} d\phi_\rho^*{\ti{A}}^{g_0}\wedge \omega-  \int_{\R^4\setminus B_\rho(0)} d{\ti{A}}^{g_0}\wedge \omega +\int_{\partial B_\rho(0)} (\phi_\rho^*{\ti{A}}^{g_0}- {\ti{A}}^{g_0})\wedge \omega\ ,
\end{split}
\]
and the last integral vanishes since $\iota^*_{\partial B_{\rho}(0)} \phi_\rho^*\ti{A}^{g_0}= \iota^*_{\partial B_{\rho}(0)}\ti{A}^{g_0}.$} hence $\underline{\ti{A}}^{g_0}$ can be approximated through convolution by $\underline{\ti{A}}^{g_0}_\ve:= \underline{\ti{A}}^{g_0} * \varphi_\ve$ in such a way that $\underline{\ti{A}}^{g_0}_\ve\to \underline{\ti{A}}^{g_0}$ in $L^4_{loc}(\R^4\setminus\{0\})$ and $d \underline{\ti{A}}^{g_0}_\ve=(d\underline{\ti{A}}^{g_0})*\varphi_\ve\to d\underline{\ti{A}}^{g_0}$ in $L^2_{loc}(\R^4\setminus\{0\})$, hence $F_{\underline{\ti{A}}^{g_0}_\ve}\to F_{\underline{\ti{A}}^{g_0}}$ in $L^2_{loc}(\R^4\setminus\{0\})$ strongly.
Using the identities
\be\label{II.8-a}
\frac{{x}}{|x|}\, d\lf( \frac{\ov{x}}{|x|}\rg)=\Im m\lf(\frac{{x}}{|x|}\, d\lf( \frac{\ov{x}}{|x|}\rg)\rg)=\Im m\lf(\frac{{x}}{|x|^2}\, d\ov{x}\rg)=\frac{{x}}{|x|}\,\Im m\lf(\frac{d\ov{x}\,{x}}{|x|^2}\,\rg)\frac{\ov{x}}{|x|} \ ,
\ee
and
\be
\label{II.8}
\Im m\lf(\frac{\ov{x}}{|x|}\,\frac{\la^2\, {x}\, d\ov{x}}{1+\la^2\,|x|^2}\,\frac{{x}}{|x|}- \frac{d\ov{x}\,{x}}{|x|^2}\rg)=-\frac{1}{|x|^2}\,\Im m\lf(\frac{d\ov{x}\,{x}}{1+\la^2\,|x|^2}\rg)=-\frac{1}{1+\la^2\,|x|^2}\,d\lf( \frac{\ov{x}}{|x|}\rg)\,\frac{{x}}{|x|}
\ee
we obtain
\be\label{II.7}
\begin{array}{l}
\ds\hat{A}(g_0,\rho,\la)\ds= \frac{{x}}{|x|}\,\lf(  \eta_\rho \,g_0^{-1}\, \underline{A}\, g_0 +(1-\eta_\rho)\ \frac{\ov{x}}{|x|}\Im m\lf(\frac{\la^2\, {x}\, d\ov{x}}{1+\la^2\,|x|^2}\rg) \frac{{x}}{|x|}  \rg)\frac{\ov{x}}{|x|}+  \eta_\rho \,\frac{{x}}{|x|}\, d\lf( \frac{\ov{x}}{|x|}\rg)\\[5mm]
\ds= \frac{{x}}{|x|}\,\lf(  \eta_\rho \,g_0^{-1}\, \underline{A}\, g_0 +(1-\eta_\rho)\ \Im m\lf( \frac{\ov{x}}{|x|}\,\frac{\la^2\, {x}\, d\ov{x}}{1+\la^2\,|x|^2}\,\frac{{x}}{|x|}-  \frac{d\ov{x}\,{x}}{|x|^2}\rg)\rg)\frac{\ov{x}}{|x|}+ \frac{{x}}{|x|}\, d\lf( \frac{\ov{x}}{|x|}\rg)\\[5mm]
\ds= \frac{{x}}{|x|}\,\lf(  \eta_\rho \,g_0^{-1}\, \underline{A}\, g_0 -(1-\eta_\rho)\ \frac{1}{1+\la^2\,|x|^2}\,d\lf( \frac{\ov{x}}{|x|}\rg)\rg)\frac{\ov{x}}{|x|}+ \frac{{x}}{|x|}\, d\lf( \frac{\ov{x}}{|x|}\rg) \ ,
\end{array}
\ee
hence
\be
\label{II.9}
\hat{A}(g_0,\rho,\la)=\frac{{x}}{|x|}\check{A}(g_0,\rho,\la)\frac{\ov{x}}{|x|}+ \frac{{x}}{|x|}\, d\lf( \frac{\ov{x}}{|x|}\rg)
\ee
where
\be
\label{II.10}
\check{A}(g_0,\rho,\la)= \eta_\rho \,g_0^{-1}\, \underline{A}\, g_0 +(1-\eta_\rho)\ \widetilde{SD_\lambda} \ ,
\ee
and
\be\label{SDtilde}
\widetilde{SD_\lambda}=\frac{\ov{x}}{|x|} SD_\lambda \frac{x}{|x|}+\frac{{x}}{|x|}\, d\lf( \frac{\ov{x}}{|x|}\rg) =-\frac{1}{1+\la^2\,|x|^2}\,d\lf( \frac{\ov{x}}{|x|}\rg)\frac{{{x}}}{|x|}
\ee
is the self-instanton after a change of gauge .

Observe that we have
\be
\label{II.10-a}
\lf\{
\begin{array}{l}
\ds F_{\hat{A}}=\frac{{x}}{|x|}\, g_0^{-1}\, F_A\, g_0\,\frac{\ov{x}}{|x|}\quad\quad\mbox{ in }{\R}^4\setminus B_\rho(0)\\[5mm]
\ds F_{\hat{A}}=F_{SD}\quad\quad\mbox{ in }B_{\tau\rho }(0)
\end{array}
\rg.
\ee

\begin{Lma}\label{degreechange} We have
$$\int_{\R^4}\mathrm{tr}\lf(F_{\hat{A}}\wedge F_{\hat{A}}\rg)= \int_{\R^4}\mathrm{tr}\lf(F_{{A}}\wedge F_{{A}}\rg)-8\pi^2\ .$$
\end{Lma}

\begin{proof}
Using the Chern-Simons formula (see \cite{CS,FU})
\be\label{transg}
\mbox{tr}\lf(F_{{A}}\wedge F_{{A}}\rg) = d\, \mbox{tr}\lf(A\wedge dA+\frac23 A\wedge [A,A]\rg)\ ,
\ee
we get, for $\ve>0$ sufficiently small,
\be\label{II.11quater}
\begin{split}
\int_{B_R(0)}\mbox{tr}\lf(F_{\hat{A}}\wedge F_{\hat{A}}\rg)&=\int_{\p B_R(0)}\mbox{tr}\lf( \hat{A}\wedge d\hat{A}+ \frac{1}{3}\hat{A}\wedge[\hat{A},\hat{A}]\rg)\\
&=\int_{\p B_R(0)}\mbox{tr}\lf( \ti{A}^{g_0}\wedge d\ti{A}^{g_0}+ \frac{1}{3}\ti{A}^{g_0}\wedge[\ti{A}^{g_0},\ti{A}^{g_0}]\rg)\\
&=\int_{B_R(0)\setminus B_\eps(0)}\mbox{tr}\lf(F_{\ti{A}^{g_0}}\wedge F_{\ti{A}^{g_0}}\rg)  +\int_{\p B_\eps(0)}\mbox{tr}\lf( \ti{A}^{g_0}\wedge d\ti{A}^{g_0}+ \frac{1}{3}\ti{A}^{g_0}\wedge[\ti{A}^{g_0},\ti{A}^{g_0}]\rg)\\
&=\int_{B_R(0)\setminus B_\eps(0)}\mbox{tr}\lf(F_{A^{g_0}}\wedge F_{A^{g_0}}\rg)  +\int_{\p B_\eps(0)}\mbox{tr}\lf( A^{g_0}\wedge d{A}^{g_0}+ \frac{1}{3}{A}^{g_0}\wedge[{A}^{g_0},{A}^{g_0}]\rg)\\
&\quad + \int_{\p B_\eps(0)}\mbox{tr}(Q) \\
& =\int_{B_R(0)} \mbox{tr}\lf(F_{A^{g_0}}\wedge F_{A^{g_0}}\rg) + \int_{\p B_\eps(0)}\mbox{tr}(Q)\ ,
\end{split}
\ee
where $Q$ denotes a $3$-form with the expansion
$$Q=\frac{x}{|x|}d\lf( \frac{\ov{x}}{|x|}\rg) \wedge d\lf(  \frac{x}{|x|}\rg)\wedge d\lf(\frac{\ov{x}}{|x|}\rg)+\frac23  \frac{x}{|x|}d\lf(\frac{\ov{x}}{|x|}\rg) \wedge\lf[ \frac{x}{|x|}d\lf( \frac{\ov{x}}{|x|}\rg), \frac{x}{|x|}d\lf( \frac{\ov{x}}{|x|}\rg)\rg]+O(\eps^{-2}),$$
as $\eps\to 0$.
We have
\be\begin{array}{l} 
\ds \frac{x}{|x|}d\lf( \frac{\ov{x}}{|x|}\rg) \wedge d\lf(  \frac{x}{|x|}\rg)\wedge d\lf(\frac{\ov{x}}{|x|}\rg)=\frac{x}{|x|}d\lf( \frac{\ov{x}}{|x|}\rg) \wedge d\lf(  \frac{x}{|x|}\rg)\frac{\ov{x}}{|x|}\wedge \frac{x}{|x|}d\lf(\frac{\ov{x}}{|x|}\rg)\\[5mm]
\ds =-\frac{\mathbf{i}Idr+\mathbf{j}Jdr+\mathbf{k}Kdr}{|x|}\wedge \frac{\mathbf{i}Idr+\mathbf{j}Jdr+\mathbf{k}Kdr}{|x|}\wedge \lf(-\frac{\mathbf{i}Idr+\mathbf{j}Jdr+\mathbf{k}Kdr}{|x|}\rg)\\[5mm]
\ds =\frac{6\,\mathbf{i}\, \mathbf{j}\,\mathbf{k}}{|x|^3} Idr\wedge Jdr\wedge Kdr=- \frac{6}{|x|^3} Idr\wedge Jdr\wedge Kdr.
\end{array}\ee
We have 
\be\begin{split}
\lf[ \frac{x}{|x|}d\lf( \frac{\ov{x}}{|x|}\rg), \frac{x}{|x|}d\lf( \frac{\ov{x}}{|x|}\rg)\rg]&=\lf[-\frac{\mathbf{i}Idr+\mathbf{j}Jdr+\mathbf{k}Kdr}{|x|},-\frac{\mathbf{i}Idr+\mathbf{j}Jdr+\mathbf{k}Kdr}{|x|}\rg]\\
&=\frac{2}{|x|^2}\lf(\mathbf{i}\, Jdr\wedge Kdr+ \mathbf{j}\, Kdr\wedge Idr+\mathbf{k}\, Idr\wedge Jdr\rg),
\end{split}
\ee
hence
\be\begin{split}
&\frac23  \frac{x}{|x|}d\lf(\frac{\ov{x}}{|x|}\rg) \wedge\lf[ \frac{x}{|x|}d\lf( \frac{\ov{x}}{|x|}\rg), \frac{x}{|x|}d\lf( \frac{\ov{x}}{|x|}\rg)\rg]\\
&=- \frac{4}{3|x|^3}(\mathbf{i}\,Idr+\mathbf{j}\,Jdr+\mathbf{k}\,Kdr)\wedge \lf(\mathbf{i}\, Jdr\wedge Kdr+ \mathbf{j}\, Kdr\wedge Idr+\mathbf{k}\, Idr\wedge Jdr\rg) \\
&=\frac{4}{|x|^3}\mathbf{1}\,Idr\wedge Jdr\wedge Kdr.
\end{split}
\ee
Then $$Q =-\frac{2}{|x|^3}\, \mathbf{1}\,Idr\wedge Jdr\wedge Kdr +O(\eps^{-2}),$$
and, recalling that $\mbox{tr}(\mathbf{1})=2$,
$$\int_{\p B_\eps(0)}\mbox{tr}(Q)=- \frac{1}{\eps^3}\int_{\p B_\eps(0)} 4\,Idr\wedge Jdr\wedge Kdr +O(\eps)=-8\pi^2+O(\eps).$$
Letting $\ve\to 0$ in \eqref{II.11quater} we conclude.
\end{proof}


Our goal now is to estimate
\be
\label{II.11}
\int_{{\R}^4}|F_{\hat{A}}|^2\ dx^4-\int_{{\R}^4}|F_{A}|^2\ dx^4\,.
\ee

\begin{Lma}\label{LmaIII.2}
We have
\be
\label{II.15}
\begin{array}{l}
\ds\int_{{\R}^4}|F_{\hat{A}}|^2\ dx^4-\int_{{\R}^4}|F_{A}|^2\ dx^4\\[5mm]
\ds=8\pi^2+2\,\int_{B_\rho(0)\setminus B_{\tau\rho }(0)}|P_-F_{\check{A}}|^2\ dx^4-2\,\int_{B_{\rho}(0)}|P_-F_{A}|^2\ dx^4\ .
\end{array}
\ee

\end{Lma}

\noindent\textbf{Proof of Lemma \ref{LmaIII.2}.}
Thanks to Lemma \ref{degreechange} and \eqref{trFA} we have
\be
\label{II.12}
\begin{array}{l}
\ds\int_{{\R}^4}|F_{\hat{A}}|^2\ dx^4-\int_{{\R}^4}|F_{A}|^2\ dx^4\\[5mm]
\ds=\int_{{\R}^4}|P_+F_{\hat{A}}|^2\ dx^4-\int_{{\R}^4}|P_+F_{A}|^2\ dx^4+\int_{{\R}^4}|P_-F_{\hat{A}}|^2\ dx^4-\int_{{\R}^4}|P_-F_{A}|^2\ dx^4\\[5mm]
\ds=8\pi^2+2\,\int_{{\R}^4}|P_-F_{\hat{A}}|^2\ dx^4-2\,\int_{{\R}^4}|P_-F_{{A}}|^2\ dx^4
\end{array}
\ee
Combining (\ref{II.10-a}) with the fact that $F_{SD_\la}$ is self-dual gives
\be
\label{II.13}
\begin{array}{l}
\ds\int_{{\R}^4}|F_{\hat{A}}|^2\ dx^4-\int_{{\R}^4}|F_{A}|^2\ dx^4=8\pi^2+2\,\int_{B_\rho(0)\setminus B_{\tau\rho }(0)}|P_-F_{\hat{A}}|^2\ dx^4-2\,\int_{B_{\rho}(0)}|P_-F_{A}|^2\ dx^4\\[5mm]
\end{array}
\ee
Because of (\ref{II.9}) we have on $B_\rho(0)\setminus B_{\tau\rho }(0)$
\be
\label{II.14-a}
F_{\hat{A}}=\frac{\ov{x}}{|x|}\,F_{\check{A}}\,\frac{x}{|x|}\ ,
\ee
and in particular
\be
\label{II.14}
P_-F_{\hat{A}}=\frac{\ov{x}}{|x|}\,P_-F_{\check{A}}\,\frac{x}{|x|}\ .
\ee
Then \eqref{II.12} follows at once.
\hfill$\square$

\medskip

We are now estimating the difference in \eqref{II.12} in the limit $\rho\rightarrow 0$ by taking $\la$ as in \eqref{defla}, where $c_0$ is going to be chosen later independent of $\rho$.

We will use the following basic estimates:

\begin{Lma}\label{basicest}
We have
\be
\begin{split}
&\frac{1}{1+\lambda^2|x|^2}=\frac{1}{\lambda^2|x|^2}+O(\rho^4)=O(\rho^2)\\
&\widetilde{SD_\lambda}=O(\rho)\\
&d\widetilde{SD_\lambda}=O(1)
\end{split}
\ee
uniformly in $B_\rho(0)\setminus B_{\tau\rho}(0)$ as $\rho\to 0$, and
\be
\begin{split}
&\check A=O(\rho)\\
&d\check A=O(1)
\end{split}
\ee
uniformly in $B_\rho(0)$ as $\rho\to 0$.
\end{Lma}

\noindent\textbf{Proof of Lemma \ref{basicest}.}
The first estimate follows at once from the definition of $\lambda$. The second one follows from the definition of $\widetilde{SD_\lambda}$ and $d\left(\frac{\ov{x}}{x}\right)=O(\rho^{-1})$ in $B_\rho(0)\setminus B_{\tau\rho}(0)$.
Combining (\ref{I.44}) in Lemma~\ref{lm-exp-gauge-dec} with (\ref{II.14}) is implying
\be
\label{II.17}
\lf\| \check{A} \rg\|_{L^\infty(B_\rho(0)\setminus B_{\tau\rho }(0))}=O(\rho)\ .
\ee
We have also for any $x\in B_\rho(0)\setminus B_{\tau\rho }(0)$
\be
\label{II.18}
\begin{array}{l}
|d\check{A}|\le \|\eta'\|_\infty\ \rho^{-1} \lf[|A(x)|+|x|^{-1}\,|(1+\la^2\,|x|^2)^{-1}|\rg]+|dA|(x)+\la^2\, (1+\la^2\,|x|^2)^{-2}\\[5mm]
\ds \le O(1)+|F_A|(x)+O(\rho^2)+O(\la^2 \rho^4)=O(1)
\end{array}
\ee
\hfill$\square$

\medskip

\begin{Lma}\label{LmaIII.2bis}
We have
\be
\label{II.20}
\begin{array}{l}
\ds\int_{{\R}^4}|F_{\hat{A}}|^2\ dx^4-\int_{{\R}^4}|F_{A}|^2\ dx^4\\[5mm]
\ds=8\pi^2+2\,\int_{B_\rho(0)\setminus B_{\tau\rho }(0)}|P_-d\check{A}|^2\ dx^4-2\,\int_{B_{\rho}(0)}|P_-F_{A}|^2\ dx^4+O(\rho^6)\ .
\end{array}
\ee
\end{Lma}

\noindent\textbf{Proof of Lemma \ref{LmaIII.2bis}.}
From Lemma \ref{basicest} we have
$$F_{\check A}= d\check A+O(\rho^2) \quad \text{in }B_\rho(0)\setminus B_{\tau\rho }(0),$$
hence
$$P_-F_{\check A}=P_-(d\check A)+O(\rho^2) \quad \text{in }B_\rho(0)\setminus B_{\tau\rho }(0),$$
and
$$|P_-F_{\check A}|^2=|P_-(d\check A)|^2+O(\rho^2) \quad \text{in }B_\rho(0)\setminus B_{\tau\rho }(0).$$
Then \eqref{II.20} follows at once from Lemma \ref{LmaIII.2}.
\hfill$\square$

\medskip

This brings us naturally to the estimation of the square of the $L^2$ norm of $P_-d\check{A}$ on $B_\rho(0)\setminus B_{\tau\rho }(0)$.

\begin{Lma}\label{LmadcheckA} We have
\be\label{eqP-dA}
\begin{split}
&P_-(d\check A)=\eta'_\rho P_-(d|x|\wedge\underline{A}^{g_0})+\eta_\rho P_-(d \underline{A}^{g_0})-\frac{1}{4\lambda^2|x|^3}\eta_\rho' d\ov{x}\wedge dx+O(\rho^2)\\
&=\eta'_\rho g_0^{-1} P_-(d|x|\wedge\underline{A})g_0+\eta_\rho g_0^{-1} P_-(d \underline{A})g_0-\frac{1}{4\lambda^2|x|^3}\eta_\rho' d\ov{x}\wedge dx+O(\rho^2)
\end{split}
\ee
uniformly in $B_{\rho}(0)\setminus B_{\tau \rho}$ as $\rho\to 0$.
\end{Lma}

\noindent\textbf{Proof of Lemma \ref{LmadcheckA}.}
From the definition of $\check A=\eta_\rho \underline{A}^{g_0}+(1-\eta_\rho)\widetilde{SD_\lambda}$ we get
$$d\check A= \eta_\rho' d|x|\wedge(\underline{A}^{g_0}-\widetilde{SD_\lambda}) +\eta_\rho d\underline{A}^{g_0} +(1-\eta_\rho)d\widetilde{SD_\lambda}.$$
With Lemma \ref{basicest} we have, uniformly in $B_\rho(0)\setminus B_{\tau\rho }(0)$
$$d\widetilde{SD_\lambda} = F_{{SD_\lambda}}+O(\rho^2),$$
and, by self-duality,
$$P_- (d\widetilde{SD_\lambda})=O(\rho^2).$$
Therefore
\be
\label{II.27-c0}
P_-(d\check A)= \eta_\rho'P_-( d|x|\wedge\underline{A}^{g_0})  +\eta_\rho P_-( d\underline{A}^{g_0} )- \eta_\rho' P_-(d|x|\wedge \widetilde{SD_\lambda})+O(\rho^2).
\ee
Now observe that
\be
\label{II.27-c}
2\,|x|^3\,P_-\lf(d|x|\wedge d\lf( \frac{\ov{x}}{|x|}\rg)\,\frac{{x}}{|x|}\rg)=P_-\lf(d|x|^2\wedge d {\ov{x}}\,{{x}}\rg)
\ee
Moreover
\be
\label{II.27-d}
\begin{array}{l}
\ds P_-\lf(d|x|^2\wedge d {\ov{x}}\,{{x}}\rg)=P_-\lf(d(\ov{x}\,x)\wedge d {\ov{x}}\,{{x}}\rg)=\ov{x}\,P_-\lf(\,dx\wedge d {\ov{x}}\rg)\,x+P_-\lf( d {\ov{x}}\,{{x}}\wedge d {\ov{x}}\,{{x}} \rg)\\[5mm]
\ds =P_-\lf( d {\ov{x}}\,{{x}}\wedge d {\ov{x}}\,{{x}} \rg)=P_-\lf( d {\ov{x}}\,{{x}}\wedge d |x|^2 \rg)-P_-\lf( d {\ov{x}}\,{{x}}\wedge {\ov{x}}\,d{{x}} \rg)\\[5mm]
\ds=-P_-\lf(d|x|^2\wedge d {\ov{x}}\,{{x}}\rg)-|x|^2\,d\ov{x}\wedge dx
\end{array}
\ee
This gives
\be
\label{II.27-e}
\begin{array}{l}
 P_-\lf(d|x|^2\wedge d {\ov{x}}\,{{x}}\rg)=-\frac{|x|^2}{2}\, d\ov{x}\wedge dx\ .
\end{array}
\ee
Hence combining (\ref{II.27-c}) and (\ref{II.27-e}) we obtain
\be
\label{II.27-f}
P_-\lf(d|x|\wedge d\lf( \frac{\ov{x}}{|x|}\rg)\,\frac{{x}}{|x|}\rg)=-\frac{1}{4\,|x|}\, d\ov{x}\wedge dx\ .
\ee
In particular
\be
\label{II.27-f'}
P_-(d|x|\wedge \widetilde{SD_\lambda})=-\frac{1}{1+\la^2|x|^2}P_-\lf(d|x|\wedge d\lf( \frac{\ov{x}}{|x|}\rg)\,\frac{{x}}{|x|}\rg)=\frac{1+O(\rho^2)}{4\la^2|x|^2}d\ov{x}\wedge dx.
\ee
Inserting \eqref{II.27-f'} into \eqref{II.27-c0} we conclude.
\hfill$\square$

\medskip

In order to compute $|d\check A|^2$ we express $P_-(d|x|\wedge \underline{A})$ and $P_-(d\underline{A})$ in terms of the orthonormal basis $\{\omega_\mathbf{i}^-(x),\omega_\mathbf{j}^-(x),\omega_\mathbf{k}^-(x)\}$ given by the following lemma, whose proof is postponed to the Appendix.

\begin{Lma}\label{Lemmaframe}
Given $x\in \R^4\setminus \{0\}$, the set of $2$-forms
\be
\label{II.31-j}
\lf\{
\begin{array}{l}
\ds\om_{\mathbf i}^-(x)=\sqrt{2}\, P_-(  dr\wedge Idr)\\[5mm]
\ds\om_{\mathbf j}^-(x)=\sqrt{2}\, P_-(  dr\wedge Jdr)\\[5mm]
\ds\om_{\mathbf k}^-(x)=\sqrt{2}\, P_-(  dr\wedge Kdr)
\end{array}
\rg.
\ee
realizes an orthonormal basis of $(\Lambda^2{\R}^4)^-$.
\end{Lma}

The compatibility with the previous notations is given by
\be
\label{II.31-k}
\lf\{
\begin{array}{l}
\ds\om_{\mathbf i}^-=\om_{\mathbf i}^-(1,0,0,0)\\[5mm]
\ds\om_{\mathbf j}^-=\om_{\mathbf j}^-(1,0,0,0)\\[5mm]
\ds\om_{\mathbf k}^-=\om_{\mathbf k}^-(1,0,0,0)
\end{array}
\rg.
\ee

\begin{Lma}\label{LmaP-dA}
On $B_\rho(0)\setminus B_{\tau\rho }(0)$ there holds
\be
\label{II.28}
\begin{array}{rl}
\ds P_-\lf(d|x|\wedge \un{A}\rg)&\ds= \frac{\rho^2}{2\sqrt{2}|x|}\,\lf(F_A(0)(\p_r,I\p_r)\ \om^-_{\mathbf i}(x)+F_A(0)(\p_r,J\p_r)\ \om^-_{\mathbf j}(x)+F_A(0)(\p_r,K\p_r)\ \om^-_{\mathbf k}(x)\rg)\\[5mm]
\ds\quad\quad&\ds\ +O(\rho^2)\, ,
\end{array}
\ee
and
\be
\label{II.30}
\begin{array}{rl}
\ds P_-d\un{A}&= \ds -\frac{\rho^2}{\sqrt{2}|x|^2}\lf( F_A(0)(J\p_r,K\p_r)\, \om_{\mathbf i}^-(x)+ F_A(0)(K\p_r,I\p_r) \, \om_{\mathbf j}^-(x)+\ F_A(0)(I\p_r,J\p_r) \, \om_{\mathbf k}^-(x)\rg)\\[5mm]
\ds\quad\quad&\ds\ +O(\rho)\ ,
\end{array}
\ee
as $\rho\to 0$.
\end{Lma}

\noindent\textbf{Proof of Lemma~\ref{LmaP-dA}}
We have
\be
\label{II.27-i}
\begin{array}{l}
\ds P_-\lf(d|x|\wedge \un{A}\rg)=\un{A}\res I\p_r\ P_-\lf(dr\wedge I dr\rg)+\un{A}\res J\p_r\ P_-\lf(dr\wedge J dr\rg)+\un{A}\res K\p_r\ P_-\lf(dr\wedge K dr\rg)\\
=\ds\frac{1}{\sqrt{2}}\ \lf(\un{A}\res I\p_r\ \om^-_{\mathbf i}+\un{A}\res J\p_r\ \om^-_{\mathbf j}+\un{A}\res K\p_r\ \om^-_{\mathbf k}\rg)
\end{array}
\ee
From Lemma \ref{lm-exp-gauge-dec} we get
\be
\label{II.28-a}
\lf\{
\begin{array}{l}
\ds\un{A}\res I\p_r=(\phi_\rho^\ast A)\res I\p_r=A(\rho)\res((\phi_\rho)_\ast I\p_r)=\frac{\rho}{|x|} A(\rho)\res(I\p_r)=\frac{\rho^2}{2\,|x|}\,F_A(0)(\p_r,I\p_r)+O(\rho^2)\\[5mm]
\ds\un{A}\res J\p_r=(\phi_\rho^\ast A)\res J\p_r=A(\rho)\res((\phi_\rho)_\ast J\p_r)=\frac{\rho}{|x|} A(\rho)\res(J\p_r)=\frac{\rho^2}{2\,|x|}\,F_A(0)(\p_r,J\p_r)+O(\rho^2)\\[5mm]
\ds\un{A}\res K\p_r=(\phi_\rho^\ast A)\res K\p_r=A(\rho)\res((\phi_\rho)_\ast K\p_r)=\frac{\rho}{|x|} A(\rho)\res(K\p_r)=\frac{\rho^2}{2\,|x|}\,F_A(0)(\p_r,K\p_r)+O(\rho^2)
\end{array}
\rg.
\ee
and \eqref{II.28} follows.

In order to prove \eqref{II.30}, observe that $\underline{A}=A$ on $\p B_\rho(0)$, while in $B_\rho(0)\setminus B_{\tau \rho}(0)$
\be
\begin{array}{l}
\ds \phi_\rho^\ast(dr\wedge Idr)=\phi_\rho^\ast(dr\wedge Jdr)=\phi_\rho^\ast(dr\wedge Kdr) =0\,,\\
[5mm]
\ds \phi_\rho^\ast(I dr\wedge Jdr)= \frac{\rho^2}{|x|^2}I dr\wedge Jdr   \,,\\
[5mm]
\ds \phi_\rho^\ast(J dr\wedge Kdr)= \frac{\rho^2}{|x|^2}J dr\wedge Kdr   \,,\\
[5mm]
\ds \phi_\rho^\ast(K dr\wedge Idr)= \frac{\rho^2}{|x|^2}K dr\wedge Idr   \,.
\end{array}
\ee
Then we have
\be
\label{II.6-b}
\begin{array}{l}
\ds d\underline{A}=\phi_\rho^\ast dA=\phi_\rho^\ast F_A+O(\rho)=F_A(0)(I\p_r,J\p_r)\ \phi_\rho^\ast(Idr\wedge Jdr)+F_A(0)(J\p_r,K\p_r)\ \phi_\rho^\ast(Jdr\wedge Kdr)\\[5mm]
\ds\ \ +F_A(0)(K\p_r,I\p_r)\ \phi_\rho^\ast(Kdr\wedge Idr)+O(\rho)\, ,
\end{array}
\ee
hence
\be
\label{II.6-c}
\begin{array}{rl}
\ds d\underline{A}=\lf(\frac{\rho}{|x|}\rg)^2& \lf[F_A(0)(I\p_r,J\p_r)\ Idr\wedge Jdr+F_A(0)(J\p_r,K\p_r)\ Jdr\wedge Kdr\rg.\\[5mm]
\ds\quad&\ds\lf.+F_A(0)(K\p_r,I\p_r)\ Kdr\wedge Idr\rg]+O(\rho)\ .
\end{array}
\ee
Observe that we have respectively
\be
\label{II.29}
\star (dr\wedge Idr)=Jdr\wedge Kdr\ ,\ \star (dr\wedge Jdr)=Kdr\wedge Idr\ ,\  \star (dr\wedge Kdr)=Idr\wedge Jdr
\ee
This gives in particular
\be
\label{II.29-a}
\lf\{
\begin{array}{l}
\ds P_-( Idr\wedge Jdr)= -\,P_-(dr\wedge Kdr)=-\sqrt{2}^{-1}\,\om_{\mathbf k}^-(x)\\[5mm]
\ds P_-( Jdr\wedge Kdr)= -\,P_-(dr\wedge Idr)=-\sqrt{2}^{-1}\,\om_{\mathbf i}^-(x)\\[5mm]
\ds P_-(Kdr\wedge Idr)=-\,P_-(dr\wedge Jdr)=-\sqrt{2}^{-1}\,\om_{\mathbf j}^-(x)
\end{array}
\rg.
\ee
and \eqref{II.30} follows.
\hfill$\square$

\medskip

Combining (\ref{II.28}) and (\ref{II.30}) with \eqref{eqP-dA} finally gives
in $B_\rho(0)\setminus B_{\tau\rho }(0)$

\be
\label{III.35}
\begin{array}{l}
g_0\,P_- d{\check{A}}\,g_0^{-1}\\[5mm]
\ds=\frac{1}{2\,\sqrt{2}}\,\frac{\rho^2}{|x|}\, \eta_\rho'\, \lf(F_A(0)(\p_r,I\p_r)\ \om^-_{\mathbf i}(x)+F_A(0)(\p_r,J\p_r)\ \om^-_{\mathbf j}(x)+F_A(0)(\p_r,K\p_r)\ \om^-_{\mathbf k}(x)\rg)\,\\[5mm]
\ds-\frac{1}{\sqrt{2}}\,\frac{\rho^2}{|x|^2}\,\eta_\rho\lf(F_A(0)(J\p_r,K\p_r) \, \om_{\mathbf i}^-(x)+F_A(0)(K\p_r,I\p_r)\, \om_{\mathbf j}^-(x)+F_A(0)(I\p_r,J\p_r)\, \om_{\mathbf k}^-(x)\rg)\,\\[5mm]
\ds\quad-\frac{1}{4\,\la^2\,|x|^3}\ \eta'_\rho \, g_0\,d\ov{x}\wedge dx\, g_0^{-1} +O(\rho)
\end{array}
\ee
Thus
\be
\label{III.53}
\begin{array}{l}
\ds\int_{B_\rho(0)\setminus B_{\tau\rho}(0)} |P_-d\check{A}|^2\ dx^4\\[5mm]
\ds=\frac{\rho^2}{8}\int_{\tau\rho}^\rho\ s \ \eta'\lf(\frac{s}{\rho}\rg)^2\ ds\ \int_{S^3}\lf(|F_A(0)(\p_r,I\p_r)|^2+  |F_A(0)(\p_r,J\p_r)|^2 +|F_A(0)(\p_r,K\p_r)|^2\rg) dvol_{S^3}\\[5mm]
\ds+\frac{\rho^4}{2}\int_{\tau\rho}^\rho\ \eta\lf(\frac{s}{\rho}\rg)^2\frac{ds}{s}\ \int_{S^3}\lf(|F_A(0)(I\p_r,J\p_r)|^2+  |F_A(0)(J\p_r,K\p_r)|^2 +|F_A(0)(K\p_r,I\p_r)|^2\rg) dvol_{S^3}\\[5mm]
\ds+\frac{1}{16\,\la^4\,\rho^2}\int_{\tau\rho}^\rho\, \eta'\lf(\frac{s}{\rho}\rg)^2 \frac{ds}{s^3}\int_{S^3}|d\ov{x}\wedge dx|^2 dvol_{S^3}\\[5mm]
\ds-\frac{\rho^3}{2}\,\int_{\tau \rho}^\rho   \eta'\lf(\frac{s}{\rho}\rg)\eta\lf(\frac{s}{\rho}\rg) \ ds\,\int_{S^3}\big[\lf<F_A(0)(\p_r,I\p_r) ,F_A(0)(J\p_r,K\p_r)\rg>\\[5mm]
\ds +\lf<F_A(0)(\p_r,J\p_r) ,F_A(0)(K\p_r,I\p_r)\rg>+\lf<F_A(0)(\p_r,K\p_r) ,F_A(0)(I\p_r,J\p_r)\rg>\big]\ dvol_{S^3}\\[5mm]
\ds-\frac{1}{4\,\sqrt{2}\,\la^2}\int_{\tau\rho}^\rho \eta'\lf(\frac{s}{\rho}\rg)^2\frac{ds}{s}\\[5mm]
\ds \times  \int_{S^3}\lf<g_0\,d\ov{x}\wedge dx\, g_0^{-1},F_A(0)(\p_r,I\p_r)\ \om^-_{\mathbf i}(x)+F_A(0)(\p_r,J\p_r)\ \om^-_{\mathbf j}(x)+F_A(0)(\p_r,K\p_r)\ \om^-_{\mathbf k}(x)\rg>dvol_{S^3}\\[5mm]
\ds+\frac{\rho}{2\,\sqrt{2}\, \la^2}\int_{\tau\rho}^\rho   \eta'\lf(\frac{s}{\rho}\rg) \eta\lf(\frac{s}{\rho}\rg)\ \frac{ds}{s^2}\\[5mm]
\ds \times \int_{S^3}\lf<g_0\,d\ov{x}\wedge dx\, g_0^{-1},F_A(0)(J\p_r,K\p_r)\ \om^-_{\mathbf i}(x)+F_A(0)(K\p_r,I\p_r)\ \om^-_{\mathbf j}(x)+F_A(0)(I\p_r,J\p_r)\ \om^-_{\mathbf k}(x)\rg>dvol_{S^3}\\[5mm]
\ds+O(\rho^5)
\end{array}
\ee

\begin{Lma}\label{lemmaint0}
We have
\be
\begin{array}{l}
\ds \int_{S^3}\lf(|F_A(0)(\p_r,I\p_r)|^2+  |F_A(0)(\p_r,J\p_r)|^2 +|F_A(0)(\p_r,K\p_r)|^2\rg) dvol_{S^3}\\[5mm]
\ds=\int_{S^3}\lf(|F_A(0)(I\p_r,J\p_r)|^2+  |F_A(0)(J\p_r,K\p_r)|^2 +|F_A(0)(K\p_r,I\p_r)|^2\rg) dvol_{S^3}\\[5mm]
\ds=\pi^2|F_A(0)|^2\ .
\end{array}
\ee
\end{Lma}

\noindent{\bf Proof of Lemma~\ref{lemmaint0}.}
We have
\be
\label{n-II.57-1}
\begin{array}{l}
\ds\int_{S^3}\lf(|F_A(0)(I\p_r,J\p_r)|^2+  |F_A(0)(J\p_r,K\p_r)|^2 +|F_A(0)(K\p_r,I\p_r)|^2\rg) dvol_{S^3}\\[5mm]
\ds=2\,\pi^2\,|F_A(0)|^2-\int_{S^3}|F_A(0)\res \p_r|^2\ dvol_{S^3}\ .
\end{array}
\ee
Moreover
\be
\label{n-II.57-2}
F_A(0)=\sum_{i<j} F^{ij}_A(0)\ dx_i\wedge dx_j=\frac{1}{2}\sum_{i,j=1}^4F^{ij}_A(0)\ dx_i\wedge dx_j\ ,
\ee
and on $S^3$
\be
\label{n-II.57-3}
F_A(0)\res\p_r=\frac{1}{2}\sum_{i,j=1}^4F^{ij}_A(0)\ dx_i\, x_j-\frac{1}{2}\sum_{i,j=1}^4F^{ij}_A(0)\ x_i\, dx_j=\sum_{i=1}^4\lf(\sum_{j=1}^4F^{ij}_A(0)\  x_j\rg)\, dx_i\ .
\ee
Hence
\be
\label{n-II.57-4}
\begin{array}{l}
\ds\int_{S^3}|F_A(0)\res\p_r|^2\ dvol_{S^3}=\sum_{i=1}^4\int_{S^3}\lf|\sum_{j=1}^4F^{ij}_A(0)\  x_j\rg|^2\ dvol_{S^3}\\[5mm]
\ds=\sum_{i=1}^4\sum_{j,l=1}^4 \lf<F^{ij}_A(0),F^{il}_A(0)\rg>\  \int_{S^3}x_j\, x_l\ dvol_{S^3}\\[5mm]
\ds=\sum_{i=1}^4\sum_{j=1}^4 |F^{ij}_A(0)|^2\ \int_{S^3}x_j^2\ dvol_{S^3}=\frac{2\pi^2}{4}\sum_{i,j=1}^4|F^{ij}_A(0)|^2=\pi^2\,\sum_{i<j}|F^{ij}_A(0)|^2=\pi^2\,|F_A(0)|^2 \ .
\end{array}
\ee
\hfill $\square$

\begin{Lma}\label{lemmaint01}
We have
\be
\label{II.45}
\begin{array}{l}
\ds \int_{S^3}\lf<F_A(0)(\p_r,I\p_r) ,F_A(0)(J\p_r,K\p_r)\rg>+\lf<F_A(0)(\p_r,J\p_r) ,F_A(0)(K\p_r,I\p_r)\rg>\  dvol_{S^3}\\[5mm]
\ds+\int_{S^3} \lf<F_A(0)(\p_r,K\p_r) ,F_A(0)(I\p_r,J\p_r)\rg>\ dvol_{S^3}=\frac{\pi^2}{\sqrt{2}}\,\lf[ |P_+F_A(0)|^2-|P_-F_A(0)|^2\rg]
\end{array}
\ee
\end{Lma}

\noindent{\bf Proof of Lemma~\ref{lemmaint01}.}
We have
\be
\label{II.3000}
\begin{array}{rl}
\ds P_\pm F_{A}(0)&= \ds \sqrt{2}^{-1}\ \lf( F_A(0)(\p_r, I\p_r)\pm F_A(0)(J\p_r,K\p_r) \rg)\, \om_{\mathbf i}^\pm(x)\\[5mm]
&\ds+ \sqrt{2}^{-1}\ \lf( F_A(0)(\p_r, J\p_r)\pm F_A(0)(K\p_r,I\p_r) \rg)\, \om_{\mathbf j}^\pm(x)\\[5mm]
&\ds+ \sqrt{2}^{-1}\ \lf( F_A(0)(\p_r, K\p_r)\pm F_A(0)(I\p_r,J\p_r) \rg)\, \om_{\mathbf k}^\pm(x)\ .
\end{array}
\ee
Then
\be
\label{II.42-a}
\begin{array}{l}
\ds |P_+F_A(0)|^2-|P_-F_A(0)|^2=\frac{4}{\sqrt{2}}\, \lf<F_A(0)(\p_r,I\p_r) ,F_A(0)(J\p_r,K\p_r)\rg>\\[5mm]
\ds+\frac{4}{\sqrt{2}}\,\,\lf<F_A(0)(\p_r,J\p_r) ,F_A(0)(K\p_r,I\p_r)\rg>
 +\frac{4}{\sqrt{2}}\,\lf<F_A(0)(\p_r,K\p_r) ,F_A(0)(I\p_r,J\p_r)\rg>\ ,
\end{array}
\ee
and \eqref{II.45} follows.
\hfill $\square$

\medskip

We denote
\be
\label{rep-0}
\lf\{
\begin{array}{l}
\ds {\mathbf i}_{g_0}:=g_0\,{\mathbf i}\,g_0^{-1}\\[3mm]
\ds {\mathbf j}_{g_0}:=g_0\,{\mathbf j}\,g_0^{-1}\\[3mm]
\ds {\mathbf i}_{g_0}:=g_0\,{\mathbf k}\,g_0^{-1}
\end{array}
\rg.
\ee
\begin{Lma}\label{lemmaint1} We have
\be
\label{rep-1}
\begin{array}{l}
 \ds\int_{S^3}\,\lf<g_0\,d\ov{x}\wedge dx\, g_0^{-1},F_A(0)(\p_r,I\p_r)\ \om^-_{\mathbf i}(x)\rg>\ dvol_{S^3}\\[5mm]
 \ds=\int_{S^3}\,\lf<g_0\,d\ov{x}\wedge dx\, g_0^{-1},F_A(0)(\p_r,J\p_r)\ \om^-_{\mathbf j}(x)\rg>\ dvol_{S^3}\\[5mm]
 \ds=\int_{S^3}\,\lf<g_0\,d\ov{x}\wedge dx\, g_0^{-1},F_A(0)(\p_r,K\p_r)\ \om^-_{\mathbf k}(x)\rg>\ dvol_{S^3}\\[5mm]
\ds= \frac{\sqrt{2}}{3}\, \pi^2\,\lf<g_0\,d\ov{x}\wedge dx\, g_0^{-1}, P_-F_A(0)\rg>\ .\\[5mm]
 \end{array}
\ee

\end{Lma}

\begin{Lma}\label{lemmaint2}
We have
\be
\label{rep-19bis}
\begin{array}{l}
 \ds\int_{S^3}\,\lf<g_0\,d\ov{x}\wedge dx\, g_0^{-1},F_A(0)(J\p_r,K\p_r)\ \om^-_{\mathbf i}(x)\rg>\ dvol_{S^3}\\[5mm]
 \ds=\int_{S^3}\,\lf<g_0\,d\ov{x}\wedge dx\, g_0^{-1},F_A(0)(K\p_r,I\p_r)\ \om^-_{\mathbf j}(x)\rg>\ dvol_{S^3}\\[5mm]
 \ds=\int_{S^3}\,\lf<g_0\,d\ov{x}\wedge dx\, g_0^{-1},F_A(0)(I\p_r,J\p_r)\ \om^-_{\mathbf k}(x)\rg>\ dvol_{S^3}\\[5mm]
\ds= -\frac{\sqrt{2}}{3}\, \pi^2\,\lf<g_0\,d\ov{x}\wedge dx\, g_0^{-1}, P_-F_A(0)\rg>\ .
 \end{array}
\ee
\end{Lma}

The computations in the proofs of Lemma \ref{lemmaint1} and Lemma \ref{lemmaint2} are rather lengthy and are postponed to the appendix.

Lemmas \ref{lemmaint0}, \ref{lemmaint01}, \ref{lemmaint1} and \ref{lemmaint2} imply
\be
\label{III.53bis}
\begin{array}{l}
\ds\int_{B_\rho(0)\setminus B_{\tau\rho}(0)} |P_-d\check{A}|^2\ dx^4=\lf(\frac{\rho^2}{8}\int_{\tau\rho}^\rho\ s \ \eta'\lf(\frac{s}{\rho}\rg)^2\ ds+\frac{\rho^4}{2}\int_{\tau\rho}^\rho\ \eta\lf(\frac{s}{\rho}\rg)^2\frac{ds}{s}\rg)\pi^2|F_A(0)|^2\\[5mm]
\ds+\frac{6\pi^2}{\la^4\,\rho^2}\int_{\tau\rho}^\rho\, \eta'\lf(\frac{s}{\rho}\rg)^2 \frac{ds}{s^3}-\frac{\pi^2\rho^3}{2\sqrt{2}}\,\int_{\tau \rho}^\rho   \eta'\lf(\frac{s}{\rho}\rg)\eta\lf(\frac{s}{\rho}\rg) \ ds\, \lf(|P_+F_A(0)|^2-|P_-F_A(0)|^2\rg)\\[5mm]
\ds -\frac{\pi^2}{\lambda^2}\lf(\frac{1}{4}\int_{\tau\rho}^\rho \eta'\lf(\frac{s}{\rho}\rg)^2\frac{ds}{s} +\frac{\rho}{2}\int_{\tau\rho}^\rho   \eta'\lf(\frac{s}{\rho}\rg) \eta\lf(\frac{s}{\rho}\rg)\ \frac{ds}{s^2}\rg)\langle g_0 d\ov{x}\wedge dx g_0^{-1},P_-F_A(0)\rangle +O(\rho^5)\ .
\end{array}
\ee
Then, with a change of variables and \eqref{defla}
\be
\label{III.53ter}
\begin{array}{l}
\ds\int_{B_\rho(0)\setminus B_{\tau\rho}(0)} |P_-d\check{A}|^2\ dx^4=\frac{\pi^2\rho^4}{2}\lf(\frac{1}{4}\int_{\tau}^1t \ \eta'(t)^2\ dt+\int_{\tau}^1 \eta(t)^2\frac{dt}{t}\rg)|F_A(0)|^2\\[5mm]
\ds+12 c_0^2\int_{\tau}^1 \eta'(t)^2 \frac{dt}{t^3}-\frac{1}{\sqrt{2}}\,\int_{\tau}^1  \eta'(t)\,\eta(t) \ dt\, \lf(|P_+F_A(0)|^2-|P_-F_A(0)|^2\rg)\\[5mm]
\ds -c_0\lf(\frac{1}{2}\int_{\tau}^1 \eta'(t)^2\frac{dt}{t} +\int_{\tau}^1   \eta'(t) \,\eta(t)\ \frac{dt}{t^2}\rg)\langle g_0 d\ov{x}\wedge dx g_0^{-1},P_-F_A(0)\rangle+O(\rho^5)\\[5mm]
=\ds\frac{\pi^2\rho^4}{2}\bigg[\lf(\frac{1}{4}\int_{\tau}^1t \ \eta'(t)^2\ dt+\int_{\tau}^1 \eta(t)^2\frac{dt}{t}-\frac{1}{\sqrt{2}}\,\int_{\tau}^1  \eta'(t)\,\eta(t) \ dt\rg) |P_+F_A(0)|^2\\[5mm]
\ds+ \lf(\frac{1}{4}\int_{\tau}^1t \ \eta'(t)^2\ dt+\int_{\tau}^1 \eta(t)^2\frac{dt}{t}+\frac{1}{\sqrt{2}}\,\int_{\tau}^1  \eta'(t)\,\eta(t) \ dt\rg) |P_-F_A(0)|^2\\
 \ds -c_0\lf(\frac{1}{2}\int_{\tau}^1 \eta'(t)^2\frac{dt}{t} +\int_{\tau}^1   \eta'(t) \,\eta(t)\ \frac{dt}{t^2}\rg)\langle g_0 d\ov{x}\wedge dx g_0^{-1},P_-F_A(0)\rangle+12 c_0^2\int_{\tau}^1 \eta'(t)^2 \frac{dt}{t^3}\bigg]+O(\rho^5)
\end{array}
\ee
We write
\be
\label{III.54-cv}
\frac{P_-F_A(0)}{|P_-F_A(0)|}={\mathbf a}\, \om_{\mathbf i}^-+{\mathbf b}\, \om_{\mathbf j}^-+{\mathbf c}\, \om_{\mathbf k}^-
\ee
where ${\mathbf a}$, ${\mathbf b}$ and ${\mathbf c}$ belong to $\Im m({\mathbb H})$ and satisfy
\be
\label{III.54-dv}
|{\mathbf a}|^2+|{\mathbf b}|^2+|{\mathbf c}|^2=1\ ,
\ee
and we have
\be
\label{III.54-ev}
\begin{array}{l}
\ds\lf<g_0^{-1}\,\frac{P_-F_A(0)}{|P_-F_A(0)|}\,g_0, {\mathbf i}\,\om_{\mathbf i}^-+  {\mathbf j}\,\om_{\mathbf j}^- +{\mathbf k}\,\om_{\mathbf k}^-\rg>= \lf<{\mathbf a},{\mathbf i}_{g_0}\rg>+\lf<{\mathbf b},{\mathbf j}_{g_0}\rg>+\lf<{\mathbf c},{\mathbf k}_{g_0}\rg>
\end{array}
\ee

We will need the following elementary result.

\begin{Lma}
\label{lm-r3}
For any triple of 3 vectors in ${\R}^3$, $(\vec{a},\vec{b},\vec{c})$ there exists a positive orthonormal basis $(\vec{e}_1,\vec{e}_2,\vec{e}_3)$ such that
\be
\label{II.64-a}
\vec{a}\cdot\vec{e}_1+\vec{b}\cdot\vec{e}_2+\vec{c}\cdot\vec{e}_3\ge \frac{1}{\sqrt{3}}\ \sqrt{|\vec{a}|^2+|\vec{b}|^2+|\vec{c}|^2}
\ee
\hfill$\Box$
\end{Lma}
\noindent{\bf Proof of lemma~\ref{lm-r3}.} By linearity we can assume $\sqrt{|\vec{a}|^2+|\vec{b}|^2+|\vec{c}|^2}=1$. There exists a vector of length at least $1/\sqrt{3}$. Assume this is $\vec{a}$.
We choose $\vec{e}_1:=\vec{a}/|\vec{a}|$. We choose $\vec{e}_2$ and $\vec{e}_3$ arbitrary such that $(\vec{e}_1,\vec{e}_2,\vec{e}_3)$ is forming a positive orthonormal basis. If $$\vec{e}_2\cdot\vec{b}+\vec{e}_3\cdot\vec{c}<0$$
we change $(\vec{e}_1,\vec{e}_2,\vec{e}_3)$  into $(\vec{e}_1,-\vec{e}_2,-\vec{e}_3)$ and we get
\[
\vec{a}\cdot\vec{e}_1+\vec{b}\cdot\vec{e}_2+\vec{c}\cdot\vec{e}_3\ge \frac{1}{\sqrt{3}}\ .
\]
Hence the lemma is proved. \hfill $\Box$

\medskip

Notice that the constant in Lemma \ref{III.53bis} is optimal, as seen by taking $(\vec{a},\vec{a},\vec{a})$ where $|\vec{a}|^2=3^{-1}$.

\noindent\textbf{Proof of Theorem \ref{th-I.1} (completed).}
Recalling \eqref{normquat} and using Lemma~\ref{lm-r3}, we choose $g_0$ such that
\be
\label{III.54-fv}
\lf<g_0^{-1}\,\frac{P_-F_A(0)}{|P_-F_A(0)|}\,g_0, {\mathbf i}\,\om_{\mathbf i}^-+  {\mathbf j}\,\om_{\mathbf j}^- +{\mathbf k}\,\om_{\mathbf k}^-\rg>\ge \sqrt{\frac{2}{3}}\ ,
\ee
so that
\be
\label{III.54-gv}
\lf<g_0d\ov{x}\wedge dx g_0^{-1}, P_-F_A(0) \rg> \ge\frac{4}{\sqrt3}|P_-F_A(0)|\ .
\ee
Moreover, since
$$\frac{1}{4} t\eta'(t)^2+\frac{\eta(t)^2}{t}-\frac{1}{\sqrt{2}}\eta'(t)\eta(t)\ge\lf(\frac{\sqrt{t}}{2}\eta'(t)-\frac{\eta(t)}{\sqrt{t}}\rg)^2\ge 0 $$
and $|P_+F_A(0)|\le |P_-F_A(0)|$, we can further estimate from \eqref{III.53ter}

\be
\label{III.53quater}
\begin{array}{l}
\ds\int_{B_\rho(0)\setminus B_{\tau\rho}(0)} |P_-d\check{A}|^2\ dx^4
\le
\ds\frac{\pi^2\rho^4}{2}\bigg[\lf(\frac{1}{2}\int_{\tau}^1t \ \eta'(t)^2\ dt+2\int_{\tau}^1 \eta(t)^2\frac{dt}{t}\rg) |P_-F_A(0)|^2\\[5mm]
 \ds -\frac{2\, c_0}{\sqrt 3}\lf(\int_{\tau}^1 \eta'(t)^2\frac{dt}{t} +2\int_{\tau}^1   \eta'(t) \,\eta(t)\ \frac{dt}{t^2}\rg)|P_-F_A(0)|+12 c_0^2\int_{\tau}^1 \eta'(t)^2 \frac{dt}{t^3}\bigg]+O(\rho^5)
\end{array}
\ee
Finally:
\be
\label{III.54-c}
\begin{array}{l}
\ds\int_{B_\rho(0)\setminus B_{\tau \rho}(0)} |P_-d\check{A}|^2\ dx^4-\int_{B_\rho(0)}|P_-F_A|^2 dx^4 \\[5mm]
\ds\le\frac{\pi^2\rho^4}{2}\bigg[\lf(\frac{1}{2}\int_{\tau}^1t \ \eta'(t)^2\ dt+2\int_{\tau}^1 \eta(t)^2\frac{dt}{t}- 1\rg) |P_-F_A(0)|^2\\[5mm]
 \ds -\frac{2c_0}{\sqrt 3}\lf(\int_{\tau}^1 \eta'(t)^2\frac{dt}{t} +2\int_{\tau}^1   \eta'(t) \,\eta(t)\ \frac{dt}{t^2}\rg)|P_-F_A(0)|+12 c_0^2\int_{\tau}^1 \eta'(t)^2 \frac{dt}{t^3}\bigg]
 +O(\rho^5)
\end{array}
\ee
Considering now that the minimum of a degree $2$ polynomial $ax^2+bx+c$ is $-\tfrac{b^2}{4a}+c$ and is attained at $x_0=-\tfrac{b^2}{2a}$, we have that  the minimum of the polynomial in $c_0$ inside the square brackets above is
\be
\begin{array}{l}
\ds m(\tau,\eta)=\lf[ -\frac{\lf(\int_{\tau}^1 \eta'(t)^2\frac{dt}{t} +2\int_{\tau}^1   \eta'(t) \,\eta(t)\ \frac{dt}{t^2}\rg)^2}{36\int_{\tau}^1 \eta'(t)^2 \frac{dt}{t^3}}+ \lf(\frac{1}{2}\int_{\tau}^1t \ \eta'(t)^2\ dt+2\int_{\tau}^1 \eta(t)^2\frac{dt}{t}- 1\rg) \rg]  |P_-F_A(0)|^2\\[5mm]
\ds=:\varphi(\tau,\eta) |P_-F_A(0)|^2.
\end{array}
\ee
Choosing
$$\eta_0(t):=\frac{t-\tau}{1-\tau},$$
leads to
\be
\begin{array}{l}
\ds m(\tau,\eta_0)=\lf[ -\frac{2\tau^2 \lf(\frac{3}{2(1-\tau)}\log\frac{1}{\tau} -1 \rg)^2}{9(1-\tau^2)} +  2\lf(\frac{5-11\tau}{8(1-\tau)}+\frac{\tau^2}{(1-\tau)^2}\log\frac{1}{\tau}-\frac{1}{2}\rg) \rg]  |P_-F_A(0)|^2\\[5mm]
\ds=:\varphi(\tau) |P_-F_A(0)|^2
\end{array}
\ee
and is attained for a constant $c_0=c_0(\tau)>0$.
Now, considering that $\varphi(\tau,\eta_0)<0$ for $\tau\in [0.3,0.4]$, we can modify $\eta_0$ to a function $\eta\in C^\infty(\R)$ with $\eta(t)=0$ for $t\le \tau$ and $\eta(t)=1$ for $t\ge 1$, choose $\tau\in [0.3,0.4]$, such that $\varphi(\eta,\tau)<0$ and conclude
\be
\label{III.54-d}
\begin{array}{l}
\ds\int_{B_\rho(0)\setminus B_{\tau \rho}(0)} |P_-d\check{A}|^2\ dx^4-\int_{B_\rho(0)}|P_-F_A|^2 dx^4 \le  \frac{\pi^2\rho^4}{2}\varphi(\tau,\eta) |P_-F_A(0)|^2  +O(\rho^5).
\end{array}
\ee
Together with Lemma \ref{LmaIII.2bis} this concludes the proof of Theorem \ref{th-I.1}. \hfill $\square$

\section{Proof of Theorem \ref{th-I.2}}

Given $\eta$, $\tilde A^{g_0}$ and $SD_{\lambda}$ as in the previous section, with
\be\label{deflambda}
\lambda^2=\frac{1}{\rho^4},
\ee
we consider
\be
\hat{A}:=\eta_{\rho^a}\tilde A^{g_0}+(1-\eta_{\rho^b}) SD_{\lambda},
\ee
where
\be\label{ab}
0<b<1<a,
\ee
and, without loss of generality we can require $\frac{3}{4}<b<1$.

Upon the same change of gauge as in \eqref{II.9}, we can write
\be
\label{II.9bis}
\hat{A}(g_0,\rho,\la)=\frac{{x}}{|x|}\check{A}(g_0,\rho,\la)\frac{\ov{x}}{|x|}+ \frac{{x}}{|x|}\, d\lf( \frac{\ov{x}}{|x|}\rg)
\ee
where
\be
\label{II.10}
\check{A}(g_0,\rho,\la)= \eta_{\rho^a} \,g_0^{-1}\, A g_0 +(1-\eta_{\rho^b})\ \widetilde{SD_\lambda}\ .
\ee
Lemma \ref{degreechange} continues to hold, and with a similar proof, Lemma \ref{LmaIII.2} becomes
\begin{Lma}\label{LmaIII.2bis}
We have
\be
\label{II.15bis}
\begin{array}{l}
\ds\int_{{\R}^4}|F_{\hat{A}}|^2\ dx^4-\int_{{\R}^4}|F_{A}|^2\ dx^4\\[5mm]
\ds=8\pi^2+2\,\int_{B_{\rho^b}(0)\setminus B_{\rho^a/2 }(0)}|P_-F_{\check{A}}|^2\ dx^4-2\,\int_{B_{\rho^b}(0)}|P_-F_{A}|^2\ dx^4\ .
\end{array}
\ee
\end{Lma}

\begin{Lma}\label{basicestbis}
We have, for $1<a<\frac{4}{3}$
\begin{align}
&\frac{1}{1+\lambda^2|x|^2}=\frac{1}{\lambda^2|x|^2}+O\lf(\frac{\rho^8}{|x|^4}\rg)=O(\rho^{4-2a})\label{be1}\\
&\widetilde{SD_\lambda}(x)=O\lf(\frac{\rho^4}{|x|^3}\rg)=O(\rho^{4-3a})\label{be2}\\
&d\widetilde{SD_\lambda}(x)=O\lf(\frac{\rho^4}{|x|^4}\rg)=O(\rho^{4-4a})\label{be3}\\
&P_-(d\widetilde{SD_\lambda})=O\lf(\frac{\rho^8}{|x|^6}\rg)=O(\rho^{8-6a})\label{be4}\\
&\check A(x)=O\lf(|x|+\frac{\rho^4}{|x|^3}\rg)\label{be5}\\
&P_-(d\check A)(x)=O(1)\label{be6}
\end{align}
uniformly for $x\in B_{\rho^b}(0)\setminus B_{\rho^a/2}(0)$ as $\rho\to 0$. More specifically, 
\be\label{be5bis}
\check A(x)=\begin{cases}
O(\rho^{4-3a})&\text{for }x\in B_{\rho}(0)\setminus B_{\rho^a/2}(0)\\
O(|x|)&\text{for }x\in B_{\rho^b}(0)\setminus B_{\rho}(0)
\end{cases}
\ee
\end{Lma}

\noindent\textbf{Proof of Lemma \ref{basicestbis}.}
Estimates \eqref{be1}-\eqref{be3} are immediate consequences of the formulas.
The self-duality of $\widetilde{SD_\lambda}$, together with \eqref{be2} to get
\[
P_-(d\widetilde{SD_\lambda})= P_-(F_{\widetilde{SD_\lambda}})- P_{-}([\widetilde{SD_\lambda},\widetilde{SD_\lambda}])=0+O\lf(\frac{\rho^8}{|x|^6}\rg)=O(\rho^{8-6a}),\quad \text{in }B_{\rho^b(0)}\setminus B_{\rho^a/2}(0)
\]
Estimate \eqref{be5} follows from \eqref{I.44} and \eqref{be2}. As for \eqref{be6}, together with \eqref{be4}, using that $a<\frac{4}{3}$, we have
\[
P_-(d\check A)=O(\eta_{\rho^a}' A)+O(\eta_{\rho^a}A)+O(\eta_{\rho^b}\widetilde{SD_\lambda})+O(d\widetilde{SD_\lambda})=O(1),\quad \text{in }B_{\rho^b(0)}\setminus B_{\rho^a/2}(0).
\]
\hfill$\square$

\medskip

\begin{Lma}\label{LmaIII.2ter}
For $\tfrac23<b<1<a<\tfrac43$ we have
\be
\label{II.20bis}
\begin{array}{l}
\ds\int_{{\R}^4}|F_{\hat{A}}|^2\ dx^4-\int_{{\R}^4}|F_{A}|^2\ dx^4\\[5mm]
\ds=8\pi^2+2\,\int_{B_{\rho^b}(0)\setminus B_{\rho^a/2}(0)}|P_-d\check{A}|^2\ dx^4-2\,\int_{B_{\rho^b}(0)}|P_-F_{A}|^2\ dx^4+O(\rho^{12-6a})+O(\rho^{6b})\ ,
\end{array}
\ee
as $\rho\to 0$.
\end{Lma}

\noindent\textbf{Proof of Lemma \ref{LmaIII.2ter}.}
With \eqref{be5bis} we estimate
$$P_-F_{\check A}=P_-(d{\check A})+[\check A,\check A]=P_-(d{\check A})+O(\rho^{8-6a})\quad \text{in }B_\rho(0)\setminus B_{\rho^a/2}(0),$$
hence, using \eqref{be6},
$$|P_-F_{\check A}|^2=|P_-(d{\check A})|^2+O(\rho^{8-6a})\quad \text{in }B_\rho(0)\setminus B_{\rho^a/2}(0),$$
and
$$\int_{B_\rho(0)\setminus B_{\rho^a/2}(0)} |P_-F_{\check A}|^2 dx^4 =\int_{B_\rho(0)\setminus B_{\rho^a/2}(0)} |P_-(d\check A)|^2 dx^4 + O(\rho^{12-6a}).$$
Similarly, again using \eqref{be5bis} and \eqref{be6}, we infer
$$P_-F_{\check A}=P_-(d{\check A})+O(|x|^2)\quad \text{in }B_{\rho^b}(0)\setminus B_{\rho}(0),$$
and
$$|P_-F_{\check A}|^2=|P_-(d{\check A})|^2+O(|x|^2) \quad \text{in }B_{\rho^b}(0)\setminus B_{\rho}(0).$$
Then
$$\int_{B_{\rho^b}(0)\setminus B_{\rho}(0)}|P_-F_{\check A}|^2 dx^4 = \int_{B_{\rho^b}(0)\setminus B_{\rho}(0)}|d{\check A}|^2 dx^4 + O(\rho^{6b}).$$
Then \eqref{II.20bis} follows from \eqref{II.15bis}.
\hfill$\square$

\medskip

\begin{Lma}\label{LmaIV.4} We have
\be\label{eqP-dAbis}
\begin{split}
&P_-(d\check A)=\eta'_{\rho^a} P_-(d|x|\wedge{A}^{g_0}) +\eta_{\rho^a} P_-(d A^{g_0})-\frac{1}{4\lambda^2|x|^3}\eta_{\rho^b}' d\ov{x}\wedge dx+O\lf(\frac{\rho^8}{|x|^6}\rg)
\end{split}
\ee
uniformly in $B_{\rho^b}(0)\setminus B_{\rho^a/2}$ as $\rho\to 0$.
\end{Lma}

\noindent\textbf{Proof of Lemma \ref{LmaIV.4}.}
This follows essentially as in the proof of Lemma \ref{LmadcheckA},  also using \eqref{be4} to bound the term
$$(1-\eta_{\rho^b})P_-(d\widetilde{SD_\lambda})=O\lf(\frac{\rho^8}{|x|^6}\rg)$$
and \eqref{be1} to bound
$$-\eta_{\rho^b}'P_{-}(d|x|\wedge \widetilde{SD_\lambda})=-\frac{\eta_{\rho^b}'}{1+\la^2|x|^2}\frac{1}{4|x|} d\ov{x}\wedge dx=-\frac{\eta_{\rho^b}'}{\la^2|x|^3} d\ov{x}\wedge dx +\eta_{\rho^b}'O\lf(\frac{\rho^8}{|x|^5}\rg),$$
and
$$\eta_{\rho^b}'O\lf(\frac{\rho^8}{|x|^5}\rg)=O\lf(\frac{\rho^8}{|x|^6}\rg)$$
since $\eta_{\rho^b}'\ne 0$ only in $B_{\rho^b}(0)\setminus B_{\rho^b/2}$, where $\eta_{\rho^b}'=O(\rho^{-b})$, and $b<a$.
\hfill$\square$

\medskip

\begin{Lma}\label{LmaIV.5} We have
\be\label{fe}
\begin{split}
&\int_{B_{\rho^b}(0)\setminus B_{\rho^a/2}(0)}|P_-(d\check A)|^2 dx^4=\int_{B_{\rho^b}(0)}|P_-F_A|^2 dx^4\\
& -\int_{B_{\rho^b}(0)\setminus B_{\rho^b/2}(0)}\frac{\eta_{\rho^b}'}{4\la^2|x|^3}\langle P_-(dA^{g_0}), d\ov{x}\wedge dx\rangle dx^4 +O(\rho^{4+\delta})
\end{split}
\ee
for some $\delta>0$ depending on $a$ and $b$. 
\end{Lma}

\noindent\textbf{Proof of Lemma \ref{LmaIV.5}.}
Recalling \eqref{eqP-dAbis}, we integrate the 10 terms arising in $|P_-(d\check A)|^2$ and estimate them one by one.
We start with
\be\label{fe1}
\int_{B_{\rho^b}(0)\setminus B_{\rho^a}(0)}|\eta_{\rho^a}'|^2 |d|x|\wedge A^{g_0}|^2 dx^4\le  \int_{B_{\rho^a}(0)\setminus B_{\rho^a}(0)} \frac{\|\eta'\|_{L^\infty}^2 }{\rho^{2a}} O(|x|^2) dx^4=O(|B_{\rho^a}(0)|)=O(\rho^{4a}).
\ee
Since
$$F_{A}=O(1)\quad \text{in }B_{\rho^b}(0)\ ,$$
and $\eta_{\rho^a}\equiv 1$ in $B_{\rho^b}(0)\setminus B_{\rho^a/2}(0)$, and $\|\eta_{\rho^a}\|_{L^\infty}=1$, we have
\be\label{fe2}
\int_{B_{\rho^b}(0)\setminus B_{\rho^a/2}(0)}\eta_{\rho^a}^2 |P_-(d A^{g_0})|^2 dx^4  = 
\int_{B_{\rho^b}(0)} |P_-(d A^{g_0})|^2 dx^4+  O(\rho^{4a})=\int_{B_{\rho^b}(0)} |P_-F_A|^2 dx^4+  O(\rho^{4a}).
\ee

From \eqref{be2} we have $\widetilde{SD_\lambda}=O(\rho^{4-3b})$ in $B_{\rho^b}(0)\setminus B_{\rho^b/2}(0)$, hence
We also have
\be\label{fe3}
\int_{B_{\rho^b}(0)\setminus B_{\rho^a/2}(0)}\frac{(\eta_{\rho^b}')^2}{16\la^4|x|^6}\lf|d\ov{x}\wedge dx\rg|^2 dx^4\le \int_{B_{\rho^b}(0)\setminus B_{\rho^b/2}(0)}\frac{\|\eta'\|_{L^\infty}^2}{\rho^{2b}}O(\rho^{8-6b})dx^4=O(\rho^{8-4b}).
\ee

Similarly
\be\label{fe4}
\int_{B_{\rho^b}(0)\setminus B_{\rho^a/2}(0)}\eta_{\rho^a}'\eta_{\rho^a}\langle P_-(d|x|\wedge A^{g_0}), P_-(dA^{g_0})\rangle dx^4\le \int_{B_{\rho^a}(0)\setminus B_{\rho^a/2}(0)} \frac{O(|x|)}{\rho^a} dx^4 =O(\rho^{4a}).
\ee

Because $\eta_{\rho^a}'$ and $\eta_{\rho^b}'$ have disjoint supports, we also get
\be\label{fe5}
\int_{B_{\rho^b}(0)\setminus B_{\rho^a/2}(0)}\eta_{\rho^a}'\eta_{\rho^b}'\lf\langle P_-(d|x|\wedge A^{g_0}), \frac{d\ov{x}\wedge dx}{4\la^2|x|^3}\rg\rangle dx^4 =0.
\ee

Using that $8a=9a-a\le 12-a$ we get
\be\label{fe6}
\int_{B_{\rho^b}(0)\setminus B_{\rho^a/2}(0)} \lf|O\lf(\frac{\rho^8}{|x|^6}\rg)\rg|^2dx^4=O\lf(\rho^{16}\int_{\rho^a/2}^{\rho^b}\frac{1}{r^9}dr\rg)=O(\rho^{16-8a})=O(\rho^{4+a}).
\ee

\be\label{fe7}
\int_{B_{\rho^b}(0)\setminus B_{\rho^a/2}(0)}|\eta_{\rho^a}'P_-(d|x|\wedge A^{g_0})| \lf|O\lf(\frac{\rho^8}{|x|^6}\rg)\rg| dx^4=\int_{B_{\rho^a}(0)\setminus B_{\rho^a/2}(0)} \lf|O\lf(\frac{\rho^{8-a}}{|x|^5}\rg)\rg|dx^4=O(\rho^{8-2a}).
\ee

\be\label{fe8}
\int_{B_{\rho^b}(0)\setminus B_{\rho^a/2}(0)}|\eta_{\rho^a}P_-(d A^{g_0})| \lf|O\lf(\frac{\rho^8}{|x|^6}\rg)\rg| dx^4=O\lf(\rho^8\int_{\rho^a/2}^{\rho^b} \frac{1}{r^3}dr\rg)=O(\rho^{8-2a}).
\ee

Using \eqref{be2} we bound
\be\label{fe9}
\begin{split}
&\int_{B_{\rho^b}(0)\setminus B_{\rho^a/2}(0)}\eta_{\rho^b}'\frac{|d\ov{x}\wedge dx|}{4\la^2|x|^3}  \lf|O\lf(\frac{\rho^8}{|x|^6}\rg)\rg| dx^4= \int_{B_{\rho^b}(0)\setminus B_{\rho^b/2}(0)} \lf|O\lf(\frac{\rho^{12-b}}{|x|^9}\rg)\rg| dx^4  \\
&=O(\rho^{12-6b})=O(\rho^6).
\end{split}
\ee

Finally we have
\be\label{fe10}
\int_{B_{\rho^b}(0)\setminus B_{\rho^a/2}(0)}\eta_{\rho^a}\frac{\eta_{\rho^b}'}{4\la^2|x|^3}\langle P_-(dA^{g_0}), d\ov{x}\wedge dx\rangle dx^4=\int_{B_{\rho^b}(0)\setminus B_{\rho^b/2}(0)}\frac{\eta_{\rho^b}'}{4\la^2|x|^3}\langle P_-(dA^{g_0}), d\ov{x}\wedge dx\rangle dx^4
\ee
Adding up \eqref{fe1}-\eqref{fe10} gives \eqref{fe}, with (recalling that $\frac{2}{3}<b<1<a<\frac{4}{3}$)
$$\delta=\min\{b, 4(a-1), b(1-a)\}.$$
\hfill $\square$

\medskip

\begin{Lma}\label{LmaIV.6} We have
\be
\int_{B_{\rho^b}(0)\setminus B_{\rho^b/2}(0)}\frac{\eta_{\rho^b}'}{4\la^2|x|^3}\langle P_-(dA^{g_0}), d\ov{x}\wedge dx\rangle dx^4= \frac{\pi^2}{2} \rho^4\langle P_- F_A(0), g_0\, d\ov{x}\wedge dx\, g_0^{-1}\rangle +  O(\rho^{4+b}).
\ee
\end{Lma}

\noindent\textbf{Proof of Lemma \ref{LmaIV.6}.}
We compute
\be
\begin{split}
&\int_{B_{\rho^b}(0)\setminus B_{\rho^b/2}(0)}\frac{\eta_{\rho^b}'}{4\la^2|x|^3}\langle P_-(dA^{g_0}), d\ov{x}\wedge dx\rangle dx^4\\
&=\int_{B_{\rho^b}(0)\setminus B_{\rho^b/2}(0)}\frac{\eta_{\rho^b}'}{4\la^2|x|^3}\langle P_-(dA^{g_0}(0)+O(|x|)), d\ov{x}\wedge dx\rangle dx^4\\
&=\int_{\rho^b/2}^{\rho^b} \frac{\eta_{\rho^b}'}{4\la^2} \lf(\int_{S^3}\langle P_-(dA^{g_0}(0)), d\ov{x}\wedge dx\rangle dvol_{S^3} + O(\rho^b)\rg)ds \\
&= \frac{\rho^4}{4}\int_{S^3}\langle P_-(F_{A^{g_0}}(0)), d\ov{x}\wedge dx\rangle dvol_{S^3} + O(\rho^{4+b})
\end{split}
\ee
We now claim that
\be\label{IV.105}
\int_{S^3}\langle P_-(F_{A^{g_0}}(0)), d\ov{x}\wedge dx\rangle dvol_{S^3} =2\pi^2\langle P_- F_A(0), g_0\, d\ov{x}\wedge dx\, g_0^{-1}\rangle.
\ee
Similar to the proof of \eqref{II.30} in Lemma \ref{LmaP-dA},  using the frame $\partial_r, I\partial_r, J\partial_r, K\partial_r,$ at an arbitrary point of $S^3$ as a basis for $T\R^4$, we write
\be
\begin{split}
&F_A(0)=F_A(0)(\partial_r, I\partial_r)dr\wedge Idr+F_A(0)(\partial_r, J\partial_r)dr\wedge Jdr+F_A(0)(\partial_r, K\partial_r)dr\wedge Kdr \\
&+F_A(0)(J\partial_r, K\partial_r ) Jdr\wedge K dr+F_A(0)(K\partial_r, I\partial_r ) Kdr\wedge I dr+ F_A(0)(I\partial_r, J\partial_r ) Idr\wedge J dr,
\end{split}
\ee

hence, using \eqref{II.29-a},
\be
\begin{split}
P_-F_A(0)&=\frac{1}{\sqrt{2}} \lf(F_A(0)(\partial_r, I\partial_r) \omega_{\textbf{i}}^-(x) + F_A(0)(\partial_r, J\partial_r) \omega_{\textbf{j}}^-(x) + F_A(0)(\partial_r, K\partial_r) \omega_{\textbf{k}}^-(x) \rg)\\
&-\frac{1}{\sqrt{2}} \lf(F_A(0)(J\partial_r, K\partial_r) \omega_{\textbf{i}}^-(x) + F_A(0)(K\partial_r, I\partial_r) \omega_{\textbf{j}}^-(x) + F_A(0)(I\partial_r, J\partial_r) \omega_{\textbf{k}}^-(x) \rg).
\end{split}
\ee
Then \eqref{IV.105} follows at once from Lemma \ref{lemmaint1} and Lemma \ref{lemmaint2}.
\hfill$\square$

\medskip

Putting Lemma \ref{LmaIII.2ter} together with Lemma \ref{LmaIV.5} and Lemma \ref{LmaIV.6} yields:

\begin{Lma}\label{LmaIV.7}
For $\tfrac23<b<1<a<\tfrac43$ we have
\be
\label{II.20bis}
\int_{{\R}^4}|F_{\hat{A}}|^2\ dx^4-\int_{{\R}^4}|F_{A}|^2\ dx^4
=8\pi^2 - \pi^2 \rho^4 \langle P_- F_A(0), g_0\, d\ov{x}\wedge dx\, g_0^{-1}\rangle +O(\rho^{4+\delta})\ ,
\ee
as $\rho\to 0$, for some $\delta>0$ depending on $a$ and $b$.
\end{Lma}

\noindent\textbf{Proof of Theorem \ref{th-I.2} (completed).} Taking Lemma \ref{LmaIV.7} into account it now suffices to choose $g_0$ as before in order to obtain \eqref{III.54-gv}.
\hfill$\square$

\appendix
\addcontentsline{toc}{section}{Appendices}
\section*{Appendix}
\section{The self-instanton}
Define
\be
\label{I.4}
SD(x)=\Im m\lf(\frac{x\, d\ov{x}}{1+|x|^2}\rg)
\ee
where the  $x\in {\R}^4$ is identified canonically with the quaternion $x:=x_1+x_2\,{\mathbf i}+x_3\,{\mathbf j}+x_4\,{\mathbf k}$ and $\ov{x}:=x_1-x_2\,{\mathbf i}-x_3\,{\mathbf j}-x_4\,{\mathbf k}$. As a consequence we have the more explicit formula
\be
\label{I.5}
\lf\{
\begin{array}{l}
\ds SD^1(x)=\frac{ x_2\,{\mathbf i}+x_3\,{\mathbf j}+x_4\,{\mathbf k}}{1+|x|^2}\\[5mm]
\ds SD^2(x)=\frac{ -\,x_1\,{\mathbf i}-x_4\,{\mathbf j}+x_3\,{\mathbf k}}{1+|x|^2}\\[5mm]
\ds  SD^3(x)=\frac{ x_4\,{\mathbf i}-x_1\,{\mathbf j}-x_2\,{\mathbf k}}{1+|x|^2}\\[5mm]
\ds  SD^4(x)=\frac{-\, x_3\,{\mathbf i}+x_2\,{\mathbf j}-x_1\,{\mathbf k}}{1+|x|^2}
\end{array}
\rg.
\ee
The curvature of this connection form is given by
\be
\label{I.6}
\begin{array}{l}
\ds F_{SD}(X,Y):=d\,\Im m\lf(\frac{x\, d\ov{x}}{1+|x|^2}\rg)(X,Y) +\lf[ \Im m\lf(\frac{x\, \ov{X}}{1+|x|^2}\rg), \Im m\lf(\frac{x\, \ov{Y}}{1+|x|^2}\rg)  \rg]

\end{array}
\ee
Recall that for any pairs of quaternions $(p,q)$ one has respectively
\be
\label{I.7}
\begin{array}{rl}
\ds [p,q]&\ds=p\,q-q\,p\\[3mm]
 &\ds=(\Re(p)+\Im m(p))\,(\Re(q)+\Im m(q))-(\Re(q)+\Im m(q))\,(\Re(p)+\Im m(p))\\[3mm]
  &\ds=[\Im m(p),\Im m(q)]\ .
\end{array}
\ee
Hence in particular
\be
\label{I.6-b}
\begin{array}{l}
\ds F_{SD}(X,Y)=d\,\Im m\lf(\frac{x\, d\ov{x}}{1+|x|^2}\rg)(X,Y) +\lf[ \frac{x\, \ov{X}}{1+|x|^2}, \frac{x\, \ov{Y}}{1+|x|^2}  \rg]

\end{array}
\ee
We have first
\be
\label{I.8}
\begin{array}{l}
\ds d\,\Im m\lf(\frac{x\, d\ov{x}}{1+|x|^2}\rg)(X,Y)=\Im m\lf(\frac{dx\wedge d\ov{x}}{1+|x|^2}\rg)(X,Y)- \lf(\frac{d|x|^2}{1+|x|^2}\wedge \Im m\lf(\frac{x\, d\ov{x}}{1+|x|^2}\rg)\rg)(X,Y)\ .
\end{array}
\ee
Observe that
\be
\label{I.9}
\begin{array}{rl}
\ds\ov{ dx\wedge d\ov{x}}&\ds=2^{-1}\,\sum_{l,m=1}^4\ov{ \p_{x_l}x\,\p_{x_m}\ov{x}- \p_{x_m}x\,\p_{x_l}\ov{x}}\ dx_l\wedge dx_m\\[5mm]
 &\ds=2^{-1}\,\sum_{l,m=1}^4 \p_{x_m}x\,\p_{x_l}\ov{x}- \p_{x_l}x\,\p_{x_m}\ov{x}\ dx_l\wedge dx_m\\[5mm]
  &\ds=2^{-1}\,\sum_{l,m=1}^4 \p_{x_l}x\,\p_{x_m}\ov{x}- \p_{x_m}x\,\p_{x_l}\ov{x}\ dx_m\wedge dx_l\\[5mm]
  &\ds= -\,dx\wedge d\ov{x}\ .
\end{array}
\ee
Hence
\be
\label{I.10}
\ds d\,\Im m\lf(\frac{x\, d\ov{x}}{1+|x|^2}\rg)(X,Y)=\frac{dx\wedge d\ov{x}}{1+|x|^2}(X,Y)- \lf(\frac{d|x|^2}{1+|x|^2}\wedge \Im m\lf(\frac{x\, d\ov{x}}{1+|x|^2}\rg)\rg)(X,Y)\ .
\ee
Observe that
\be
\label{I.11}
\Re e\lf(\frac{x\,d\ov{x}}{1+|x|^2}\rg)=\frac{1}{2}\,\frac{x\,d\ov{x}+dx\,\ov{x}}{1+|x|^2}=\frac{1}{2}\,\frac{d|x|^2}{1+|x|^2}\ ,
\ee
where we have used the identity $x\ov{x}=|x|^2$. Combining (\ref{I.10}) and (\ref{I.11}) gives then
\be
\label{I.12}
\begin{array}{l}
\ds d\,\Im m\lf(\frac{x\, d\ov{x}}{1+|x|^2}\rg)(X,Y)\ds=\frac{dx\wedge d\ov{x}}{1+|x|^2}(X,Y)- \lf(\frac{d|x|^2}{1+|x|^2}\wedge \frac{x\,d\ov{x}}{1+|x|^2}\rg)(X,Y)\\[5mm]
\ds=\frac{dx\wedge d\ov{x}}{1+|x|^2}(X,Y)- \lf(\frac{x\,d\ov{x}}{1+|x|^2}\wedge \frac{x\,d\ov{x}}{1+|x|^2}\rg)(X,Y)- \lf(\frac{dx\,\ov{x}}{1+|x|^2}\wedge \frac{x\,d\ov{x}}{1+|x|^2}\rg)(X,Y)\\[5mm]
\ds=\frac{dx\wedge d\ov{x}}{1+|x|^2}(X,Y)-\frac{x\,\ov{X}\,x\,\ov{Y}-x\,\ov{Y}\,x\,\ov{X}}{(1+|x|^2)^2}-|x|^2\,\frac{dx\wedge d\ov{x}}{(1+|x|^2)^2}(X,Y)\ .
\end{array}
\ee
Combining (\ref{I.6-b}) and (\ref{I.12}) is implying
\be
\label{I.13}
 F_{SD}(X,Y)=\frac{dx\wedge d\ov{x}}{(1+|x|^2)^2}(X,Y)
\ee

\section{Proofs of Lemmas \ref{Lemmaframe}, \ref{lemmaint1} and \ref{lemmaint2}}
\reset

\begin{Lma}\label{CorPol} We have
\[
\begin{split}
\int_{S^3}x_i^2\, d\sigma&=\frac{\pi^2}{2}\quad \text{for }1\le i\le 4\,,\\
\int_{S^3}x_i^4\, d\sigma&=\frac{\pi^2}{4}\quad \text{for }1\le i\le 4\,,\\
\int_{S^3}x_i^2x_j^2\, d\sigma&=\frac{\pi^2}{12}\quad \text{for }1\le i<j\le 4\,.
\end{split}
\]
\end{Lma}
\noindent{\bf Proof of Lemma~\ref{CorPol}.}
We recall the following well-known formula, see e.g. \cite{Fol}:
Given a monomial $p(x_1,x_2,x_3,x_4)= x_1^{\alpha_1}x_2^{\alpha_2}x_3^{\alpha_3} x_4^{\alpha_4}$ with $\alpha_1,\dots,\alpha_4\in 2\mathbb{N}$, we have
\be\label{intpoly}
\int_{S^3} p\, d\sigma=\frac{2\Gamma(\beta_1)\Gamma(\beta_2)\Gamma(\beta_3)\Gamma(\beta_4)}{\Gamma(\beta_1+\beta_2+\beta_3+\beta_4)},
\ee
where $\beta_l=\frac12(\alpha_l+1)$ for $l=1,\dots,4$, and $\Gamma$ denotes the usual $\Gamma$ function.

The lemma follows \eqref{intpoly}, using that $\Gamma(\tfrac12)=\sqrt{\pi}$, $\Gamma(\tfrac{3}{2})=\frac{1}{2}\sqrt{\pi}$, $\Gamma(\tfrac52)=\frac{3}{4}\sqrt{\pi}$, and $\Gamma(k)=(k-1)!$ for $k$ positive integer.
 \hfill $\Box$

\noindent{\bf Proof of Lemma~\ref{Lemmaframe}.}
We compute
\be
\label{II.31-a}
\begin{array}{l}
\ds r^2\, dr\wedge Idr=(x_1\,dx_1+x_2\,dx_2+x_3\,dx_3+x_4\,dx_4)\wedge(x_1\, d{x_2}-x_2\,d{x_1}+x_3\,d{x_4}-x_4\,d{x_3})\\[5mm]
\ds\quad= (x_1^2+x_2^2)\, dx_1\wedge dx_2+(x_3^2+x_4^2)\, dx_3\wedge dx_4 +x_1\,x_3\, (dx_1\wedge dx_4+dx_3\wedge dx_2)\\[5mm]
\ds\quad+ x_1\, x_4\, (dx_3\wedge dx_1+dx_4\wedge dx_2)+ x_2\, x_3\, (dx_2\wedge dx_4+dx_1\wedge dx_3)\\[5mm]
\ds\quad+x_2\, x_4 \,(dx_3\wedge dx_2+dx_1\wedge dx_4)= (x_1^2+x_2^2)\, dx_1\wedge dx_2+(x_3^2+x_4^2)\, dx_3\wedge dx_4\\[5mm]
\ds\quad+   (x_1\,x_3+x_2\,x_4)\, \sqrt{2}\, \om_{\mathbf k}^-+(x_2\,x_3- x_1\, x_4)\, \sqrt{2}\, \om_{\mathbf j}^-\\[5mm]
\ds\quad= (x_1^2+x_2^2)\, \sqrt{2}^{-1}(\om_{\mathbf i}^++\om_{\mathbf i}^-)+(x_3^2+x_4^2)\, \sqrt{2}^{-1}(\om_{\mathbf i}^+-\om_{\mathbf i}^-)\\[5mm]
\ds\quad+   (x_1\,x_3+x_2\,x_4)\, \sqrt{2}\, \om_{\mathbf k}^-+(x_2\,x_3- x_1\, x_4)\, \sqrt{2}\, \om_{\mathbf j}^-\\[5mm]
\ds\quad= \frac{r^2}{\sqrt{2}}\,\om_{\mathbf i}^++\frac{(x_1^2+x_2^2-x_3^2-x_4^2)}{\sqrt{2}}\,\om_{\mathbf i}^-+   (x_1\,x_3+x_2\,x_4)\, \sqrt{2}\, \om_{\mathbf k}^-+(x_2\,x_3- x_1\, x_4)\, \sqrt{2}\, \om_{\mathbf j}^-\ .
\end{array}
\ee
Hence
\be
\label{II.31-b}
P_+( r^2\, dr\wedge Idr)= 2^{-1}\,|x|^2\, ( dx_1\wedge dx_2+dx_3\wedge dx_4)=\frac{r^2}{\sqrt{2}}\,\om_{\mathbf i}^+\ .
\ee
This implies
\be
\label{II.31-c}
r^2\, dr\wedge Idr=\sqrt{2}^{-1}\,r^2\,\om_{\mathbf i}^++P_-( r^2\, dr\wedge Idr)\ .
\ee
Similarly
\be
\label{II.31-d}
\begin{array}{l}
\ds r^2\, dr\wedge Jdr=(x_1\,dx_1+x_2\,dx_2+x_3\,dx_3+x_4\,dx_4)\wedge(x_1\, d{x_3}-x_3\,d{x_1}+x_4\,d{x_2}-x_2\,d{x_4})\\[5mm]
\ds\quad=(x_1^2+x_3^2)  \,dx_1\wedge dx_3+(x_2^2+x_4^2)\,   dx_4\wedge dx_2+x_1\,x_4\, (dx_1\wedge dx_2+dx_4\wedge dx_3)\\[5mm]
\ds\quad+x_1\,x_2\, (dx_4\wedge dx_1+dx_2\wedge dx_3)+ x_2\,x_3\, (dx_1\wedge dx_2+dx_4\wedge dx_3)\\[5mm]
\ds\quad+x_3\,x_4\, (dx_3\wedge dx_2+dx_1\wedge dx_4)=(x_1^2+x_3^2)  \,dx_1\wedge dx_3+(x_2^2+x_4^2)\,   dx_4\wedge dx_2\\[5mm]
\ds\quad+(x_1\, x_4+x_2\,x_3)\,\sqrt{2}\, \om_{\mathbf i}^-+(x_3\,x_4-x_1\,x_2)\, \,\sqrt{2}\, \om_{\mathbf k}^-\\[5mm]
\ds\quad= \frac{r^2}{\sqrt{2}}\,\om_{\mathbf j}^++\frac{(x_1^2+x_3^2-x_4^2-x_2^2)}{\sqrt{2}}\,\om_{\mathbf j}^-+   (x_1\,x_4+x_2\,x_3)\, \sqrt{2}\, \om_{\mathbf i}^-+(x_3\,x_4- x_1\, x_2)\, \sqrt{2}\, \om_{\mathbf k}^-\ .

\end{array}
\ee
Hence
\be
\label{II.31-e}
P_+( r^2\, dr\wedge Jdr)= 2^{-1}\, |x|^2\,( dx_1\wedge dx_3+dx_4\wedge dx_2)=\frac{r^2}{\sqrt{2}}\,\om_{\mathbf j}^+\ .
\ee
This implies
\be
\label{II.31-f}
r^2\, dr\wedge Jdr=\sqrt{2}^{-1}\,r^2\,\om_{\mathbf j}^++P_-( r^2\, dr\wedge Jdr)\ .
\ee
We have also
\be
\label{II.31-g}
\begin{array}{l}
\ds r^2\, dr\wedge Kdr=(x_1\,dx_1+x_2\,dx_2+x_3\,dx_3+x_4\,dx_4)\wedge(x_1\, d{x_4}-x_4\,d{x_1}+x_2\,d{x_3}-x_3\,d{x_2})\\[5mm]
\ds\quad=(x_1^2+x_4^2)\, dx_1\wedge dx_4+(x_2^2+x_3^2)\,dx_2\wedge dx_3+x_1\, x_3 (dx_2\wedge dx_1+dx_3\wedge dx_4)\\[5mm]
\ds\quad+ x_1\,x_2\,(dx_1\wedge dx_3 +dx_2\wedge dx_4)+x_2\, x_4\,(dx_1\wedge dx_2+dx_4\wedge dx_3)\\[5mm]
\ds\quad+x_3\,x_4\,(dx_1\wedge dx_3+dx_2\wedge dx_4)=(x_1^2+x_4^2)\, dx_1\wedge dx_4+(x_2^2+x_3^2)\,dx_2\wedge dx_3\\[5mm]
\ds\quad+(  x_2\,x_4-x_1\,x_3 )\sqrt{2}\, \om_{\mathbf i}^-+\,( x_1\,x_2+ x_3\,x_4 )\,\sqrt{2}\, \om_{\mathbf j}^-\\[5mm]
\ds\quad= \frac{r^2}{\sqrt{2}}\,\om_{\mathbf k}^++\frac{(x_1^2+x_4^2-x_2^2-x_3^2)}{\sqrt{2}}\,\om_{\mathbf k}^-+   (x_2\,x_4-x_1\,x_3)\, \sqrt{2}\, \om_{\mathbf i}^-+(x_1\,x_2+x_3\, x_4)\, \sqrt{2}\, \om_{\mathbf j}^-\ .
\end{array}
\ee
Hence
\be
\label{II.31-h}
P_+( r^2\, dr\wedge Kdr)= 2^{-1}\,|x|^2 ( dx_1\wedge dx_4+dx_2\wedge dx_3)=\sqrt{2}^{-1}\,r^2\,\om_{\mathbf k}^+\ .
\ee
This implies
\be
\label{II.31-i}
r^2\, dr\wedge Kdr=\sqrt{2}^{-1}\,r^2\,\om_{\mathbf k}^++P_-( r^2\, dr\wedge Kdr)\ .
\ee
From (\ref{II.31-c}), (\ref{II.31-f}) and (\ref{II.31-i}) we deduce that
\be
\label{II.31-100}
\begin{array}{l}
\ds1=| dr\wedge Idr|^2=2^{-1}+|P_-(dr\wedge Idr)|^2\ ,\\[5mm]
\ds 1=| dr\wedge Jdr|^2=2^{-1}+|P_-(dr\wedge Jdr)|^2\ ,\\[5mm]
  \ds1=| dr\wedge Kdr|^2=2^{-1}+|P_-(dr\wedge Kdr)|^2\ .
\end{array}
\ee
Moreover
\be
\label{II.31-101}
\begin{array}{l}
\ds 0=\langle dr\wedge Idr,dr\wedge Jdr\rangle=\langle P_-(dr\wedge Idr),P_-(dr\wedge Jdr)\rangle\ ,\\[5mm]
\ds 0=\langle dr\wedge Idr,dr\wedge Kdr\rangle=\langle P_-(dr\wedge Idr),P_-(dr\wedge Kdr)\rangle\ ,\\[5mm]
\ds 0=\langle dr\wedge Jdr,dr\wedge Kd\rangle=\langle P_-(dr\wedge Jdr),P_-(dr\wedge Kdr)\rangle\ .
\end{array}
\ee
This implies that $\{\omega_\mathbf{i}^-(x),\, \omega_\mathbf{j}^-(x),\, \omega_\mathbf{k}^-(x)\}$ is a orthonomal basis.
\hfill $\Box$

\noindent{\bf Proof of Lemma~\ref{lemmaint1}.}
We have respectively
\be
\label{rep-1}
\begin{array}{l}
 \ds\int_{S^3}\,\lf<g_0\,d\ov{x}\wedge dx\, g_0^{-1},F_A(0)(\p_r,I\p_r)\ \om^-_{\mathbf i}(x)\rg>\ dvol_{S^3}\\[5mm]
 \ds=\int_{S^3}\,2\,\sqrt{2}\ \om_{\mathbf i}^-\cdot\om^-_{\mathbf i}(x)\, \lf<{\mathbf i}_{g_0},F_A(0)(\p_r,I\p_r)\rg>\ dvol_{S^3}\\[5mm]
 \ds+\int_{S^3}\,2\,\sqrt{2}\ \om_{\mathbf j}^-\cdot\om^-_{\mathbf i}(x)\, \lf<{\mathbf j}_{g_0},F_A(0)(\p_r,I\p_r)\rg>\ dvol_{S^3}\\[5mm]
\ds+ \int_{S^3}\,2\,\sqrt{2}\ \om_{\mathbf k}^-\cdot\om^-_{\mathbf i}(x)\, \lf<{\mathbf k}_{g_0},F_A(0)(\p_r,I\p_r)\rg>\ dvol_{S^3}\\[5mm]
 \end{array}
\ee
We recall from (\ref{II.31-a})
\be
\label{rep-2}
\lf\{
\begin{array}{l}
\ds  \om_{\mathbf i}^-\cdot\om^-_{\mathbf i}(x)= x_1^2+x_2^2-x_3^2-x_4^2  \\[5mm]
 \ds  \om_{\mathbf j}^-\cdot\om^-_{\mathbf i}(x)= 2\ (x_2\,x_3-x_1\,x_4)  \\[5mm]
\ds  \om_{\mathbf k}^-\cdot\om^-_{\mathbf i}(x)= 2\, (x_1\,x_3+x_2\,x_4)
\end{array}
\rg.
\ee
Recall that on $S^3$ we have
\be
\label{rep-3}
F_A(0)(\p_r,I\p_r)=F_A(0)(x_1\, \p_{x_1}+x_2\, \p_{x_2}+x_3\, \p_{x_3}+x_4\, \p_{x_4}, x_1\, \p_{x_2}-x_2\,\p_{x_1}+x_3\,\p_{x_4}-x_4\,\p_{x_3})
\ee
Hence we have successively using corollary~\ref{CorPol}
\be
\label{rep-4}
\begin{array}{l}
\ds \int_{S^3}\,2\,\sqrt{2}\ \om_{\mathbf i}^-\cdot\om^-_{\mathbf i}(x)\, \lf<{\mathbf i}_{g_0},F_A(0)(\p_r,I\p_r)\rg>\ dvol_{S^3}\\[5mm]
 \ds=\,2\,\sqrt{2}\ \int_{S^3}(x_1^2+x_2^2-x_3^2-x_4^2)\ (x_1^2+x_2^2)\,\ dvol_{S^3}\ \lf< {\mathbf i}_{g_0},F_A^{12}(0)\rg>\\[5mm]
\ds+\,2\,\sqrt{2}\ \int_{S^3}(x_1^2+x_2^2-x_3^2-x_4^2)\ (x_3^2+x_4^2)\,\ dvol_{S^3}\ \lf< {\mathbf i}_{g_0},F_A^{34}(0)\rg>\\[5mm]
\ds=\,2\,\sqrt{2}\ \int_{S^3}(x_1^2+x_2^2-x_3^2-x_4^2)\ (x_1^2+x_2^2)\,\ dvol_{S^3}\ \lf< {\mathbf i}_{g_0},F_A^{12}(0)-F^{34}_A(0)\rg>\\[5mm]
\ds=\,\frac{2\,\sqrt{2}}{3}\ \pi^2\ \lf< {\mathbf i}_{g_0},F_A^{12}(0)+F^{43}_A(0)\rg>
\end{array}
\ee
we have
\be
\label{rep-5}
\begin{array}{l}
\ds \int_{S^3}\,2\,\sqrt{2}\ \om_{\mathbf j}^-\cdot\om^-_{\mathbf i}(x)\, \lf<{\mathbf j}_{g_0},F_A(0)(\p_r,I\p_r)\rg>\ dvol_{S^3}\\[5mm]
\ds=\,4\,\sqrt{2}\ \int_{S^3} x_2^2\,x_3^2\  \ dvol_{S^3}\ \lf< {\mathbf j}_{g_0},F_A^{24}(0)+F_A^{13}(0)\rg>\\[5mm]
\ds+4\,\sqrt{2}\ \int_{S^3} x_1^2\,x_4^2\  \ dvol_{S^3}\ \lf< {\mathbf j}_{g_0},F_A^{24}(0)+F_A^{13}(0)\rg>\\[5mm]
\ds=\,8\,\sqrt{2}\ \int_{S^3} x_2^2\,x_3^2\  \ dvol_{S^3}\ \lf< {\mathbf j}_{g_0},F_A^{24}(0)+F_A^{13}(0)\rg>\\[5mm]
\ds=\,\frac{2\,\sqrt{2}}{3}\ \pi^2\ \lf< {\mathbf j}_{g_0},F_A^{13}(0)+F_A^{24}(0)\rg>
\end{array}
\ee
and
\be
\label{rep-6}
\begin{array}{l}
\ds \int_{S^3}\,2\,\sqrt{2}\ \om_{\mathbf k}^-\cdot\om^-_{\mathbf i}(x)\, \lf<{\mathbf k}_{g_0},F_A(0)(\p_r,I\p_r)\rg>\ dvol_{S^3}\\[5mm]
\ds=4\,\sqrt{2}\ \int_{S^3} x_1^2\,x_3^2\  \ dvol_{S^3}\ \lf< {\mathbf k}_{g_0},F_A^{14}(0)+F_A^{32}(0)\rg>\\[5mm]
\ds+4\,\sqrt{2}\ \int_{S^3} x_2^2\,x_4^2\  \ dvol_{S^3}\ \lf< {\mathbf k}_{g_0},F_A^{14}(0)+F_A^{32}(0)\rg>\\[5mm]
\ds=8\,\sqrt{2}\ \int_{S^3} x_1^2\,x_3^2\  \ dvol_{S^3}\ \lf< {\mathbf k}_{g_0},F_A^{14}(0)+F_A^{32}(0)\rg>\\[5mm]
\ds=\,\frac{2\,\sqrt{2}}{3}\ \pi^2\ \lf< {\mathbf k}_{g_0},F_A^{14}(0)+F_A^{32}(0)\rg>
\end{array}
\ee
We deduce
\be
\label{rep-6-a}
\begin{array}{l}
 \ds\int_{S^3}\,\lf<g_0\,d\ov{x}\wedge dx\, g_0^{-1},F_A(0)(\p_r,I\p_r)\ \om^-_{\mathbf i}(x)\rg>\ dvol_{S^3}\\[5mm]
 \ds=\frac{2\,\sqrt{2}}{3}\ \pi^2\ \lf< {\mathbf i}_{g_0},F_A^{12}(0)+F^{43}_A(0)\rg>+\frac{2\,\sqrt{2}}{3}\ \pi^2\ \lf< {\mathbf j}_{g_0},F_A^{13}(0)+F_A^{24}(0)\rg>\\[5mm]
 \ds+\frac{2\,\sqrt{2}}{3}\ \pi^2\ \lf< {\mathbf k}_{g_0},F_A^{14}(0)+F_A^{32}(0)\rg>
 \end{array}
\ee
We recall
\be
\label{p-fa0}
\begin{array}{rl}
P_-F_A(0)&\ds= 2^{-1}\, \lf(F_A^{12}(0)+F_A^{43}(0)\rg) \lf(dx_1\wedge dx_2-dx_3\wedge dx_4\rg)\\[5mm]
 &\ds\ + 2^{-1}\, \lf(F_A^{13}(0)+F_A^{24}(0)\rg) \lf(dx_1\wedge dx_3-dx_4\wedge dx_2\rg)\\[5mm]
 &\ds\ +2^{-1}\, \lf(F_A^{14}(0)+F_A^{32}(0)\rg) \lf(dx_1\wedge dx_4-dx_2\wedge dx_3\rg)\\[5mm]
 &\ds=\lf[  \frac{\lf(F_A^{12}(0)+F_A^{43}(0)\rg)}{\sqrt{2}} \ \om_{\mathbf i}^-+  \frac{\lf(F_A^{13}(0)+F_A^{24}(0)\rg)}{\sqrt{2}}  \ \om_{\mathbf j}^-+ \frac{\lf(F_A^{14}(0)+F_A^{32}(0)\rg)}{\sqrt{2}}  \ \om_{\mathbf k}^- \rg]
 \end{array}
\ee
Thus 
\be
\label{f-1}
\begin{array}{l}
\ds\lf<g_0\,d\ov{x}\wedge dx\, g_0^{-1},P_-F_A(0)\rg>= 2\,\sqrt{2}\,\lf<{\mathbf i}_{g_0}\,\om_{\mathbf i}^-+{\mathbf j}_{g_0}\,\om_{\mathbf j}^-+{\mathbf k}_{g_0}\,\om_{\mathbf k}^-, \rg.\\[5mm]
\ds\lf.  \lf[  \frac{\lf(F_A^{12}(0)+F_A^{43}(0)\rg)}{\sqrt{2}} \ \om_{\mathbf i}^-+  \frac{\lf(F_A^{13}(0)+F_A^{24}(0)\rg)}{\sqrt{2}}  \ \om_{\mathbf j}^-+ \frac{\lf(F_A^{14}(0)+F_A^{32}(0)\rg)}{\sqrt{2}}  \ \om_{\mathbf k}^- \rg]  \rg>\\[5mm]
\ds = 2\,\lf<{\mathbf i}_{g_0},F_A^{12}(0)+F_A^{43}(0)\rg>+ 2\,\lf<{\mathbf j}_{g_0},F_A^{13}(0)+F_A^{24}(0)\rg>+2\, \lf<  {\mathbf k}_{g_0},F_A^{14}(0)+F_A^{32}(0)\rg>
\end{array}
\ee
and finally
\be
\label{rep-6-b}
\begin{array}{l}
 \ds\int_{S^3}\,\lf<g_0\,d\ov{x}\wedge dx\, g_0^{-1},F_A(0)(\p_r,I\p_r)\ \om^-_{\mathbf i}(x)\rg>\ dvol_{S^3}
\ds= \frac{\sqrt{2}}{3}\, \pi^2\,\lf<g_0\,d\ov{x}\wedge dx\, g_0^{-1}, P_-F_A(0)\rg>\ .
\end{array}
\ee
We have then respectively
\be
\label{rep-7}
\begin{array}{l}
 \ds\int_{S^3}\,\lf<g_0\,d\ov{x}\wedge dx\, g_0^{-1},F_A(0)(\p_r,J\p_r)\ \om^-_{\mathbf j}(x)\rg>\ dvol_{S^3}\\[5mm]
 \ds=\int_{S^3}\,2\,\sqrt{2}\ \om_{\mathbf i}^-\cdot\om^-_{\mathbf j}(x)\, \lf<{\mathbf i}_{g_0},F_A(0)(\p_r,J\p_r)\rg>\ dvol_{S^3}\\[5mm]
 \ds+\int_{S^3}\,2\,\sqrt{2}\ \om_{\mathbf j}^-\cdot\om^-_{\mathbf j}(x)\, \lf<{\mathbf j}_{g_0},F_A(0)(\p_r,J\p_r)\rg>\ dvol_{S^3}\\[5mm]
\ds+ \int_{S^3}\,2\,\sqrt{2}\ \om_{\mathbf k}^-\cdot\om^-_{\mathbf j}(x)\, \lf<{\mathbf k}_{g_0},F_A(0)(\p_r,J\p_r)\rg>\ dvol_{S^3}\\[5mm]
 \end{array}
\ee
We recall from (\ref{II.31-a})
\be
\label{rep-8}
\lf\{
\begin{array}{l}
\ds  \om_{\mathbf i}^-\cdot\om^-_{\mathbf j}(x)= 2\,(x_1\,x_4+x_2\,x_3)  \\[5mm]
 \ds  \om_{\mathbf j}^-\cdot\om^-_{\mathbf j}(x)=  x_1^2+x_3^2-x_2^2-x_4^2 \\[5mm]
\ds  \om_{\mathbf k}^-\cdot\om^-_{\mathbf j}(x)= 2\, (x_3\,x_4-x_1\,x_2)
\end{array}
\rg.
\ee
Recall that on $S^3$ we have
\be
\label{rep-9}
F_A(0)(\p_r,J\p_r)=F_A(0)(x_1\, \p_{x_1}+x_2\, \p_{x_2}+x_3\, \p_{x_3}+x_4\, \p_{x_4},x_1\, \p_{x_3}-x_3\,\p_{x_1}+x_4\,\p_{x_2}-x_2\,\p_{x_4} )
\ee
Hence we have successively
\be
\label{rep-10}
\begin{array}{l}
\ds \int_{S^3}\,2\,\sqrt{2}\ \om_{\mathbf i}^-\cdot\om^-_{\mathbf j}(x)\, \lf<{\mathbf i}_{g_0},F_A(0)(\p_r,J\p_r)\rg>\ dvol_{S^3}\\[5mm]
 \ds=\,4\,\sqrt{2}\ \int_{S^3}\ x_1^2\,x_4^2\,  dvol_{S^3}\ < {\mathbf i}_{g_0},F_A^{12}(0)+F^{43}_A(0)>\\[5mm]
\ds+\,4\,\sqrt{2}\ \int_{S^3}\ x_2^2\,x_3^2\ dvol_{S^3}\ < {\mathbf i}_{g_0},F_A^{12}(0)+F_A^{43}(0)>\\[5mm]
 \ds=\,8\,\sqrt{2}\ \int_{S^3}\ x_1^2\,x_4^2\,  dvol_{S^3}\ < {\mathbf i}_{g_0},F_A^{12}(0)+F^{43}_A(0)>\\[5mm]
 \ds=  \,\frac{2\,\sqrt{2}}{3}\ \pi^2\ \ < {\mathbf i}_{g_0},F_A^{12}(0)+F^{43}_A(0)>
\end{array}
\ee
we have
{\be
\label{rep-11}
\begin{array}{l}
\ds \int_{S^3}\,2\,\sqrt{2}\ \om_{\mathbf j}^-\cdot\om^-_{\mathbf j}(x)\, \lf<{\mathbf j}_{g_0},F_A(0)(\p_r,J\p_r)\rg>\ dvol_{S^3}\\[5mm]
\ds=\,2\,\sqrt{2}\ \int_{S^3} \ (x_1^2+x_3^2-x_2^2-x_4^2) (x_1^2+x_3^2) \ dvol_{S^3}\ < {\mathbf j}_{g_0},F_A^{13}(0)>\\[5mm]
\ds+2\,\sqrt{2}\ \int_{S^3}   \  (x_1^2+x_3^2-x_2^2-x_4^2) (x_2^2+x_4^2) \ dvol_{S^3}\ < {\mathbf j}_{g_0},F_A^{42}(0)>\\[5mm]
\ds=\,2\,\sqrt{2}\ \int_{S^3}   (x_1^2+x_3^2-x_2^2-x_4^2) (x_1^2+x_3^2) \ dvol_{S^3}\ < {\mathbf j}_{g_0},F_A^{13}(0)-F_A^{42}(0)>\\[5mm]
\ds=\,\frac{2\,\sqrt{2}}{3}\ \pi^2\ \ < {\mathbf j}_{g_0},F_A^{13}(0)+F_A^{24}(0)>
\end{array}
\ee}
and
\be
\label{rep-12}
\begin{array}{l}
\ds \int_{S^3}\,2\,\sqrt{2}\ \om_{\mathbf k}^-\cdot\om^-_{\mathbf j}(x)\, \lf<{\mathbf k}_{g_0},F_A(0)(\p_r,J\p_r)\rg>\ dvol_{S^3}\\[5mm]
\ds=4\,\sqrt{2}\ \int_{S^3} x_3^2\,x_4^2\  \ dvol_{S^3}\ < {\mathbf k}_{g_0},F_A^{14}(0)+F_A^{32}(0)>\\[5mm]
\ds+4\,\sqrt{2}\ \int_{S^3} x_1^2\,x_2^2\  \ dvol_{S^3}\ < {\mathbf k}_{g_0},F_A^{14}(0)+F_A^{32}(0)>\\[5mm]
\ds=\,\frac{2\,\sqrt{2}}{3}\ \pi^2\ \  < {\mathbf k}_{g_0},F_A^{14}(0)+F_A^{32}(0)>\
\end{array}
\ee
Hence we have also
\be
\label{rep-6-c}
\begin{array}{l}
 \ds\int_{S^3}\,\lf<g_0\,d\ov{x}\wedge dx\, g_0^{-1},F_A(0)(\p_r,J\p_r)\ \om^-_{\mathbf j}(x)\rg>\ dvol_{S^3}
\ds= \frac{\sqrt{2}}{3}\, \pi^2\,\lf<g_0\,d\ov{x}\wedge dx\, g_0^{-1}, P_-F_A(0)\rg>\ .
\end{array}
\ee
We have 
\be
\label{rep-13}
\begin{array}{l}
 \ds\int_{S^3}\,\lf<g_0\,d\ov{x}\wedge dx\, g_0^{-1},F_A(0)(\p_r,K\p_r)\ \om^-_{\mathbf k}(x)\rg>\ dvol_{S^3}\\[5mm]
 \ds=\int_{S^3}\,2\,\sqrt{2}\ \om_{\mathbf i}^-\cdot\om^-_{\mathbf k}(x)\, \lf<{\mathbf i}_{g_0},F_A(0)(\p_r,K\p_r)\rg>\ dvol_{S^3}\\[5mm]
 \ds+\int_{S^3}\,2\,\sqrt{2}\ \om_{\mathbf j}^-\cdot\om^-_{\mathbf k}(x)\, \lf<{\mathbf j}_{g_0},F_A(0)(\p_r,K\p_r)\rg>\ dvol_{S^3}\\[5mm]
\ds+ \int_{S^3}\,2\,\sqrt{2}\ \om_{\mathbf k}^-\cdot\om^-_{\mathbf k}(x)\, \lf<{\mathbf k}_{g_0},F_A(0)(\p_r,K\p_r)\rg>\ dvol_{S^3}\\[5mm]
 \end{array}
\ee
We recall from (\ref{II.31-a})
\be
\label{rep-14}
\lf\{
\begin{array}{l}
\ds  \om_{\mathbf i}^-\cdot\om^-_{\mathbf k}(x)= 2\,(x_2\,x_4-x_1\,x_3)  \\[5mm]
 \ds  \om_{\mathbf j}^-\cdot\om^-_{\mathbf k}(x)=  2\ (x_1\,x_2+x_3\,x_4) \\[5mm]
\ds  \om_{\mathbf k}^-\cdot\om^-_{\mathbf k}(x)= x_1^2+x_4^2- x_2^2-x_3^2
\end{array}
\rg.
\ee
Recall that on $S^3$ we have
\be
\label{rep-15}
F_A(0)(\p_r,K\p_r)=F_A(0)(x_1\, \p_{x_1}+x_2\, \p_{x_2}+x_3\, \p_{x_3}+x_4\, \p_{x_4},x_1\, \p_{x_4}-x_4\,\p_{x_1}+x_2\,\p_{x_3}-x_3\,\p_{x_2})\ .
\ee
Hence we have successively
\be
\label{rep-16}
\begin{array}{l}
\ds \int_{S^3}\,2\,\sqrt{2}\ \om_{\mathbf i}^-\cdot\om^-_{\mathbf k}(x)\, \lf<{\mathbf i}_{g_0},F_A(0)(\p_r,K\p_r)\rg>\ dvol_{S^3}\\[5mm]
 \ds=\,4\,\sqrt{2}\ \int_{S^3}\ x_2^2\,x_4^2\,  dvol_{S^3}\ < {\mathbf i}_{g_0},F_A^{12}(0)+F^{43}_A(0)>\\[5mm]
\ds+\,4\,\sqrt{2}\ \int_{S^3}\ x_1^2\,x_3^2\ dvol_{S^3}\ < {\mathbf i}_{g_0},F_A^{12}(0)+F_A^{43}(0)>\\[5mm]
 \ds=\,8\,\sqrt{2}\ \int_{S^3}\ x_2^2\,x_4^2\,  dvol_{S^3}\ < {\mathbf i}_{g_0},F_A^{12}(0)+F^{43}_A(0)>\\[5mm]
 \ds=  \,\frac{2\,\sqrt{2}}{3}\ \pi^2\ \ < {\mathbf i}_{g_0},F_A^{12}(0)+F^{43}_A(0)>
\end{array}
\ee
we have
{\be
\label{rep-17}
\begin{array}{l}
\ds \int_{S^3}\,2\,\sqrt{2}\ \om_{\mathbf j}^-\cdot\om^-_{\mathbf k}(x)\, \lf<{\mathbf j}_{g_0},F_A(0)(\p_r,K\p_r)\rg>\ dvol_{S^3}\\[5mm]
\ds=\,4\,\sqrt{2}\ \int_{S^3} \ x_1^2\,x_2^2  \ dvol_{S^3}\ < {\mathbf j}_{g_0},F_A^{13}(0)+F_A^{24}(0)>\\[5mm]
\ds+4\,\sqrt{2}\ \int_{S^3}   \  x_3^2\,x_4^2  \ dvol_{S^3}\ < {\mathbf j}_{g_0},F_A^{13}(0)+F_A^{24}(0)>\\[5mm]
\ds=\,8\,\sqrt{2}\ \int_{S^3}    \ dvol_{S^3}\ < {\mathbf j}_{g_0},F_A^{13}(0)+F_A^{24}(0)>\\[5mm]
\ds=\,\frac{2\,\sqrt{2}}{3}\ \pi^2\ \ < {\mathbf j}_{g_0},F_A^{13}(0)+F_A^{24}(0)>
\end{array}
\ee}
and
{\be
\label{rep-18}
\begin{array}{l}
\ds \int_{S^3}\,2\,\sqrt{2}\ \om_{\mathbf k}^-\cdot\om^-_{\mathbf k}(x)\, \lf<{\mathbf k}_{g_0},F_A(0)(\p_r,K\p_r)\rg>\ dvol_{S^3}\\[5mm]
\ds=\,2\,\sqrt{2}\ \int_{S^3} \ (x_1^2+x_4^2-x_2^2-x_3^2) (x_1^2+x_4^2) \ dvol_{S^3}\ < {\mathbf k}_{g_0},F_A^{14}(0)>\\[5mm]
\ds+2\,\sqrt{2}\ \int_{S^3}   \  (x_1^2+x_4^2-x_2^2-x_3^2) (x_2^2+x_3^2) \ dvol_{S^3}\ < {\mathbf k}_{g_0},F_A^{23}(0)>\\[5mm]
\ds=\,2\,\sqrt{2}\ \int_{S^3}   (x_1^2+x_4^2-x_2^2-x_3^2) (x_1^2+x_4^2) \ dvol_{S^3}\ < {\mathbf k}_{g_0},F_A^{14}(0)-F_A^{23}(0)>\\[5mm]
\ds=\,\frac{2\,\sqrt{2}}{3}\ \pi^2\ \ < {\mathbf k}_{g_0},F_A^{14}(0)+F_A^{32}(0)>
\end{array}
\ee
Hence we have also
\be
\label{rep-6-d}
\begin{array}{l}
 \ds\int_{S^3}\,\lf<g_0\,d\ov{x}\wedge dx\, g_0^{-1},F_A(0)(\p_r,K\p_r)\ \om^-_{\mathbf k}(x)\rg>\ dvol_{S^3}
\ds= \frac{\sqrt{2}}{3}\, \pi^2\,\lf<g_0\,d\ov{x}\wedge dx\, g_0^{-1}, P_-F_A(0)\rg>\ .
\end{array}
\ee
\hfill $\square$

\noindent{\bf Proof of Lemma~\ref{lemmaint2}.}
We have 
\be
\label{rep-19}
\begin{array}{l}
 \ds\int_{S^3}\,\lf<g_0\,d\ov{x}\wedge dx\, g_0^{-1},F_A(0)(J\p_r,K\p_r)\ \om^-_{\mathbf i}(x)\rg>\ dvol_{S^3}\\[5mm]
 \ds=\int_{S^3}\,2\,\sqrt{2}\ \om_{\mathbf i}^-\cdot\om^-_{\mathbf i}(x)\, \lf<{\mathbf i}_{g_0},F_A(0)(J\p_r,K\p_r)\rg>\ dvol_{S^3}\\[5mm]
 \ds+\int_{S^3}\,2\,\sqrt{2}\ \om_{\mathbf j}^-\cdot\om^-_{\mathbf i}(x)\, \lf<{\mathbf j}_{g_0},F_A(0)(J\p_r,K\p_r)\rg>\ dvol_{S^3}\\[5mm]
\ds+ \int_{S^3}\,2\,\sqrt{2}\ \om_{\mathbf k}^-\cdot\om^-_{\mathbf i}(x)\, \lf<{\mathbf k}_{g_0},F_A(0)(J\p_r,K\p_r)\rg>\ dvol_{S^3}\\[5mm]
 \end{array}
\ee
We recall from (\ref{II.31-a})
\be
\label{rep-20}
\lf\{
\begin{array}{l}
\ds  \om_{\mathbf i}^-\cdot\om^-_{\mathbf i}(x)= x_1^2+x_2^2-x_3^2-x_4^2  \\[5mm]
 \ds  \om_{\mathbf j}^-\cdot\om^-_{\mathbf i}(x)= 2\ (x_2\,x_3-x_1\,x_4)  \\[5mm]
\ds  \om_{\mathbf k}^-\cdot\om^-_{\mathbf i}(x)= 2\, (x_1\,x_3+x_2\,x_4)
\end{array}
\rg.
\ee
Recall that on $S^3$ we have
\be
\label{rep-21}
F_A(0)(J\p_r,K\p_r)=F_A(0)(x_1\, \p_{x_3}-x_3\,\p_{x_1}+x_4\,\p_{x_2}-x_2\,\p_{x_4}, x_1\, \p_{x_4}-x_4\,\p_{x_1}+x_2\,\p_{x_3}-x_3\,\p_{x_2} )
\ee
Hence we have successively using corollary~\ref{CorPol}
\be
\label{rep-22}
\begin{array}{l}
\ds \int_{S^3}\,2\,\sqrt{2}\ \om_{\mathbf i}^-\cdot\om^-_{\mathbf i}(x)\, \lf<{\mathbf i}_{g_0},F_A(0)(J\p_r,K\p_r)\rg>\ dvol_{S^3}\\[5mm]
 \ds=\,2\,\sqrt{2}\ \int_{S^3}(x_1^2+x_2^2-x_3^2-x_4^2)\ (x_1^2+x_2^2)\,\ dvol_{S^3}\ < {\mathbf i}_{g_0},F_A^{34}(0)>\\[5mm]
\ds+\,2\,\sqrt{2}\ \int_{S^3}(x_1^2+x_2^2-x_3^2-x_4^2)\ (x_3^2+x_4^2)\,\ dvol_{S^3}\ < {\mathbf i}_{g_0},F_A^{12}(0)>\\[5mm]
\ds=\,2\,\sqrt{2}\ \int_{S^3}(x_1^2+x_2^2-x_3^2-x_4^2)\ (x_1^2+x_2^2)\,\ dvol_{S^3}\ < {\mathbf i}_{g_0},-F_A^{12}(0)+F^{34}_A(0)>\\[5mm]
\ds=\,-\,\frac{2\,\sqrt{2}}{3}\ \pi^2\ < {\mathbf i}_{g_0},F_A^{12}(0)+F^{43}_A(0)>
\end{array}
\ee
we have
\be
\label{rep-23}
\begin{array}{l}
\ds \int_{S^3}\,2\,\sqrt{2}\ \om_{\mathbf j}^-\cdot\om^-_{\mathbf i}(x)\, \lf<{\mathbf j}_{g_0},F_A(0)(J\p_r,K\p_r)\rg>\ dvol_{S^3}\\[5mm]
\ds=\,4\,\sqrt{2}\ \int_{S^3} x_2^2\,x_3^2\  \ dvol_{S^3}\ < {\mathbf j}_{g_0},F_A^{42}(0)+F_A^{31}(0)>\\[5mm]
\ds+4\,\sqrt{2}\ \int_{S^3} x_1^2\,x_4^2\  \ dvol_{S^3}\ < {\mathbf j}_{g_0},F_A^{42}(0)+F_A^{31}(0)>\\[5mm]
\ds=-\,8\,\sqrt{2}\ \int_{S^3} x_2^2\,x_3^2\  \ dvol_{S^3}\ < {\mathbf j}_{g_0},F_A^{24}(0)+F_A^{13}(0)>\\[5mm]
\ds=-\,\frac{2\,\sqrt{2}}{3}\ \pi^2\ \ < {\mathbf j}_{g_0},F_A^{13}(0)+F_A^{24}(0)>
\end{array}
\ee
and
\be
\label{rep-24}
\begin{array}{l}
\ds \int_{S^3}\,2\,\sqrt{2}\ \om_{\mathbf k}^-\cdot\om^-_{\mathbf i}(x)\, \lf<{\mathbf k}_{g_0},F_A(0)(J\p_r,K\p_r)\rg>\ dvol_{S^3}\\[5mm]
\ds=4\,\sqrt{2}\ \int_{S^3} x_1^2\,x_3^2\  \ dvol_{S^3}\ < {\mathbf k}_{g_0},F_A^{41}(0)+F_A^{23}(0)>\\[5mm]
\ds+4\,\sqrt{2}\ \int_{S^3} x_2^2\,x_4^2\  \ dvol_{S^3}\ < {\mathbf k}_{g_0},F_A^{41}(0)+F_A^{23}(0)>\\[5mm]
\ds=-\,8\,\sqrt{2}\ \int_{S^3} x_1^2\,x_3^2\  \ dvol_{S^3}\ < {\mathbf k}_{g_0},F_A^{14}(0)+F_A^{32}(0)>\\[5mm]
\ds=\,-\frac{2\,\sqrt{2}}{3}\ \pi^2\ \ < {\mathbf k}_{g_0},F_A^{14}(0)+F_A^{32}(0)>
\end{array}
\ee
We finally get
\be
\label{rep-6-e}
\begin{array}{l}
 \ds\int_{S^3}\,\lf<g_0\,d\ov{x}\wedge dx\, g_0^{-1},F_A(0)(J\p_r,K\p_r)\ \om^-_{\mathbf i}(x)\rg>\ dvol_{S^3}
\ds= -\frac{\sqrt{2}}{3}\, \pi^2\,\lf<g_0\,d\ov{x}\wedge dx\, g_0^{-1}, P_-F_A(0)\rg>\ .
\end{array}
\ee
We have 
\be
\label{rep-25}
\begin{array}{l}
 \ds\int_{S^3}\,\lf<g_0\,d\ov{x}\wedge dx\, g_0^{-1},F_A(0)(K\p_r,I\p_r)\ \om^-_{\mathbf j}(x)\rg>\ dvol_{S^3}\\[5mm]
 \ds=\int_{S^3}\,2\,\sqrt{2}\ \om_{\mathbf i}^-\cdot\om^-_{\mathbf j}(x)\, \lf<{\mathbf i}_{g_0},F_A(0)(K\p_r,I\p_r)\rg>\ dvol_{S^3}\\[5mm]
 \ds+\int_{S^3}\,2\,\sqrt{2}\ \om_{\mathbf j}^-\cdot\om^-_{\mathbf j}(x)\, \lf<{\mathbf j}_{g_0},F_A(0)(K\p_r,I\p_r)\rg>\ dvol_{S^3}\\[5mm]
\ds+ \int_{S^3}\,2\,\sqrt{2}\ \om_{\mathbf k}^-\cdot\om^-_{\mathbf j}(x)\, \lf<{\mathbf k}_{g_0},F_A(0)(K\p_r,I\p_r)\rg>\ dvol_{S^3}\\[5mm]
 \end{array}
\ee
We recall from (\ref{II.31-a})
\be
\label{rep-26}
\lf\{
\begin{array}{l}
\ds  \om_{\mathbf i}^-\cdot\om^-_{\mathbf j}(x)= 2\,(x_1\,x_4+x_2\,x_3)  \\[5mm]
 \ds  \om_{\mathbf j}^-\cdot\om^-_{\mathbf j}(x)=  x_1^2+x_3^2-x_2^2-x_4^2 \\[5mm]
\ds  \om_{\mathbf k}^-\cdot\om^-_{\mathbf j}(x)= 2\, (x_3\,x_4-x_1\,x_2)
\end{array}
\rg.
\ee
Recall that on $S^3$ we have
\be
\label{rep-27}
F_A(0)(K\p_r,I\p_r)=F_A(0)(x_1\, \p_{x_4}-x_4\,\p_{x_1}+x_2\,\p_{x_3}-x_3\,\p_{x_2}, x_1\, \p_{x_2}-x_2\,\p_{x_1}+x_3\,\p_{x_4}-x_4\,\p_{x_3} )
\ee
Hence we have successively using corollary~\ref{CorPol}
\be
\label{rep-28}
\begin{array}{l}
\ds \int_{S^3}\,2\,\sqrt{2}\ \om_{\mathbf i}^-\cdot\om^-_{\mathbf j}(x)\, \lf<{\mathbf i}_{g_0},F_A(0)(K\p_r,I\p_r)\rg>\ dvol_{S^3}\\[5mm]
 \ds=-\,4\,\sqrt{2}\ \int_{S^3}\ x_1^2\,x_4^2\,  dvol_{S^3}\ < {\mathbf i}_{g_0},F_A^{12}(0)+F^{43}_A(0)>\\[5mm]
\ds-\,4\,\sqrt{2}\ \int_{S^3}\ x_2^2\,x_3^2\ dvol_{S^3}\ < {\mathbf i}_{g_0},F_A^{12}(0)+F_A^{43}(0)>\\[5mm]
 \ds=-\,8\,\sqrt{2}\ \int_{S^3}\ x_1^2\,x_4^2\,  dvol_{S^3}\ < {\mathbf i}_{g_0},F_A^{12}(0)+F^{43}_A(0)>\\[5mm]
 \ds=  -\,\frac{2\,\sqrt{2}}{3}\ \pi^2\ \ < {\mathbf i}_{g_0},F_A^{12}(0)+F^{43}_A(0)>
\end{array}
\ee
we have
{\be
\label{rep-29}
\begin{array}{l}
\ds \int_{S^3}\,2\,\sqrt{2}\ \om_{\mathbf j}^-\cdot\om^-_{\mathbf j}(x)\, \lf<{\mathbf j}_{g_0},F_A(0)(K\p_r,I\p_r)\rg>\ dvol_{S^3}\\[5mm]
\ds=\,2\,\sqrt{2}\ \int_{S^3} \ (x_1^2+x_3^2-x_2^2-x_4^2) (x_1^2+x_3^2) \ dvol_{S^3}\ < {\mathbf j}_{g_0},F_A^{42}(0)>\\[5mm]
\ds+2\,\sqrt{2}\ \int_{S^3}   \  (x_1^2+x_3^2-x_2^2-x_4^2) (x_2^2+x_4^2) \ dvol_{S^3}\ < {\mathbf j}_{g_0},F_A^{13}(0)>\\[5mm]
\ds=-\,2\,\sqrt{2}\ \int_{S^3}   (x_1^2+x_3^2-x_2^2-x_4^2) (x_1^2+x_3^2) \ dvol_{S^3}\ < {\mathbf j}_{g_0},F_A^{13}(0)-F_A^{42}(0)>\\[5mm]
\ds=\,-\,\frac{2\,\sqrt{2}}{3}\ \pi^2\ \ < {\mathbf j}_{g_0},F_A^{13}(0)+F_A^{24}(0)>
\end{array}
\ee}
and
\be
\label{rep-30}
\begin{array}{l}
\ds \int_{S^3}\,2\,\sqrt{2}\ \om_{\mathbf k}^-\cdot\om^-_{\mathbf j}(x)\, \lf<{\mathbf k}_{g_0},F_A(0)(K\p_r,I\p_r)\rg>\ dvol_{S^3}\\[5mm]
\ds=-4\,\sqrt{2}\ \int_{S^3} x_3^2\,x_4^2\  \ dvol_{S^3}\ < {\mathbf k}_{g_0},F_A^{41}(0)+F_A^{23}(0)>\\[5mm]
\ds-4\,\sqrt{2}\ \int_{S^3} x_1^2\,x_2^2\  \ dvol_{S^3}\ < {\mathbf k}_{g_0},F_A^{14}(0)+F_A^{32}(0)>\\[5mm]
\ds=-\,\frac{2\,\sqrt{2}}{3}\ \pi^2\ \  < {\mathbf k}_{g_0},F_A^{14}(0)+F_A^{32}(0)>\ .
\end{array}
\ee
We finally get
\be
\label{rep-6-f}
\begin{array}{l}
 \ds\int_{S^3}\,\lf<g_0\,d\ov{x}\wedge dx\, g_0^{-1},F_A(0)(K\p_r,I\p_r)\ \om^-_{\mathbf j}(x)\rg>\ dvol_{S^3}
\ds= -\frac{\sqrt{2}}{3}\, \pi^2\,\lf<g_0\,d\ov{x}\wedge dx\, g_0^{-1}, P_-F_A(0)\rg>\ .
\end{array}
\ee
We have 
\be
\label{rep-31}
\begin{array}{l}
 \ds\int_{S^3}\,\lf<g_0\,d\ov{x}\wedge dx\, g_0^{-1},F_A(0)(I\p_r,J\p_r)\ \om^-_{\mathbf k}(x)\rg>\ dvol_{S^3}\\[5mm]
 \ds=\int_{S^3}\,2\,\sqrt{2}\ \om_{\mathbf i}^-\cdot\om^-_{\mathbf k}(x)\, \lf<{\mathbf i}_{g_0},F_A(0)(I\p_r,J\p_r)\rg>\ dvol_{S^3}\\[5mm]
 \ds+\int_{S^3}\,2\,\sqrt{2}\ \om_{\mathbf j}^-\cdot\om^-_{\mathbf k}(x)\, \lf<{\mathbf j}_{g_0},F_A(0)(I\p_r,J\p_r)\rg>\ dvol_{S^3}\\[5mm]
\ds+ \int_{S^3}\,2\,\sqrt{2}\ \om_{\mathbf k}^-\cdot\om^-_{\mathbf k}(x)\, \lf<{\mathbf k}_{g_0},F_A(0)(I\p_r,J\p_r)\rg>\ dvol_{S^3}\\[5mm]
 \end{array}
\ee
We recall from (\ref{II.31-a})
\be
\label{rep-32}
\lf\{
\begin{array}{l}
\ds  \om_{\mathbf i}^-\cdot\om^-_{\mathbf k}(x)= 2\,(x_2\,x_4-x_1\,x_3)  \\[5mm]
 \ds  \om_{\mathbf j}^-\cdot\om^-_{\mathbf k}(x)=  2\ (x_1\,x_2+x_3\,x_4) \\[5mm]
\ds  \om_{\mathbf k}^-\cdot\om^-_{\mathbf k}(x)= x_1^2+x_4^2- x_2^2-x_3^2
\end{array}
\rg.
\ee
Recall that on $S^3$ we have
\be
\label{rep-33}
F_A(0)(I\p_r,J\p_r)=F_A(0)( x_1\, \p_{x_2}-x_2\,\p_{x_1}+x_3\,\p_{x_4}-x_4\,\p_{x_3}),x_1\, \p_{x_3}-x_3\,\p_{x_1}+x_4\,\p_{x_2}-x_2\,\p_{x_4})\\ .
\ee
Hence we have successively
\be
\label{rep-34}
\begin{array}{l}
\ds \int_{S^3}\,2\,\sqrt{2}\ \om_{\mathbf i}^-\cdot\om^-_{\mathbf k}(x)\, \lf\langle{\mathbf i}_{g_0},F_A(0)(I\p_r,J\p_r)\rg>\ dvol_{S^3}\\[5mm]
 \ds=-\,4\,\sqrt{2}\ \int_{S^3}\ x_2^2\,x_4^2\,  dvol_{S^3}\ \langle {\mathbf i}_{g_0},F_A^{12}(0)+F^{43}_A(0)\rangle\\[5mm]
\ds-\,4\,\sqrt{2}\ \int_{S^3}\ x_1^2\,x_3^2\ dvol_{S^3}\ \langle {\mathbf i}_{g_0},F_A^{12}(0)+F_A^{43}(0)\rangle\\[5mm]
 \ds=-\,8\,\sqrt{2}\ \int_{S^3}\ x_2^2\,x_4^2\,  dvol_{S^3}\ \langle {\mathbf i}_{g_0},F_A^{12}(0)+F^{43}_A(0)\rangle\\[5mm]
 \ds= - \,\frac{2\,\sqrt{2}}{3}\ \pi^2\  \langle {\mathbf i}_{g_0},F_A^{12}(0)+F^{43}_A(0)\rangle
\end{array}
\ee
we have
{\be
\label{rep-35}
\begin{array}{l}
\ds \int_{S^3}\,2\,\sqrt{2}\ \om_{\mathbf j}^-\cdot\om^-_{\mathbf k}(x)\, \lf<{\mathbf j}_{g_0},F_A(0)(I\p_r,J\p_r)\rg>\ dvol_{S^3}\\[5mm]
\ds=-\,4\,\sqrt{2}\ \int_{S^3} \ x_1^2\,x_2^2  \ dvol_{S^3}\ \lf< {\mathbf j}_{g_0},F_A^{13}(0)+F_A^{24}(0)\rg>\\[5mm]
\ds-4\,\sqrt{2}\ \int_{S^3}   \  x_3^2\,x_4^2  \ dvol_{S^3}\ \lf< {\mathbf j}_{g_0},F_A^{13}(0)+F_A^{24}(0)\rg>\\[5mm]
\ds=-\,8\,\sqrt{2}\ \int_{S^3}    \ dvol_{S^3}\ \lf< {\mathbf j}_{g_0},F_A^{13}(0)+F_A^{24}(0)\rg>\\[5mm]
\ds=-\,\frac{2\,\sqrt{2}}{3}\ \pi^2 \ \lf< {\mathbf j}_{g_0},F_A^{13}(0)+F_A^{24}(0)\rg>
\end{array}
\ee}
and
{\be
\label{rep-36}
\begin{array}{l}
\ds \int_{S^3}\,2\,\sqrt{2}\ \om_{\mathbf k}^-\cdot\om^-_{\mathbf k}(x)\, \lf<{\mathbf k}_{g_0},F_A(0)(I\p_r,J\p_r)\rg>\ dvol_{S^3}\\[5mm]
\ds=\,2\,\sqrt{2}\ \int_{S^3} \ (x_1^2+x_4^2-x_2^2-x_3^2) (x_1^2+x_4^2) \ dvol_{S^3}\ \lf< {\mathbf k}_{g_0},F_A^{23}(0)\rg>\\[5mm]
\ds+2\,\sqrt{2}\ \int_{S^3}   \  (x_1^2+x_4^2-x_2^2-x_3^2) (x_2^2+x_3^2) \ dvol_{S^3}\ \lf< {\mathbf k}_{g_0},F_A^{14}(0)\rg>\\[5mm]
\ds=\,2\,\sqrt{2}\ \int_{S^3}   (x_1^2+x_4^2-x_2^2-x_3^2) (x_1^2+x_4^2) \ dvol_{S^3}\ \lf< {\mathbf k}_{g_0},-F_A^{14}(0)+F_A^{23}(0)\rg>\\[5mm]
\ds=-\,\frac{2\,\sqrt{2}}{3}\ \pi^2\  \lf< {\mathbf k}_{g_0},F_A^{14}(0)+F_A^{32}(0)\rg>
\end{array}
\ee
We finally get
\be
\label{rep-6-g}
\begin{array}{l}
 \ds\int_{S^3}\,\lf<g_0\,d\ov{x}\wedge dx\, g_0^{-1},F_A(0)(I\p_r,J\p_r)\ \om^-_{\mathbf k}(x)\rg>\ dvol_{S^3}
\ds= -\frac{\sqrt{2}}{3}\, \pi^2\,\lf<g_0\,d\ov{x}\wedge dx\, g_0^{-1}, P_-F_A(0)\rg>\ .
\end{array}
\ee
\hfill $\square$

\end{document}